%% file: KPP_systems_2_final_version.tex
\documentclass[oneside,american]{amsart}
\pdfoutput=1
\usepackage[T1]{fontenc}
\usepackage[latin9]{inputenc}
\usepackage{babel}
\usepackage{verbatim}
\usepackage{refstyle}
\usepackage{mathrsfs}
\usepackage{units}
\usepackage{enumitem}
\usepackage{amsbsy}
\usepackage{amstext}
\usepackage{amsthm}
\usepackage{amssymb}
\usepackage{graphicx}
\usepackage[unicode=true,pdfusetitle,
 bookmarks=true,bookmarksnumbered=false,bookmarksopen=false,
 breaklinks=false,pdfborder={0 0 1},backref=false,colorlinks=false]
 {hyperref}
\usepackage{breakurl}

\makeatletter

\AtBeginDocument{\providecommand\subsecref[1]{\ref{subsec:#1}}}
\AtBeginDocument{\providecommand\thmref[1]{\ref{thm:#1}}}
\AtBeginDocument{\providecommand\figref[1]{\ref{fig:#1}}}
\AtBeginDocument{\providecommand\conjref[1]{\ref{conj:#1}}}
\AtBeginDocument{\providecommand\lemref[1]{\ref{lem:#1}}}
\AtBeginDocument{\providecommand\propref[1]{\ref{prop:#1}}}
\providecommand{\tabularnewline}{\\}
\RS@ifundefined{subsecref}
  {\newref{subsec}{name = \RSsectxt}}
  {}
\RS@ifundefined{thmref}
  {\def\RSthmtxt{theorem~}\newref{thm}{name = \RSthmtxt}}
  {}
\RS@ifundefined{lemref}
  {\def\RSlemtxt{lemma~}\newref{lem}{name = \RSlemtxt}}
  {}

\numberwithin{equation}{section}
\numberwithin{figure}{section}
 \theoremstyle{definition}
 \newtheorem*{defn*}{\protect\definitionname}
\theoremstyle{plain}
\newtheorem{thm}{\protect\theoremname}[section]
  \theoremstyle{plain}
  \newtheorem{conjecture}[thm]{\protect\conjecturename}
  \theoremstyle{plain}
  \newtheorem{lem}[thm]{\protect\lemmaname}
  \theoremstyle{plain}
  \newtheorem{prop}[thm]{\protect\propositionname}
  \theoremstyle{remark}
  \newtheorem*{rem*}{\protect\remarkname}

\usepackage{pgfplots}

\newref{prop}{name=Proposition~}
\newref{thm}{name=Theorem~}
\newref{lem}{name=Lemma~}
\newref{cor}{name=Corollary~}
\newref{conj}{name=Conjecture~}
\newref{fig}{name=Figure~}
\newref{subsec}{name=Subsection~}

\@ifundefined{showcaptionsetup}{}{%
 \PassOptionsToPackage{caption=false}{subfig}}
\usepackage{subfig}
\makeatother

  \providecommand{\conjecturename}{Conjecture}
  \providecommand{\definitionname}{Definition}
  \providecommand{\lemmaname}{Lemma}
  \providecommand{\propositionname}{Proposition}
  \providecommand{\remarkname}{Remark}
\providecommand{\theoremname}{Theorem}

\begin{document}

\title[Non-cooperative KPP systems]{Non-cooperative Fisher\textendash KPP systems: asymptotic behavior
of traveling waves}

\author{Léo Girardin}

\thanks{This work has been carried out in the framework of the NONLOCAL project
(ANR-14-CE25-0013) funded by the French National Research Agency (ANR).\\
Laboratoire Jacques-Louis Lions, CNRS UMR 7598, Université Pierre
et Marie Curie, 4 place Jussieu, 75005 Paris, France}

\email{girardin@ljll.math.upmc.fr}
\begin{abstract}
This paper is concerned with non-cooperative parabolic reaction\textendash diffusion
systems which share structural similarities with the scalar Fisher\textendash KPP
equation. In a previous paper, we established that these systems admit
traveling wave solutions whose profiles connect the null state to
a compact subset of the positive cone. The main object of the present
paper is the investigation of a more precise description of these
profiles. Non-cooperative KPP systems can model various phenomena
where the following three mechanisms occur: local diffusion in space,
linear cooperation and superlinear competition.
\end{abstract}


\maketitle

\section{Introduction}

This paper is a sequel to a previous paper by the same author \cite{Girardin_2016_2}
where the so-called \textit{KPP systems} were investigated. The prototypical
and, arguably, most famous KPP system is the \textit{Lotka\textendash Volterra
mutation\textendash competition\textendash diffusion system}:
\[
\frac{\partial\mathbf{u}}{\partial t}-\text{diag}\left(\mathbf{d}\right)\Delta_{x}\mathbf{u}=\text{diag}\left(\mathbf{r}\right)\mathbf{u}+\mathbf{M}\mathbf{u}-\text{diag}\left(\mathbf{u}\right)\mathbf{C}\mathbf{u},
\]
where $\mathbf{u}$ is a nonnegative vector containing phenotypical
densities, $\mathbf{d}$ and $\mathbf{r}$ are positive vectors containing
respectively diffusion rates and growth rates, $\mathbf{M}$ is an
essentially nonnegative irreducible matrix with null Perron\textendash Frobenius
eigenvalue containing mutation rates (typically a discrete Neumann
Laplacian) and $\mathbf{C}$ is a positive matrix containing competition
rates. Although the Lotka\textendash Volterra competition\textendash diffusion
system (without mutations) is a very classical research subject, mutations
can dramatically influence some of its properties and their overall
effect is still poorly understood. 

More generally, KPP systems as defined in \cite{Girardin_2016_2}
are non-cooperative (or non-monotone, i.e. they do not satisfy a comparison
principle; see Protter\textendash Weinberger \cite[Chapter 3, Section 8]{Protter_Weinberger})
and have started to attract attention relatively recently. Their study
requires innovative ideas and the literature is limited; a detailed
bibliography can be found in \cite{Girardin_2016_2}.

By adapting proofs and methods well-known in the context of the scalar
KPP equation,
\[
\frac{\partial u}{\partial t}-d\Delta_{x}u=ru-cu^{2},
\]
 first studied by Fisher \cite{Fisher_1937} and Kolmogorov, Petrovsky
and Piskunov \cite{KPP_1937}, various properties of these systems
were established in \cite{Girardin_2016_2}. In particular, a KPP
system equipped with a reaction term sufficiently analogous to $u-u^{2}$
admits traveling wave solutions with a half-line of possible speeds
and a positive minimal speed $c^{\star}$. These traveling waves are
defined in a very general way: it is merely required that they describe
the invasion of $\mathbf{0}$ by a positive population density. A
very natural subsequent question is that of the evolution of the distribution
$\mathbf{u}$ during the invasion. Which components lead the invasion?
Which components settle once the invasion is over?

Having in mind that the waves traveling at speed $c^{\star}$ should
attract front-like and compactly supported initial data (although
this statement has yet to be proven, since \cite{Girardin_2016_2}
only established the equality between $c^{\star}$ and the spreading
speed associated with such initial data, and it is expected to be
a very difficult problem), a more general question is then: given
a class of initial data, what is the long-time distribution of the
solution? 

In the rest of the introduction, we present more precisely the problem
and state our main results. Sections 2, 3 and 4 are dedicated to the
proofs of these results. Finally, open questions, interesting remarks
and numerical experiments are discussed in Section 5. 

\subsection{The non-cooperative KPP system}

From now on, an integer $N\geq2$ is fixed. 

A positive vector $\mathbf{d}\in\mathsf{K}^{++}$, a square matrix
$\mathbf{L}\in\mathsf{M}$ and a vector field $\mathbf{c}\in\mathscr{C}^{1}\left(\mathbb{R}^{N},\mathbb{R}^{N}\right)$
are fixed. We denote for the sake of brevity $\mathbf{D}=\text{diag}\left(\mathbf{d}\right)$.
\begin{table}[b]
{\footnotesize{}}%
\begin{tabular}{|c|c|}
\hline 
{\footnotesize{}Notation} & {\footnotesize{}Definition}\tabularnewline
\hline 
\hline 
{\footnotesize{}$\left[n\right]$} & {\footnotesize{}$\left[1,n\right]\cap\mathbb{N}$}\tabularnewline
\hline 
{\footnotesize{} $\left(\mathbf{e}_{n,i}\right)_{i\in\left[n\right]}$} & {\footnotesize{}canonical basis of $\mathbb{R}^{n}$}\tabularnewline
\hline 
{\footnotesize{}$\left|\bullet\right|_{n}$} & {\footnotesize{}Euclidean norm of $\mathbb{R}^{n}$}\tabularnewline
\hline 
{\footnotesize{} $\mathsf{B}_{n}\left(\mathbf{v},r\right)$, $\mathsf{S}_{n}\left(\mathbf{v},r\right)$ } & {\footnotesize{}open ball and sphere of center $\mathbf{v}\in\mathbb{R}^{n}$
and radius $r>0$}\tabularnewline
\hline 
{\footnotesize{}$\geq_{n}$, $>_{n}$, $\gg_{n}$} & {\footnotesize{}$v_{i}\geq\hat{v}_{i}$ for all $i\in\left[n\right]$,
$\mathbf{v}\geq_{n}\hat{\mathbf{v}}$ and $\mathbf{v}\neq\hat{\mathbf{v}}$,
$v_{i}>\hat{v}_{i}$ for all $i\in\left[n\right]$}\tabularnewline
\hline 
{\footnotesize{}nonnegative, nonneg. nonzero, positive $\mathbf{v}\in\mathbb{R}^{n}$} & {\footnotesize{}$\mathbf{v}\geq_{n}\mathbf{0}$, $\mathbf{v}>_{n}\mathbf{0}$,
$\mathbf{v}\gg_{n}\mathbf{0}$}\tabularnewline
\hline 
{\footnotesize{}$\mathsf{K}_{n}$, $\mathsf{K}_{n}^{+}$, $\mathsf{K}_{n}^{++}$} & {\footnotesize{}sets of all nonnegative, nonneg. nonzero, positive
vectors}\tabularnewline
\hline 
{\footnotesize{}$\mathsf{S}_{n}^{+}\left(\mathbf{0},1\right)$, $\mathsf{S}_{n}^{++}\left(\mathbf{0},1\right)$} & {\footnotesize{}$\mathsf{K}_{n}^{+}\cap\mathsf{S}_{n}\left(\mathbf{0},1\right)$,
$\mathsf{K}_{n}^{++}\cap\mathsf{S}_{n}\left(\mathbf{0},1\right)$ }\tabularnewline
\hline 
{\footnotesize{}$\mathsf{M}_{n,n'}$, $\mathsf{M}_{n}$} & {\footnotesize{}sets of all real matrices of dimension $n\times n'$,
$n\times n$ }\tabularnewline
\hline 
{\footnotesize{}$\mathbf{I}_{n}$, $\mathbf{1}_{n,n'}$} & {\footnotesize{}identity matrix, matrix whose every entry is equal
to $1$ }\tabularnewline
\hline 
{\footnotesize{}$\text{diag}\left(\mathbf{v}\right)$} & {\footnotesize{}diagonal matrix whose $i$-th diagonal entry is $v_{i}$}\tabularnewline
\hline 
{\footnotesize{}essentially nonnegative matrix} & {\footnotesize{}matrix $\mathbf{A}$ such that $\mathbf{A}-\min\limits _{i\in\left[n\right]}\left(a_{i,i}\right)\mathbf{I}_{n}$
is nonnegative}\tabularnewline
\hline 
{\footnotesize{}$\mathbf{A}\circ\mathbf{B}$} & {\footnotesize{}Hadamard (entry-by-entry) product $\left(a_{i,j}b_{i,j}\right)_{\left(i,j\right)\in\left[n\right]\times\left[n'\right]}$}\tabularnewline
\hline 
{\footnotesize{}$\mathbf{f}\left[\hat{\mathbf{f}}\right]$} & {\footnotesize{}composition of the functions $\mathbf{f}$ and $\hat{\mathbf{f}}$}\tabularnewline
\hline 
\end{tabular}{\footnotesize \par}

\caption{General notations ( the subscripts depending only on $1$ or $N$
are omitted when the context is unambiguous)}
\end{table}

We consider the following semilinear parabolic system:
\[
\partial_{t}\mathbf{u}-\mathbf{D}\partial_{xx}\mathbf{u}=\mathbf{L}\mathbf{u}-\mathbf{c}\left[\mathbf{u}\right]\circ\mathbf{u},\quad\left(E_{KPP}\right)
\]
with $\mathbf{u}:\left(t,x\right)\in\mathbb{R}^{2}\mapsto\mathbf{u}\left(t,x\right)\in\mathbb{R}^{N}$
as unknown. In order to ease the notations, we only consider one-dimensional
spaces, however all forthcoming results could be applied directly
to traveling plane waves in multidimensional spaces (these solutions
being in fact one-dimensional).

When restricted to solutions $\mathbf{u}:\mathbb{R}\to\mathbb{R}^{N}$
which are constant in space, $\left(E_{KPP}\right)$ reduces to

\[
\mathbf{u}'=\mathbf{L}\mathbf{u}-\mathbf{c}\left[\mathbf{u}\right]\circ\mathbf{u}.\quad\left(E_{KPP}^{0}\right)
\]

When restricted to solutions $\mathbf{u}:\mathbb{R}\to\mathbb{R}^{N}$
which are constant in time, $\left(E_{KPP}\right)$ reduces to
\[
-\mathbf{D}\mathbf{u}''=\mathbf{L}\mathbf{u}-\mathbf{c}\left[\mathbf{u}\right]\circ\mathbf{u}.\quad\left(S_{KPP}\right)
\]

When restricted to traveling solutions of the form $\mathbf{u}:\left(t,x\right)\mapsto\mathbf{p}\left(x-ct\right)$
with $c\in\mathbb{R}$, $\left(E_{KPP}\right)$ reduces to

\[
-\mathbf{D}\mathbf{p}''-c\mathbf{p}'=\mathbf{L}\mathbf{p}-\mathbf{c}\left[\mathbf{p}\right]\circ\mathbf{p}.\quad\left(TW\left[c\right]\right)
\]

\subsubsection{Basic KPP assumptions}

The basic assumptions introduced in \cite{Girardin_2016_2} are the
following ones.

{\renewcommand\labelenumi{(${H}_\theenumi$)}
\begin{enumerate}
\item $\mathbf{L}$ is essentially nonnegative and irreducible. 
\item $\mathbf{c}\left(\mathsf{K}\right)\subset\mathsf{K}$.
\item $\mathbf{c}\left(\mathbf{0}\right)=\mathbf{0}$.
\item There exists
\[
\left(\underline{\alpha},\delta,\underline{\mathbf{c}}\right)\in[1,+\infty)^{2}\times\mathsf{K}^{++}
\]
such that
\[
\sum_{j=1}^{N}l_{i,j}n_{j}\geq0\implies\alpha^{\delta}\underline{c}_{i}\leq c_{i}\left(\alpha\mathbf{n}\right)
\]
for all
\[
\left(\mathbf{n},\alpha,i\right)\in\mathsf{S}^{+}\left(\mathbf{0},1\right)\times[\underline{\alpha},+\infty)\times\left[N\right].
\]
\end{enumerate}
The assumption $\left(H_{4}\right)$ loosely means that $\mathbf{c}$
grows at least linearly at infinity. The precise condition means,
however, that in the set $\left\{ \mathbf{v}\in\mathsf{K}\ |\ \left(\mathbf{L}\mathbf{v}\right)_{i}<0\right\} $
(which is nonempty if and only if $l_{i,i}<0$ and contains in such
a case the open half-line $\text{span}\left(\mathbf{e}_{i}\right)\cap\mathsf{K}^{+}$),
the growth of $c_{i}$ is not important. Anyway, $\left(H_{4}\right)$
includes the Lotka\textendash Volterra form of competition (linear
and positive $\mathbf{c}$) as well as more general forms (see for
instance Gilpin\textendash Ayala \cite{Gilpin_Ayala}). 

Recall from the Perron\textendash Frobenius theorem that if $\mathbf{L}$
is nonnegative and irreducible, its spectral radius $\rho\left(\mathbf{L}\right)$
is also its dominant eigenvalue, called the \textit{Perron\textendash Frobenius
eigenvalue} $\lambda_{PF}\left(\mathbf{L}\right)$, and is the unique
eigenvalue associated with a positive eigenvector. Recall also that
if $\mathbf{L}$ is essentially nonnegative and irreducible, the Perron\textendash Frobenius
theorem can still be applied. In such a case, the unique eigenvalue
of $\mathbf{L}$ associated with a positive eigenvector is $\lambda_{PF}\left(\mathbf{L}\right)=\rho\left(\mathbf{L}-\min\limits _{i\in\left[N\right]}\left(l_{i,i}\right)\mathbf{I}_{N}\right)+\min\limits _{i\in\left[N\right]}\left(l_{i,i}\right)$.
Any eigenvector associated with $\lambda_{PF}\left(\mathbf{L}\right)$
is referred to as a \textit{Perron\textendash Frobenius eigenvector}
and the unit one is denoted $\mathbf{n}_{PF}\left(\mathbf{L}\right)$.

In view of \cite[Theorems 1.3, 1.4, 1.5]{Girardin_2016_2}, in order
to study traveling waves and non-trivial long-time behavior, the following
assumption is also necessary. 
\begin{enumerate}[resume]
\item $\lambda_{PF}\left(\mathbf{L}\right)>0$.
\end{enumerate}
The collection $\left(H_{1}\right)$\textendash $\left(H_{5}\right)$
is always assumed from now on. Notice that, although this does not
bring any new result, the scalar KPP equation could be seen as a particular
KPP system (understanding the pair $\left(H_{1}\right)$ and $\left(H_{5}\right)$
as $r>0$). Biological interpretations of these assumptions can be
found in \cite[Section 1.5]{Girardin_2016_2}.

\subsubsection{Traveling waves}

Traveling waves are defined in \cite{Girardin_2016_2} as follows. 
\begin{defn*}
A \textit{traveling wave solution} of $\left(E_{KPP}\right)$ is a
\textit{profile\textendash speed} pair
\[
\left(\mathbf{p},c\right)\in\mathscr{C}^{2}\left(\mathbb{R},\mathbb{R}^{N}\right)\times[0,+\infty)
\]
 which satisfies:
\begin{itemize}
\item $\mathbf{u}:\left(t,x\right)\mapsto\mathbf{p}\left(x-ct\right)$ is
a bounded positive classical solution of $\left(E_{KPP}\right)$;
\item $\left(\liminf\limits _{\xi\to-\infty}p_{i}\left(\xi\right)\right)_{i\in\left[N\right]}>\mathbf{0}$;
\item $\lim\limits _{\xi\to+\infty}\mathbf{p}\left(\xi\right)=\mathbf{0}$.
\end{itemize}
\end{defn*}
By construction, a traveling wave solution $\left(\mathbf{p},c\right)$
solves $\left(TW\left[c\right]\right)$.

The set of all profiles associated with some speed $c$ is denoted
$\mathscr{P}_{c}$. By \cite[Theorems 1.5, 1.7]{Girardin_2016_2},
$\mathscr{P}_{c}$ is empty if
\[
c<c^{\star}=\min_{\mu>0}\frac{\lambda_{PF}\left(\mu^{2}\mathbf{D}+\mathbf{L}\right)}{\mu}.
\]
The converse statement (existence of a profile if $c\geq c^{\star}$)
is likely false in general but is true provided $\mathbf{c}$ is monotonic
in the following sense:
\[
D\mathbf{c}\left(\mathbf{v}\right)\geq\mathbf{0}\text{ for all }\mathbf{v}\in\mathsf{K}.
\]

\subsection{Results: at the edge of the fronts}

The distribution of the profiles near $+\infty$ follows the \textquotedblleft rule
of thumb\textquotedblright{} unfolded in \cite{Girardin_2016_2}:
for several standard problems, KPP systems can be addressed exactly
as KPP equations and the results are analogous.

Recall from \cite[Lemma 6.2]{Girardin_2016_2} the notation $\mathbf{n}_{\mu}=\mathbf{n}_{PF}\left(\mu^{2}\mathbf{D}+\mathbf{L}\right)$
for all $\mu\in\mathbb{R}$. Recall also that the equation 
\[
\frac{\lambda_{PF}\left(\mu^{2}\mathbf{D}+\mathbf{L}\right)}{\mu}=c
\]
 admits no real solution if $c<c^{\star}$, exactly one real solution
$\mu_{c^{\star}}>0$ if $c=c^{\star}$ and exactly two real solutions
$\mu_{2,c}>\mu_{1,c}>0$ if $c>c^{\star}$. Define subsequently for
all $c\geq c^{\star}$ the quantity
\[
\mu_{c}=\min\left\{ \mu>0\ |\ \frac{\lambda_{PF}\left(\mu^{2}\mathbf{D}+\mathbf{L}\right)}{\mu}=c\right\} =\left\{ \begin{matrix}\mu_{c^{\star}} & \text{ if }c=c^{\star},\\
\mu_{1,c} & \text{ if }c>c^{\star}.
\end{matrix}\right.
\]
\begin{thm}
\label{thm:Distribution_edge} Let 
\[
k_{c}=\left\{ \begin{matrix}0 & \text{if }c>c^{\star},\\
1 & \text{if }c=c^{\star}.
\end{matrix}\right.
\]

For all traveling wave solutions $\left(\mathbf{p},c\right)$, there
exists $A>0$ such that, as $\xi\to+\infty$,
\[
\left\{ \begin{matrix}\mathbf{p}\left(\xi\right)\sim A\xi^{k_{c}}\text{e}^{-\mu_{c}\xi}\mathbf{n}_{\mu_{c}},\\
\mathbf{p}'\left(\xi\right)\sim-\mu_{c}\mathbf{p}\left(\xi\right),\\
\mathbf{p}''\left(\xi\right)\sim\mu_{c}^{2}\mathbf{p}\left(\xi\right).
\end{matrix}\right.
\]

In particular, if $\mathbf{d}=\mathbf{1}_{N,1}$, 
\[
\mathbf{p}\left(\xi\right)\sim A\xi^{k_{c}}\text{e}^{-\frac{1}{2}\left(c-\sqrt{c^{2}-4\lambda_{PF}\left(\mathbf{L}\right)}\right)\xi}\mathbf{n}_{PF}\left(\mathbf{L}\right).
\]
\end{thm}

This result is proved in Section 2.

Recall that up to a well-known change of variable $x$, we can always
assume without loss of generality $\max\limits _{i\in\left[N\right]}d_{i}=1$. 

If we have in mind the mutation\textendash competition\textendash diffusion
system, then the ecological interpretation of this result is the following:
at the leading edge of the invasion, the normalized distribution in
phenotypes is $\mathbf{n}_{\mu_{c}}$ and the total population is
proportional to $\left(x-ct\right)^{k_{c}}\text{e}^{-\mu_{c}\left(x-ct\right)}$. 

In the special case $c=c^{\star}$, this theorem answers positively
a conjecture of Morris, Börger and Crooks \cite[Section 4]{Morris_Borger_Crooks}.

Recall that, for the scalar KPP equation, the analogous result on
exponential decays has two common proofs, one using ODE arguments
and especially phase-plane analysis and the other one using elliptic
arguments and especially the comparison principle. Although we could
prove the above result by phase-plane analysis indeed, the proof we
will provide uses a third technique relying upon the monotonicity
of the profiles near $+\infty$, bilateral Laplace transforms and
a Ikehara theorem. In our opinion, this technique of proof has independent
interest: on one hand, it does not require the comparison principle
and, on the other hand, it might be generalizable to non-ODE settings
(space-periodic media and pulsating fronts, for instance). 

\subsection{Results: at the back of the fronts}

On the contrary, the distribution of the profiles near $-\infty$
is a much more intricate question, where the multidimensional and
non-cooperative structure of the KPP system become preponderant.

Given a positive classical solution $\mathbf{u}$ of $\left(S_{KPP}\right)$,
a \textit{traveling wave connecting $\mathbf{0}$ to $\mathbf{u}$}
is a traveling wave whose profile $\mathbf{p}$ converges to $\mathbf{u}$
as $\xi\to-\infty$. The general aim is to prove that all traveling
waves connect $\mathbf{0}$ to some positive classical solution of
$\left(S_{KPP}\right)$ and, when several solutions can be connected
to $\mathbf{0}$, to determine somehow which connection prevails.
However, as will be explained in \subsecref{Delicate_back} (and was
first pointed out in Barles\textendash Evans\textendash Souganidis
\cite{Barles_Evans_S}), a general and precise treatment of this problem
is likely impossible. It is necessary to focus on special cases. Looking
at the literature, we find two frameworks commonly assumed to be mathematically
tractable:
\begin{itemize}
\item competition terms $c_{i}\left(\mathbf{v}\right)$ with separated dependencies
on $i$ and on $\mathbf{v}$ (Coville\textendash Fabre \cite{Coville_Fabre_},
Dockery\textendash Hutson\textendash Mischaikow\textendash Pernarowski
\cite{Dockery_1998}, Griette\textendash Raoul \cite{Griette_Raoul},
Leman\textendash Méléard\textendash Mirrahimi \cite{Leman_Meleard_}), 
\item two-component systems with linear competition and vanishingly small
mutations (Dockery\textendash Hutson\textendash Mischaikow\textendash Pernarowski
\cite{Dockery_1998}, Griette\textendash Raoul \cite{Griette_Raoul},
Morris\textendash Börger\textendash Crooks \cite{Morris_Borger_Crooks}).
\end{itemize}

\subsubsection{Separated competition}
\begin{enumerate}[resume]
\item There exist $\mathbf{a}\in\mathsf{K}^{++}$ and $b:\mathbb{R}^{N}\to\mathbb{R}$
such that:
\begin{itemize}
\item $\mathbf{c}\left(\mathbf{v}\right)=b\left(\mathbf{v}\right)\mathbf{a}$
for all $\mathbf{v}\in\mathsf{K}$;
\item the function $w\mapsto b\left(w\mathbf{e}_{i}+\mathbf{v}\right)$
is increasing in $\left(0,+\infty\right)$ for all $\mathbf{v}\in\mathsf{K}$
and all $i\in\left[N\right]$.
\end{itemize}
\end{enumerate}
By monotonicity of $\mathbf{c}$, supplementing $\left(H_{1}\right)$\textendash $\left(H_{5}\right)$
with $\left(H_{6}\right)$ implies the existence of a profile $\mathbf{p}\in\mathscr{P}_{c}$
for all $c\geq c^{\star}$. The decomposition $\mathbf{c}=b\mathbf{a}$
is unique up to a multiplicative normalization and we will assume
for instance $\max\limits _{i\in\left[N\right]}a_{i}=1$. We denote
$\mathbf{A}=\text{diag}\left(\mathbf{a}\right)$ (so that $\mathbf{c}\left(\mathbf{v}\right)\circ\mathbf{v}=b\left(\mathbf{v}\right)\mathbf{A}\mathbf{v}$).

An especially interesting subcase is the intersection between $\left(H_{6}\right)$
and the Lotka\textendash Volterra competition form, where $b$ is
a linear functional, that is where there exists $\mathbf{b}\in\mathsf{K}^{++}$
such that 
\[
b\left(\mathbf{v}\right)=\mathbf{b}^{T}\mathbf{v}\text{ for all }\mathbf{v}\in\mathsf{K}.
\]
The system $\left(E_{KPP}\right)$ then reads 
\[
\partial_{t}\mathbf{u}-\mathbf{D}\partial_{xx}\mathbf{u}=\mathbf{L}\mathbf{u}-\left(\mathbf{b}^{T}\mathbf{u}\right)\mathbf{A}\mathbf{u}.
\]
The systems studied in Dockery\textendash Hutson\textendash Mischaikow\textendash Pernarowski
\cite{Dockery_1998} and in Griette\textendash Raoul \cite{Griette_Raoul}
correspond respectively to 
\[
\mathbf{a}=\mathbf{b}=\mathbf{1}_{N,1}
\]
 and to 
\[
\left(\mathbf{a},\mathbf{b}\right)=\left(\left(\frac{K}{r},1\right)^{T},\frac{r}{K}\mathbf{1}_{2,1}\right).
\]

The matrix $\mathbf{A}^{-1}\mathbf{L}$ being essentially nonnegative
and irreducible, the following eigenpair is well-defined:
\[
\left(\lambda_{\mathbf{a}},\mathbf{n}_{\mathbf{a}}\right)=\left(\lambda_{PF}\left(\mathbf{A}^{-1}\mathbf{L}\right),\mathbf{n}_{PF}\left(\mathbf{A}^{-1}\mathbf{L}\right)\right).
\]
Applying \cite[Theorem 1.4]{Girardin_2016_2} to the following two
pairs of parameters $\left(\mathbf{L},\mathbf{c}\right)$:
\[
\left(\mathbf{L},\mathbf{v}\mapsto\left(\mathbf{1}_{1,N}\mathbf{v}\right)\mathbf{a}\right),
\]
\[
\left(\mathbf{A}^{-1}\mathbf{L},\mathbf{v}\mapsto\left(\mathbf{1}_{1,N}\mathbf{v}\right)\mathbf{1}_{N,1}\right),
\]
it is easily deduced that $\lambda_{PF}\left(\mathbf{L}\right)>0$
if and only if $\lambda_{\mathbf{a}}>0$. By strict monotonicity of
$\alpha\mapsto b\left(\alpha\mathbf{n}_{\mathbf{a}}\right)$, we can
define $\alpha^{\star}>0$ as the unique solution of $b\left(\alpha\mathbf{n}_{\mathbf{a}}\right)=\lambda_{\mathbf{a}}$.
It follows easily that $\mathbf{v}^{\star}=\alpha^{\star}\mathbf{n}_{\mathbf{a}}$
is the unique positive constant solution of $\left(S_{KPP}\right)$.
In particular, if $b$ is a linear functional, then
\[
\mathbf{v}^{\star}=\frac{\lambda_{\mathbf{a}}}{\mathbf{b}^{T}\mathbf{n}_{\mathbf{a}}}\mathbf{n}_{\mathbf{a}}.
\]
\begin{thm}
\label{thm:Distribution_back_(H6)} Assume $\left(H_{6}\right)$,
$\mathbf{d}=\mathbf{1}_{N,1}$ and $\mathbf{a}=\mathbf{1}_{N,1}$. 

For all $c\in[c^{\star},+\infty)$, let $p_{c}\in\mathscr{C}^{2}\left(\mathbb{R}\right)$
such that $\left(p_{c},c\right)$ is the unique traveling wave solution
of the scalar equation
\[
\partial_{t}u-\partial_{xx}u=\lambda_{PF}\left(\mathbf{L}\right)u-b\left(u\mathbf{n}_{PF}\left(\mathbf{L}\right)\right)u
\]
connecting $0$ to $\alpha^{\star}$ and satisfying $p_{c}\left(0\right)=\frac{\alpha^{\star}}{2}$. 

Then all $\mathbf{p}\in\mathscr{P}_{c}$ have the form
\[
\mathbf{p}:\xi\mapsto p_{c}\left(\xi-\xi_{0}\right)\mathbf{n}_{PF}\left(\mathbf{L}\right)\text{ with }\xi_{0}\in\mathbb{R}.
\]
Consequently, $\mathbf{p}\in\mathscr{P}_{c}$ is unique up to translation
and connects $\mathbf{0}$ to $\mathbf{v}^{\star}$.
\end{thm}

This result is proved in Section 3.2.

This theorem establishes that the set of assumptions $\left(H_{6}\right)$,
$\mathbf{d}=\mathbf{1}_{N,1}$, $\mathbf{a}=\mathbf{1}_{N,1}$ is
so restrictive that the multidimensional problem can in fact be reduced
to the scalar one. This is really the strongest result we could hope
for.

Notice that it shows that the following two mutation\textendash competition\textendash diffusion
systems:
\[
\partial_{t}\mathbf{u}-\partial_{xx}\mathbf{u}=r\mathbf{u}+\mathbf{M}_{1}\mathbf{u}-\left(\mathbf{b}^{T}\mathbf{u}\right)\mathbf{u},
\]
\[
\partial_{t}\mathbf{u}-\partial_{xx}\mathbf{u}=r\mathbf{u}+\mathbf{M}_{2}\mathbf{u}-\left(\mathbf{b}^{T}\mathbf{u}\right)\mathbf{u},
\]
 with $r>0$ and $\mathbf{M}_{1}$ and $\mathbf{M}_{2}$ essentially
nonnegative irreducible with null Perron\textendash Frobenius eigenvalues
and equal Perron\textendash Frobenius eigenvectors, have the exact
same traveling wave solutions. In other words, all else being equal
(\textit{neutral} internal structure), the mutation strategy does
not matter. In the absence of mutations, \textit{neutral genetic diversity}
has been studied recently in a collection of papers by Garnier, Hamel,
Roques and others (for instance, we refer to \cite{Bonnefon_Coville_Garnier,Garnier_Giletti_Hamel_Roques}).
In view of their results on pulled fronts, the preceding theorem indicates
that the presence of mutations is a necessary and sufficient condition
to ensure the preservation of the genetic diversity during the invasion.

As a side note (slightly off topic), we can use the reduction to the
scalar problem to prove the following generalization of a result due
to Coville and Fabre \cite[Theorem 1.1]{Coville_Fabre_}. 
\begin{thm}
\label{thm:Global_asymptotic_stability_(H6)} Assume $\left(H_{6}\right)$
and $\mathbf{a}=\mathbf{1}_{N,1}$. 

All positive classical solutions of $\left(E_{KPP}^{0}\right)$ set
in $\left(0,+\infty\right)$ converge as $t\to+\infty$ to $\mathbf{v}^{\star}$.

Furthermore, if $\mathbf{d}=\mathbf{1}_{N,1}$, then, for all bounded
intervals $I\subset\mathbb{R}$, all bounded positive classical solutions
$\mathbf{u}$ of $\left(E_{KPP}\right)$ set in $\left(0,+\infty\right)\times\mathbb{R}$
satisfy
\[
\lim_{t\to+\infty}\sup_{x\in I}\left|\mathbf{u}\left(t,x\right)-\mathbf{v}^{\star}\right|=0.
\]

Consequently, if $\mathbf{d}=\mathbf{1}_{N,1}$, the set of bounded
nonnegative classical solutions of $\left(S_{KPP}\right)$ is exactly
$\left\{ \mathbf{0},\mathbf{v}^{\star}\right\} $.
\end{thm}

This result is proved in Section 3.3.

We believe that the preceding two theorems are robust, in that they
should remain true in a neighborhood of $\left(\mathbf{d},\mathbf{a}\right)=\left(\mathbf{1}_{N,1},\mathbf{1}_{N,1}\right)$.
In particular, \thmref{Distribution_back_(H6)} could be extended
by showing with the implicit function theorem that no solution of
$\left(TW\left[c\right]\right)$ bifurcates from $\mathbf{v}^{\star}$
at $\left(\mathbf{d},\mathbf{a}\right)=\left(\mathbf{1}_{N,1},\mathbf{1}_{N,1}\right)$.
\thmref{Global_asymptotic_stability_(H6)} could be extended thanks
to Conley index theory and a Morse decomposition, exactly as in Dockery\textendash Hutson\textendash Mischaikow\textendash Pernarowski
\cite[Section 4]{Dockery_1998}. For the sake of brevity, we do not
address these questions. 

\subsubsection{Two-component systems with linear competition and small mutations}
\begin{enumerate}[resume]
\item $N=2$, there exists $\mathbf{C}\gg\mathbf{0}$ such that
\[
\mathbf{c}\left(\mathbf{v}\right)=\mathbf{C}\mathbf{v}\text{ for all }\mathbf{v}\in\mathsf{K},
\]
and the vector $\mathbf{r}\in\mathbb{R}^{N}$ given by the unique
decomposition of $\mathbf{L}$ of the form
\[
\mathbf{L}=\text{diag}\left(\mathbf{r}\right)+\mathbf{M}\text{ with }\mathbf{1}_{1,N}\mathbf{M}=\mathbf{0}
\]
is positive.
\end{enumerate}
}

By monotonicity of $\mathbf{c}$, supplementing $\left(H_{1}\right)$\textendash $\left(H_{5}\right)$
with $\left(H_{7}\right)$ implies the existence of a profile $\mathbf{p}\in\mathscr{P}_{c}$
for all $c\geq c^{\star}$.

When $\left(H_{7}\right)$ is satisfied, we denote $\mathbf{R}=\text{diag}\left(\mathbf{r}\right)$
and define $\left(\eta,\mathbf{m}\right)\in\left(0,+\infty\right)\times\mathsf{S}^{++}\left(\mathbf{0},1\right)$
such that
\[
\mathbf{M}=\eta\left(\begin{matrix}-1 & 1\\
1 & -1
\end{matrix}\right)\text{diag}\left(\mathbf{m}\right).
\]
The quantity $\eta$ is unique and commonly referred to as the\textit{
mutation rate}.

In other words, we are considering the following system:
\[
\left\{ \begin{matrix}\partial_{t}u_{1}-d_{1}\partial_{xx}u_{1}=r_{1}u_{1}-\left(c_{1,1}u_{1}+c_{1,2}u_{2}\right)u_{1}+\eta m_{1}\left(u_{2}-u_{1}\right)\\
\partial_{t}u_{2}-d_{2}\partial_{xx}u_{2}=r_{2}u_{2}-\left(c_{2,1}u_{1}+c_{2,2}u_{2}\right)u_{2}+\eta m_{2}\left(u_{1}-u_{2}\right)
\end{matrix}\right.
\]

The idea is to assume that $\eta$ is small compared to $\mathbf{r}$
so that the mutation\textendash competition\textendash diffusion system
is close to the pure competition\textendash diffusion system
\[
\left\{ \begin{matrix}\partial_{t}u_{1}-d_{1}\partial_{xx}u_{1}=r_{1}u_{1}-\left(c_{1,1}u_{1}+c_{1,2}u_{2}\right)u_{1}\\
\partial_{t}u_{2}-d_{2}\partial_{xx}u_{2}=r_{2}u_{2}-\left(c_{2,1}u_{1}+c_{2,2}u_{2}\right)u_{2}
\end{matrix}\right..\quad\left(E_{KPP}\right)_{0}
\]
Indeed, two-component competition\textendash diffusion systems being
cooperative up to the change of unknowns $v=\frac{r_{2}}{c_{2,2}}-u_{2}$,
the maximum principle then simplifies noticeably the characterization
of the asymptotic behaviors. In particular, defining $\alpha_{i}=\frac{r_{i}}{c_{i,i}}$
for all $i\in\left\{ 1,2\right\} $ and, if $\det\mathbf{C}\neq0$,
\[
\mathbf{v}_{m}=\frac{1}{\det\mathbf{C}}\left(\begin{matrix}r_{1}c_{2,2}-r_{2}c_{1,2}\\
r_{2}c_{1,1}-r_{1}c_{2,1}
\end{matrix}\right),
\]
 the asymptotic behavior of the solutions of the spatially homogeneous
competitive system
\[
\mathbf{u}'=\mathbf{R}\mathbf{u}-\left(\mathbf{C}\mathbf{u}\right)\circ\mathbf{u}
\]
 is well-known.
\begin{enumerate}
\item {[}Extinction of $u_{2}${]} If $\frac{r_{1}}{r_{2}}\geq\max\left(\frac{c_{1,1}}{c_{2,1}},\frac{c_{1,2}}{c_{2,2}}\right)$
and $\frac{r_{1}}{r_{2}}>\min\left(\frac{c_{1,1}}{c_{2,1}},\frac{c_{1,2}}{c_{2,2}}\right)$
, then $\alpha_{1}\mathbf{e}_{1}$ is globally asymptotically stable
in $\mathsf{K}^{++}\cup\left(\text{span}\left(\mathbf{e}_{1}\right)\cap\mathsf{K}^{+}\right)$
and $\alpha_{2}\mathbf{e}_{2}$ is globally asymptotically stable
in $\text{span}\left(\mathbf{e}_{2}\right)\cap\mathsf{K}^{+}$.
\item {[}Coexistence{]} If $\frac{c_{1,2}}{c_{2,2}}<\frac{r_{1}}{r_{2}}<\frac{c_{1,1}}{c_{2,1}}$,
then $\mathbf{v}_{m}\in\mathsf{K}^{++}$, $\mathbf{v}_{m}$ is globally
asymptotically stable in $\mathsf{K}^{++}$ and, for all $i\in\left\{ 1,2\right\} $,
$\alpha_{i}\mathbf{e}_{i}$ is globally asymptotically stable in $\text{span}\left(\mathbf{e}_{i}\right)\cap\mathsf{K}^{+}$.
\item {[}Competitive exclusion{]} If $\frac{c_{1,2}}{c_{2,2}}>\frac{r_{1}}{r_{2}}>\frac{c_{1,1}}{c_{2,1}}$,
then $\mathbf{v}_{m}\in\mathsf{K}^{++}$ and a one-dimensional curve
$\mathsf{S}$, referred to as the separatrix, induces a partition
$\left(\mathsf{K}_{1}^{+},\mathsf{S},\mathsf{K}_{2}^{+}\right)$ of
$\mathsf{K}^{+}$ such that $\alpha_{i}\mathbf{e}_{i}$ is globally
asymptotically stable in $\mathsf{K}_{i}^{+}$ for all $i\in\left\{ 1,2\right\} $
and $\mathbf{v}_{m}$ is globally asymptotically stable in $\mathsf{S}$.
\item {[}Extinction of $u_{1}${]} If $\frac{r_{1}}{r_{2}}\leq\min\left(\frac{c_{1,1}}{c_{2,1}},\frac{c_{1,2}}{c_{2,2}}\right)$
and $\frac{r_{1}}{r_{2}}<\max\left(\frac{c_{1,1}}{c_{2,1}},\frac{c_{1,2}}{c_{2,2}}\right)$,
then $\alpha_{2}\mathbf{e}_{2}$ is globally asymptotically stable
in $\mathsf{K}^{++}\cup\left(\text{span}\left(\mathbf{e}_{2}\right)\cap\mathsf{K}^{+}\right)$
and $\alpha_{1}\mathbf{e}_{1}$ is globally asymptotically stable
in $\text{span}\left(\mathbf{e}_{1}\right)\cap\mathsf{K}^{+}$.
\end{enumerate}
The cases 1, 2 and 4 are \textit{monostable} whereas the case 3 is
\textit{bistable}. The case $\frac{r_{1}}{r_{2}}=\frac{c_{1,1}}{c_{2,1}}=\frac{c_{1,2}}{c_{2,2}}$
is degenerate and is usually discarded. 

In the forthcoming statements, $\eta$ is understood as a positive
parameter which can be passed to the limit $\eta\to0$ (notice that
for all $\eta>0$, $\left(H_{1}\right)$\textendash $\left(H_{5}\right)$
is satisfied indeed). The system $\left(E_{KPP}\right)$ and the objects
$\mathscr{P}_{c}$ and $c^{\star}$ depend on $\eta$ and might be
denoted respectively $\left(E_{KPP}\right)_{\eta}$, $\mathscr{P}_{c,\eta}$
and $c_{\eta}^{\star}$. We define subsequently $\mathscr{E}$ as
the set of all $\left(\eta,\mathbf{p},c\right)\in\left(0,+\infty\right)\times\mathscr{C}^{2}\left(\mathbb{R},\mathbb{R}^{2}\right)\times\left(0,+\infty\right)$
such that $\left(\mathbf{p},c\right)$ is a traveling wave solution
of $\left(E_{KPP}\right)_{\eta}$. Contrarily to the case $\eta>0$,
a traveling wave solution of the limiting system $\left(E_{KPP}\right)_{0}$
has no prescribed asymptotic behaviors. 

We point out that Morris\textendash Börger\textendash Crooks \cite{Morris_Borger_Crooks}
showed that the limit $c_{0}^{\star}$ of $\left(c_{\eta}^{\star}\right)_{\eta>0}$
as $\eta\to0$ is well-defined and satisfies as expected
\[
c_{0}^{\star}\geq2\sqrt{\max\limits _{i\in\left\{ 1,2\right\} }\left(d_{i}r_{i}\right)},
\]
with, quite interestingly, strict inequality if 
\[
1+\sqrt{1+\frac{\alpha_{i}}{\alpha_{3-i}}}<\frac{2c_{3-i,3-i}}{c_{i,3-i}}\text{ and }\frac{d_{i}}{d_{3-i}}+\frac{r_{i}}{r_{3-i}}>2\text{ for all }i\in\left\{ 1,2\right\} .
\]
 However, they did not characterize the limiting profiles. This is
what we intend to do here (but will only partially achieve). 

In the following conjecture, stability is to be understood as local
asymptotic stability with respect to $\left(E_{KPP}^{0}\right)$. 
\begin{conjecture}
\label{conj:(H_7)} Assume $\left(H_{7}\right)$. Let $\left(\mathbf{p}_{\eta}\right)_{\eta>0}$
and $\left(c_{\eta}\right)_{\eta\geq0}$ such that
\[
\left\{ \begin{matrix}\left(\eta,\mathbf{p}_{\eta},c_{\eta}\right)\in\mathscr{E} & \text{for all }\eta>0,\\
c_{0}=\lim\limits _{\eta\to0}c_{\eta}.
\end{matrix}\right.
\]
\begin{enumerate}[label=\roman*)]
\item  \label{enu:con_(H7)_bistable} Assume that both $\alpha_{1}\mathbf{e}_{1}$
and $\alpha_{2}\mathbf{e}_{2}$ are stable and that $c_{\alpha_{1}\mathbf{e}_{1}\to\alpha_{2}\mathbf{e}_{2}}\neq0$.
Then there exists $\left(\xi_{\eta}\right)_{\eta>0}$ such that $\left(\xi\mapsto\mathbf{p}_{\eta}\left(\xi+\xi_{\eta}\right),c_{\eta}\right)_{\eta>0}$
converges in $\left(\mathscr{C}_{loc}^{2}\left(\mathbb{R},\mathbb{R}^{2}\right)\cap\mathscr{L}^{\infty}\left(\mathbb{R},\mathbb{R}^{2}\right)\right)\times\mathbb{R}$
as $\eta\to0$ to a semi-extinct traveling wave solution $\left(p_{0}\mathbf{e}_{i},c_{0}\right)$
of $\left(E_{KPP}\right)_{0}$ connecting $\mathbf{0}$ to $\alpha_{i}\mathbf{e}_{i}$
with
\[
i=\left\{ \begin{matrix}1 & \text{if }c_{\alpha_{1}\mathbf{e}_{1}\to\alpha_{2}\mathbf{e}_{2}}>0,\\
2 & \text{if }c_{\alpha_{1}\mathbf{e}_{1}\to\alpha_{2}\mathbf{e}_{2}}<0.
\end{matrix}\right.
\]
\item \label{enu:con_(H7)_monostable} Assume that there is a unique stable
state $\mathbf{v}_{s}\in\left\{ \alpha_{1}\mathbf{e}_{1},\alpha_{2}\mathbf{e}_{2},\mathbf{v}_{m}\right\} $.
Then one and only one of the following two properties holds true.
\begin{enumerate}
\item There exists $\left(\xi_{\eta}\right)_{\eta>0}$ such that $\left(\xi\mapsto\mathbf{p}_{\eta}\left(\xi+\xi_{\eta}\right),c_{\eta}\right)_{\eta>0}$
converges in $\left(\mathscr{C}_{loc}^{2}\left(\mathbb{R},\mathbb{R}^{2}\right)\cap\mathscr{L}^{\infty}\left(\mathbb{R},\mathbb{R}^{2}\right)\right)\times\mathbb{R}$
as $\eta\to0$ to a component-wise monotonic traveling wave solution
$\left(\mathbf{p}_{0},c_{0}\right)$ of $\left(E_{KPP}\right)_{0}$
connecting $\mathbf{0}$ to $\mathbf{v}_{s}$.
\item There exist $\left(\xi_{\eta}^{1}\right)_{\eta>0}$, $\left(\xi_{\eta}^{2}\right)_{\eta>0}$
and a unique $i\in\mathsf{I}_{u}$ such that, as $\eta\to0$:
\begin{itemize}
\item $\xi_{\eta}^{2}-\xi_{\eta}^{1}\to+\infty$;
\item $\left(\xi\mapsto\mathbf{p}_{\eta}\left(\xi+\xi_{\eta}^{2}\right),c_{\eta}\right)_{\eta>0}$
converges in $\mathscr{C}_{loc}^{2}\left(\mathbb{R},\mathbb{R}^{2}\right)\times\mathbb{R}$
to a semi-extinct traveling wave solution $\left(p_{front}\mathbf{e}_{i},c_{0}\right)$
of $\left(E_{KPP}\right)_{0}$ connecting $\mathbf{0}$ to $\alpha_{i}\mathbf{e}_{i}$;
\item $\left(\xi\mapsto\mathbf{p}_{\eta}\left(\xi+\xi_{\eta}^{1}\right),c_{\eta}\right)_{\eta>0}$
converges in $\mathscr{C}_{loc}^{2}\left(\mathbb{R},\mathbb{R}^{2}\right)\times\mathbb{R}$
to a component-wise monotonic traveling wave solution $\left(\mathbf{p}_{back},c_{0}\right)$
of $\left(E_{KPP}\right)_{0}$ connecting $\alpha_{i}\mathbf{e}_{i}$
to $\mathbf{v}_{s}$.
\end{itemize}
\end{enumerate}
\end{enumerate}
\end{conjecture}

We emphasize once more that traveling waves with minimal speed $c_{\eta}^{\star}$
do not, in general, converge to a traveling wave with minimal speed.
In particular, \figref{Numerics_bistable_spreading} illustrates an
interesting case of invasion driven by the fast phenotype $u_{2}$
but where the only settler is the slow phenotype $u_{1}$. This is
reminiscent of Griette\textendash Raoul \cite{Griette_Raoul}, where
an analogous result was established analytically under a stronger
scaling. 

\conjref{(H_7)}, \ref{enu:con_(H7)_bistable} is expected to be a
very difficult problem and seems to be beyond our reach. We leave
it as an open problem. 

On the contrary, regarding \conjref{(H_7)}, \ref{enu:con_(H7)_monostable},
a partial confirmation is within reach. On one hand, we point out
that the special case
\[
\frac{c_{1,1}}{c_{2,1}}=\frac{c_{1,2}}{c_{2,2}}=1\text{ and }\mathbf{d}=\mathbf{1}_{2,1}
\]
is somehow solved by \thmref{Distribution_back_(H6)} without any
assumption on $\mathbf{r}$. On the other hand, we also have the following
general theorem which concerns all monostable cases apart from
\[
\frac{c_{1,1}}{c_{2,1}}<\frac{c_{1,2}}{c_{2,2}}=\frac{r_{1}}{r_{2}},
\]
\[
\frac{r_{1}}{r_{2}}=\frac{c_{1,1}}{c_{2,1}}<\frac{c_{1,2}}{c_{2,2}}.
\]
\begin{thm}
\label{thm:(H7)_Monostable} Assume $\left(H_{7}\right)$ and the
existence of $i\in\left\{ 1,2\right\} $ such that
\[
\frac{r_{i}}{r_{3-i}}>\frac{c_{i,3-i}}{c_{3-i,3-i}}.
\]

Let 
\[
\mathbf{v}_{s}=\left\{ \begin{matrix}\alpha_{i}\mathbf{e}_{i} & \text{if }\frac{r_{i}}{r_{3-i}}\geq\frac{c_{i,i}}{c_{3-i,i}},\\
\mathbf{v}_{m} & \text{if }\frac{r_{i}}{r_{3-i}}<\frac{c_{i,i}}{c_{3-i,i}}.
\end{matrix}\right.
\]

For all $\left(\mathbf{p}_{\eta}\right)_{\eta>0}$ and $\left(c_{\eta}\right)_{\eta\geq0}$
such that
\[
\left\{ \begin{matrix}\left(\eta,\mathbf{p}_{\eta},c_{\eta}\right)\in\mathscr{E} & \text{for all }\eta>0,\\
c_{0}=\lim\limits _{\eta\to0}c_{\eta},
\end{matrix}\right.
\]
 there exists $\left(\zeta_{\eta}\right)_{\eta>0}$ such that, as
$\eta\to0$, $\left(\xi\mapsto\mathbf{p}_{\eta}\left(\xi+\zeta_{\eta}\right),c_{\eta}\right)_{\eta>0}$
converges up to extraction in $\mathscr{C}_{loc}^{2}\left(\mathbb{R},\mathbb{R}^{2}\right)\times\mathbb{R}$
to a traveling wave solution $\left(\mathbf{p},c_{0}\right)$ of $\left(E_{KPP}\right)_{0}$
achieving one of the following connections:
\begin{enumerate}
\item $\mathbf{0}$ to $\mathbf{v}_{s}$,
\item $\alpha_{3-i}\mathbf{e}_{3-i}$ to $\mathbf{v}_{s}$,
\item $\mathbf{0}$ to $\alpha_{i}\mathbf{e}_{i}$ with $\mathbf{p}$ semi-extinct.
\end{enumerate}
\end{thm}

This result is proved in Section 4.

Let us clarify how this result confirms partially \conjref{(H_7)},
\ref{enu:con_(H7)_monostable} and what are the remaining open questions. 
\begin{itemize}
\item Assume $\mathbf{v}_{s}=\mathbf{v}_{m}$. Up to the component-wise
monotonicity of the profile in the first and second cases, the three
connections above correspond exactly to the three possible limiting
profiles of \conjref{(H_7)}, \ref{enu:con_(H7)_monostable}. Moreover
we can apply the theorem with $i=1$ and $i=2$ and obtain two limiting
profiles. However, at this point, the normalizations $\left(\zeta_{\eta}^{1}\right)_{\eta>0}$
and $\left(\zeta_{\eta}^{2}\right)_{\eta>0}$ are unrelated and nine
possible pairs of profiles seem to exist. We do not know how to prove
that only the three following situations actually occur: $\mathbf{0}$
to $\mathbf{v}_{m}$ and $\mathbf{0}$ to $\mathbf{v}_{m}$ with $\left(\zeta_{\eta}^{2}-\zeta_{\eta}^{1}\right)_{\eta>0}$
bounded, semi-extinct $\mathbf{0}$ to $\alpha_{1}\mathbf{e}_{1}$
and $\alpha_{1}\mathbf{e}_{1}$ to $\mathbf{v}_{m}$ with $\zeta_{\eta}^{2}-\zeta_{\eta}^{1}\to-\infty$,
semi-extinct $\mathbf{0}$ to $\alpha_{2}\mathbf{e}_{2}$ and $\alpha_{2}\mathbf{e}_{2}$
to $\mathbf{v}_{m}$ with $\zeta_{\eta}^{2}-\zeta_{\eta}^{1}\to+\infty$.
\item Assume $\mathbf{v}_{s}=\alpha_{i}\mathbf{e}_{i}$. The third connection
above is actually a subcase of the first one and the normalization
$\left(\zeta_{\eta}\right)_{\eta>0}$ is unable to track the semi-extinct
limiting profile connecting $\mathbf{0}$ to $\alpha_{3-i}\mathbf{e}_{3-i}$.
This is not a question of optimality of the proof: the normalization
$\left(\zeta_{\eta}\right)_{\eta>0}$ is precisely chosen so that
$p_{i}$ is always non-zero. Hence $\left(\zeta_{\eta}\right)_{\eta>0}$
corresponds either to $\left(\xi_{\eta}\right)_{\eta>0}$ or to $\left(\xi_{\eta}^{1}\right)_{\eta>0}$.
The construction of the normalization $\left(\xi_{\eta}^{2}\right)_{\eta>0}$
of \conjref{(H_7)}, \ref{enu:con_(H7)_monostable} is a completely
open problem. Of course, once this problem is solved, it remains to
relate the limiting profiles and the normalizations, as in the case
$\mathbf{v}_{s}=\mathbf{v}_{m}$.
\end{itemize}

\section{The edge of the fronts}

In this section, we fix a traveling wave $\left(\mathbf{p},c\right)$
and we prove \thmref{Distribution_edge}.

\subsection{Preparatory lemmas and the Ikehara theorem}
\begin{lem}
\label{lem:Exponential_decay_of_the_profile} For all $i\in\left[N\right]$,
\[
\left\{ \liminf_{+\infty}\frac{-p_{i}'}{p_{i}},\limsup_{+\infty}\frac{-p_{i}'}{p_{i}}\right\} \subset\left\{ \mu\in\left(0,+\infty\right)\ |\ \frac{\lambda_{PF}\left(\mu^{2}\mathbf{D}+\mathbf{L}\right)}{\mu}=c\right\} ,
\]
\[
\left\{ \liminf_{+\infty}\frac{p_{i}''}{p_{i}},\limsup_{+\infty}\frac{p_{i}''}{p_{i}}\right\} \subset\left\{ \mu^{2}\in\left(0,+\infty\right)\ |\ \frac{\lambda_{PF}\left(\mu^{2}\mathbf{D}+\mathbf{L}\right)}{\mu}=c\right\} .
\]

Consequently, there exists $\tilde{\xi}\in\mathbb{R}$ such that $\mathbf{p}$
is component-wise strictly convex in $[\tilde{\xi},+\infty)$.
\end{lem}

\begin{proof}
The proof of
\[
\min\limits _{i\in\left[N\right]}\liminf\limits _{+\infty}\frac{-p_{i}'}{p_{i}}\in\left\{ \mu>0\ |\ \frac{\lambda_{PF}\left(\mu^{2}\mathbf{D}+\mathbf{L}\right)}{\mu}=c\right\} 
\]
can be found in \cite[Proposition 6.10]{Girardin_2016_2}. The proof
also directly yields that for any sequence $\left(\xi_{n}\right)_{n\in\mathbb{N}}$
such that $\xi_{n}\to+\infty$ and such that there exists $j\in\left[N\right]$
satisfying
\[
\lim\limits _{n\to+\infty}\frac{-p_{j}'\left(\xi_{n}\right)}{p_{j}\left(\xi_{n}\right)}=\min\limits _{i\in\left[N\right]}\liminf\limits _{+\infty}\frac{-p_{i}'}{p_{i}},
\]
convergence occurs in the following sense:
\[
\lim\limits _{n\to+\infty}\left(\frac{-p_{i}'\left(\xi_{n}\right)}{p_{i}\left(\xi_{n}\right)}\right)_{i\in\left[N\right]}=\left(\min\limits _{i\in\left[N\right]}\liminf\limits _{+\infty}\frac{-p_{i}'}{p_{i}}\right)\mathbf{1}_{N,1}.
\]

The proof of 
\[
\max\limits _{i\in\left[N\right]}\limsup\limits _{+\infty}\frac{-p_{i}'}{p_{i}}\in\left\{ \mu>0\ |\ \frac{\lambda_{PF}\left(\mu^{2}\mathbf{D}+\mathbf{L}\right)}{\mu}=c\right\} 
\]
is a slight modification of the preceding proof, where the quantity
\[
\overline{\Lambda}=\max_{i\in\left[N\right]}\limsup_{\xi\to+\infty}\frac{p_{i}'\left(\xi\right)}{p_{i}\left(\xi\right)}
\]
is replaced by
\[
\underline{\Lambda}=\min_{i\in\left[N\right]}\liminf_{\xi\to+\infty}\frac{p_{i}'\left(\xi\right)}{p_{i}\left(\xi\right)}.
\]
 Similarly, we also obtain directly that for any sequence $\left(\xi_{n}\right)_{n\in\mathbb{N}}$
such that $\xi_{n}\to+\infty$ and such that there exists $j\in\left[N\right]$
satisfying
\[
\lim\limits _{n\to+\infty}\frac{-p_{j}'\left(\xi_{n}\right)}{p_{j}\left(\xi_{n}\right)}=\max\limits _{i\in\left[N\right]}\limsup\limits _{+\infty}\frac{-p_{i}'}{p_{i}},
\]
convergence occurs in the following sense:
\[
\lim\limits _{n\to+\infty}\left(\frac{-p_{i}'\left(\xi_{n}\right)}{p_{i}\left(\xi_{n}\right)}\right)_{i\in\left[N\right]}=\left(\max\limits _{i\in\left[N\right]}\limsup\limits _{+\infty}\frac{-p_{i}'}{p_{i}}\right)\mathbf{1}_{N,1}.
\]

The statements regarding $\left(\frac{p_{i}''}{p_{i}}\right)_{i\in\left[N\right]}$
are again established very similarly. The quantity
\[
\overline{\Lambda}=\max_{i\in\left[N\right]}\limsup_{\xi\to+\infty}\frac{p_{i}'\left(\xi\right)}{p_{i}\left(\xi\right)}
\]
is replaced by
\[
\underline{\Theta}=\min_{i\in\left[N\right]}\liminf_{\xi\to+\infty}\frac{p_{i}''\left(\xi\right)}{p_{i}\left(\xi\right)}
\]
and 
\[
\overline{\Theta}=\max_{i\in\left[N\right]}\liminf_{\xi\to+\infty}\frac{p_{i}''\left(\xi\right)}{p_{i}\left(\xi\right)}
\]
 respectively, and the function 
\[
\mathbf{w}_{n}=\overline{\Lambda}\hat{\mathbf{p}}_{n}-\hat{\mathbf{p}}_{n}'
\]
 is replaced by
\[
\mathbf{w}_{n}=\underline{\Theta}\hat{\mathbf{p}}_{n}-\hat{\mathbf{p}}_{n}''
\]
and
\[
\mathbf{w}_{n}=\overline{\Theta}\hat{\mathbf{p}}_{n}-\hat{\mathbf{p}}_{n}''
\]
 respectively. Since $\hat{\mathbf{p}}_{\infty}$ is nonnegative nonzero
and $\mathbf{w}_{\infty}=\mathbf{0}$, necessarily $\underline{\Theta}>0$
and $\overline{\Theta}>0$ and then, as in \cite[Proposition 6.10]{Girardin_2016_2},
both quantities have the form $\mu^{2}$ with $\mu$ solution of $\frac{\lambda_{PF}\left(\mu^{2}\mathbf{D}+\mathbf{L}\right)}{\mu}=c$. 

Finally, the strict convexity in a neighborhood of $+\infty$ is deduced
exactly as the monotonicity in the proof of \cite[Proposition 6.10]{Girardin_2016_2}.
\end{proof}
We will also need the Ikehara theorem \cite[Proposition 2.3]{Carr_Chmaj},
commonly used in such problems (see for instance Guo\textendash Wu
\cite{Jong_Shenq12}), as well as a lemma due to Volpert, Volpert
and Volpert \cite[Chapter 5, Lemma 4.1]{V_V_V}.
\begin{thm}
{[}Ikehara{]} Let $f:\left(0,+\infty\right)\to\left(0,+\infty\right)$
be a decreasing function. Assume that there exist $\overline{\lambda}\in\left(0,+\infty\right)$,
$k\in\left(-1,+\infty\right)$ and an analytic function 
\[
h:\left(0,\overline{\lambda}\right]+i\mathbb{R}\to\left(0,+\infty\right)
\]
such that
\[
\int_{0}^{+\infty}\text{e}^{\lambda x}f\left(x\right)\text{d}x=\frac{h\left(\lambda\right)}{\left(\overline{\lambda}-\lambda\right)^{k+1}}\text{ for all }\lambda\in\left(0,\overline{\lambda}\right).
\]

Then
\[
\lim_{x\to+\infty}f\left(x\right)\frac{\text{e}^{\overline{\lambda}x}}{x^{k}}=\frac{h\left(\overline{\lambda}\right)}{\Gamma\left(\overline{\lambda}+1\right)}.
\]
\end{thm}

\begin{lem}
\label{lem:Volperts} {[}Volpert\textendash Volpert\textendash Volpert{]}
Let $\mathbf{A}$ be an essentially nonnegative matrix and let $\mathbf{z}\in\mathbb{C}^{N}$. 

If 
\[
\left\{ \begin{matrix}\text{sp}\mathbf{A}\subset\left(-\infty,0\right)+i\mathbb{R},\\
\left(\text{Re}\left(z_{k}\right)\right)_{k\in\left[N\right]}\leq\mathbf{0},
\end{matrix}\right.
\]
then
\[
\text{sp}\left(\mathbf{A}+\text{diag}\left(\mathbf{z}\right)\right)\subset\left(-\infty,0\right)+i\mathbb{R}.
\]
\end{lem}

\subsection{Convergence at the edge}

Let 
\[
k_{c}=\left\{ \begin{matrix}0 & \text{if }c>c^{\star},\\
1 & \text{if }c=c^{\star}.
\end{matrix}\right.
\]
\begin{prop}
\label{prop:Edge} There exists $A>0$ such that, as $\xi\to+\infty$,
\[
\left\{ \begin{matrix}\mathbf{p}\left(\xi\right)\sim A\xi^{k_{c}}\text{e}^{-\mu_{c}\xi}\mathbf{n}_{\mu_{c}},\\
\mathbf{p}'\left(\xi\right)\sim-\mu_{c}\mathbf{p}\left(\xi\right),\\
\mathbf{p}''\left(\xi\right)\sim\mu_{c}^{2}\mathbf{p}\left(\xi\right).
\end{matrix}\right.
\]
\end{prop}

\begin{proof}
Fix temporarily $\mu\in\left(0,\mu_{c}\right)+i\mathbb{R}$. In view
of \lemref{Exponential_decay_of_the_profile} and of the Gronwall
lemma, 
\[
\xi\mapsto\text{e}^{\mu\xi}\mathbf{p}\left(\xi\right)\in\mathscr{L}^{1}\left(\mathbb{R},\mathbb{C}^{N}\right),
\]
\[
\xi\mapsto\text{e}^{\mu\xi}\mathbf{c}\left(\mathbf{p}\left(\xi\right)\right)\circ\mathbf{p}\left(\xi\right)\in\mathscr{L}^{1}\left(\mathbb{R},\mathbb{C}^{N}\right).
\]
Multiplying $\left(TW\left[c\right]\right)$ by $\text{e}^{\mu\xi}$,
integrating by parts over $\mathbb{R}$ and defining 
\[
\mathbf{f}_{+}\left(\mu\right)=\int_{0}^{+\infty}\text{e}^{\mu\xi}\mathbf{p}\left(\xi\right)\text{d}\xi,
\]
\[
\mathbf{f}_{-}\left(\mu\right)=\int_{-\infty}^{0}\text{e}^{\mu\xi}\mathbf{p}\left(\xi\right)\text{d}\xi,
\]
\[
\mathbf{f}_{\mathbf{c}}\left(\mu\right)=\int_{\mathbb{R}}\text{e}^{\mu\xi}\mathbf{c}\left(\mathbf{p}\left(\xi\right)\right)\circ\mathbf{p}\left(\xi\right)\text{d}\xi,
\]
we get easily
\[
\left(\mu^{2}\mathbf{D}-c\mu\mathbf{I}+\mathbf{L}\right)\left(\mathbf{f}_{+}\left(\mu\right)+\mathbf{f}_{-}\left(\mu\right)\right)=\mathbf{f}_{\mathbf{c}}\left(\mu\right),
\]
whence, denoting $\text{adj}\left(\mu^{2}\mathbf{D}-c\mu\mathbf{I}+\mathbf{L}\right)$
the adjugate matrix of $\mu^{2}\mathbf{D}-c\mu\mathbf{I}+\mathbf{L}$,
we find
\[
\det\left(\mu^{2}\mathbf{D}-c\mu\mathbf{I}+\mathbf{L}\right)\mathbf{f}_{+}\left(\mu\right)=\text{adj}\left(\mu^{2}\mathbf{D}-c\mu\mathbf{I}+\mathbf{L}\right)\mathbf{f}_{\mathbf{c}}\left(\mu\right)-\det\left(\mu^{2}\mathbf{D}-c\mu\mathbf{I}+\mathbf{L}\right)\mathbf{f}_{-}\left(\mu\right).
\]

The functions $\mathbf{f}_{+}$, $\mathbf{f}_{-}$ and $\mathbf{f}_{\mathbf{c}}$
defined above are respectively analytic in $\left(0,\mu_{c}\right)+i\mathbb{R}$,
$\left(0,+\infty\right)+i\mathbb{R}$ and $\left(0,2\mu_{c}\right)+i\mathbb{R}$
(by local Lipschitz-continuity of $\mathbf{c}$, $\left(H_{2}\right)$
and global boundedness of $\mathbf{p}$).

The function 
\[
\begin{matrix}\mathbb{C} & \to & \mathbb{C}\\
\mu & \mapsto & \det\left(\mu^{2}\mathbf{D}-c\mu\mathbf{I}+\mathbf{L}\right)
\end{matrix}
\]
 is polynomial (whence analytic). Let $\mathsf{Z}\subset\mathbb{C}$
be the finite set of its roots, counted with algebraic multiplicity.
In particular, $\mu_{c}\in\mathsf{Z}$ with multiplicity $k_{c}+1$.

For all $\mu\in\left(\left(0,\mu_{c}\right)+i\mathbb{R}\right)\backslash\mathsf{Z}$,
\[
\mathbf{f}_{+}\left(\mu\right)=\left(\mu^{2}\mathbf{D}-c\mu\mathbf{I}+\mathbf{L}\right)^{-1}\mathbf{f}_{\mathbf{c}}\left(\mu\right)-\mathbf{f}_{-}\left(\mu\right).
\]
The function 
\[
\mu\mapsto\left(\mu^{2}\mathbf{D}-c\mu\mathbf{I}+\mathbf{L}\right)^{-1}\mathbf{f}_{\mathbf{c}}\left(\mu\right)
\]
 is well-defined and analytic in $\left(\left(0,\mu_{c}\right)+i\mathbb{R}\right)\backslash\mathsf{Z}$,
where it coincides with $\mathbf{f}_{+}+\mathbf{f}_{-}$ which is
analytic in $\left(0,\mu_{c}\right)+i\mathbb{R}$. 

Define the analytic function
\[
\begin{matrix}\mathbf{h}: & \left(0,\mu_{c}\right)+i\mathbb{R} & \to & \mathbb{R}^{N}\\
 & \mu & \mapsto & \left(\mu_{c}-\mu\right)^{k_{c}+1}\mathbf{f}_{+}\left(\mu\right)
\end{matrix}
\]
 so that 
\[
\mathbf{f}_{+}\left(\mu\right)=\frac{\mathbf{h}\left(\mu\right)}{\left(\mu_{c}-\mu\right)^{k_{c}+1}}\text{ for all }\mu\in\left(0,\mu_{c}\right)+i\mathbb{R}.
\]
Since, for all $\mu\in\left(\left(0,\mu_{c}\right)+i\mathbb{R}\right)\backslash\mathsf{Z}$,
\[
\mathbf{h}\left(\mu\right)=\frac{\left(\mu_{c}-\mu\right)^{k_{c}+1}}{\det\left(\mu^{2}\mathbf{D}-c\mu\mathbf{I}+\mathbf{L}\right)}\text{adj}\left(\mu^{2}\mathbf{D}-c\mu\mathbf{I}+\mathbf{L}\right)\mathbf{f}_{\mathbf{c}}\left(\mu\right)-\left(\mu_{c}-\mu\right)^{k_{c}+1}\mathbf{f}_{-}\left(\mu\right)
\]
the function $\mathbf{h}$ can be analytically extended on $\left(0,\mu_{c}\right]+i\mathbb{R}$
if and only if 
\[
\mu\mapsto\frac{\left(\mu_{c}-\mu\right)^{k_{c}+1}}{\det\left(\mu^{2}\mathbf{D}-c\mu\mathbf{I}+\mathbf{L}\right)}
\]
 has no pole in $\left\{ \mu_{c}\right\} +i\mathbb{R}$. 

Let $\theta\in\mathbb{R}\backslash\left\{ 0\right\} $. In view of
\[
\left(\mu_{c}+i\theta\right)^{2}\mathbf{D}-c\left(\mu_{c}+i\theta\right)\mathbf{I}+\mathbf{L}=\mu_{c}^{2}\mathbf{D}-c\mu_{c}\mathbf{I}+\mathbf{L}-\theta^{2}\mathbf{D}+i\theta\left(2\mu_{c}\mathbf{D}-c\mathbf{I}\right)
\]
 and
\begin{align*}
\lambda_{PF}\left(\mu_{c}^{2}\mathbf{D}-c\mu_{c}\mathbf{I}+\mathbf{L}-\theta^{2}\mathbf{D}\right) & \leq\lambda_{PF}\left(\mu_{c}^{2}\mathbf{D}-c\mu_{c}\mathbf{I}+\mathbf{L}-\theta^{2}\min_{k\in\left[N\right]}d_{k}\right)\\
 & =\lambda_{PF}\left(\mu_{c}^{2}\mathbf{D}-c\mu_{c}\mathbf{I}+\mathbf{L}\right)-\theta^{2}\min_{k\in\left[N\right]}d_{k}\\
 & =-\theta^{2}\min_{k\in\left[N\right]}d_{k}
\end{align*}
 \lemref{Volperts} yields that
\[
\text{sp}\left(\left(\mu_{c}^{2}\mathbf{D}-c\mu_{c}\mathbf{I}+\mathbf{L}-\theta^{2}\mathbf{D}\right)+\text{diag}\left(i\theta\left(2\mu_{c}d_{k}-c\right)\right)_{k\in\left[N\right]}\right)\subset\left(-\infty,0\right)+i\mathbb{R}.
\]
Hence $\mu\mapsto\frac{\left(\mu_{c}-\mu\right)^{k_{c}+1}}{\det\left(\mu^{2}\mathbf{D}-c\mu\mathbf{I}+\mathbf{L}\right)}$
has no pole in $\left\{ \mu_{c}\right\} +i\left(\mathbb{R}\backslash\left\{ 0\right\} \right)$
and then it has no pole in $\left\{ \mu_{c}\right\} +i\mathbb{R}$
indeed.

We are now in position to apply the Ikehara theorem component-wise
and to deduce from it the existence of $\mathbf{n}\in\mathsf{S}^{+}\left(\mathbf{0},1\right)$
and $A\geq0$ such that
\[
\lim_{\xi\to+\infty}\mathbf{p}\left(\xi\right)\frac{\text{e}^{\mu_{c}\xi}}{\xi^{k_{c}}}=A\mathbf{n}.
\]
In particular, for all $k\in\left[N\right]$ such that $n_{k}>0$,
\[
\lim_{\zeta\to+\infty}\frac{\mathbf{p}\left(\xi+\zeta\right)}{p_{k}\left(\zeta\right)}\text{e}^{\mu_{c}\xi}=\frac{1}{n_{k}}\mathbf{n}.
\]
However, back to the proof of \lemref{Exponential_decay_of_the_profile},
there exists $k\in\left[N\right]$ and a sequence $\left(\xi_{n}\right)_{n\in\mathbb{N}}$
such that $\xi_{n}\to+\infty$, $\left(\frac{-p_{k}'\left(\xi_{n}\right)}{p_{k}\left(\xi_{n}\right)}\right)_{n\in\mathbb{N}}$
converges to
\[
\mu=\max\limits _{k\in\left[N\right]}\limsup\limits _{+\infty}\frac{-p_{k}'}{p_{k}},
\]
 and 
\[
\left(\xi\mapsto\frac{\mathbf{p}\left(\xi+\zeta_{n}\right)}{p_{k}\left(\zeta_{n}\right)}\right)_{n\in\mathbb{N}}
\]
 converges in $\mathscr{C}_{loc}^{2}$ to 
\[
\xi\mapsto\frac{1}{n_{\mu,k}}\text{e}^{-\mu\xi}\mathbf{n}_{\mu}.
\]
This clearly implies $\mu=\mu_{c}$ and $\mathbf{n}=\mathbf{n}_{\mu_{c}}$. 

Consequently, $A>0$,
\[
\lim_{\xi\to+\infty}\mathbf{p}\left(\xi\right)\frac{\text{e}^{\mu_{c}\xi}}{\xi^{k_{c}}}=A\mathbf{n}_{\mu_{c}},
\]
and, by \lemref{Exponential_decay_of_the_profile},
\[
\mu_{c}\leq\min\limits _{k\in\left[N\right]}\liminf\limits _{+\infty}\frac{-p_{k}'}{p_{k}}\leq\max\limits _{k\in\left[N\right]}\limsup\limits _{+\infty}\frac{-p_{k}'}{p_{k}}=\mu_{c},
\]
that is
\[
\lim_{+\infty}\left(\frac{-p_{k}'}{p_{k}}\right)_{k\in\left[N\right]}=\mu_{c}.
\]
Quite similarly, we also obtain
\[
\lim_{+\infty}\left(\frac{p_{k}''}{p_{k}}\right)_{k\in\left[N\right]}=\mu_{c}^{2}.
\]
\end{proof}
If $\mathbf{d}=\mathbf{1}_{N,1}$, the quantities at hand are:
\begin{align*}
\left(\mu_{c},\mathbf{n}_{\mu_{c}}\right) & =\left(\min\left\{ \mu>0\ |\ \lambda_{PF}\left(\mu^{2}\mathbf{I}+\mathbf{L}\right)=c\mu\right\} ,\mathbf{n}_{PF}\left(\mu_{c}^{2}\mathbf{I}+\mathbf{L}\right)\right)\\
 & =\left(\frac{1}{2}\left(c-\sqrt{c^{2}-4\lambda_{PF}\left(\mathbf{L}\right)}\right),\mathbf{n}_{PF}\left(\mathbf{L}\right)\right)
\end{align*}
and an obvious corollary follows.

\section{The back of the fronts: separated competition}

In this section, we assume $\left(H_{6}\right)$ and $\mathbf{a}=\mathbf{1}_{N,1}$
and prove \thmref{Distribution_back_(H6)} and \thmref{Global_asymptotic_stability_(H6)}. 

\subsection{Main tools: Jordan normal form and Perron\textendash Frobenius projection}

Let $m\in\left[N\right]$ be the number of pairwise distinct eigenvalues
of $\mathbf{L}$ ($\lambda_{PF}\left(\mathbf{L}\right)$ being simple,
$m\geq2$) and let $\left(\lambda_{k}\right)_{k\in\left[m\right]}\in\mathbb{C}^{m}$
be the pairwise distinct complex eigenvalues of $\mathbf{L}$ ordered
so that $\left(\text{Re}\left(\lambda_{k}\right)\right)_{k\in\left[m\right]}$
is a nondecreasing family (in particular, $\lambda_{m}=\lambda_{PF}\left(\mathbf{L}\right)$
and $\text{Re}\left(\lambda_{m-1}\right)<\lambda_{PF}\left(\mathbf{L}\right)$). 

Let $\mathbf{P}\in\mathsf{GL}\left(\mathbb{C}\right)$ be such that
$\mathbf{J}=\mathbf{P}\mathbf{L}\mathbf{P}^{-1}$ is the Jordan normal
form of $\mathbf{L}$:
\[
\mathbf{J}=\left(\begin{matrix}\lambda_{PF}\left(\mathbf{L}\right) & \mathbf{0} & \cdots & \mathbf{0}\\
\mathbf{0} & \mathbf{J}_{m-1} & \ddots & \vdots\\
\vdots & \ddots & \ddots & \mathbf{0}\\
\mathbf{0} & \cdots & \mathbf{0} & \mathbf{J}_{1}
\end{matrix}\right),
\]
 where, for all $k\in\left[m-1\right]$, $\mathbf{J}_{k}$ is the
(upper triangular) Jordan block associated with the eigenvalue $\lambda_{k}$. 

Noticing that
\[
\mathbf{L}\mathbf{P}^{-1}\mathbf{e}_{1}=\mathbf{P}^{-1}\mathbf{J}\mathbf{e}_{1}=\lambda_{PF}\left(\mathbf{L}\right)\mathbf{P}^{-1}\mathbf{e}_{1},
\]
\[
\mathbf{e}_{1}^{T}\mathbf{P}\mathbf{L}=\mathbf{e}_{1}^{T}\mathbf{J}\mathbf{P}=\lambda_{PF}\left(\mathbf{L}\right)\mathbf{e}_{1}^{T}\mathbf{P},
\]
it follows that $\mathbf{P}^{-1}\mathbf{e}_{1}\in\text{span}\mathbf{n}_{PF}\left(\mathbf{L}\right)$
and $\mathbf{e}_{1}^{T}\mathbf{P}\in\text{span}\mathbf{n}_{PF}\left(\mathbf{L}^{T}\right)^{T}$.
In particular, we can normalize without loss of generality $\mathbf{P}$
so that $\mathbf{P}^{-1}\mathbf{e}_{1}=\mathbf{n}_{PF}\left(\mathbf{L}\right)$
and then deduce from $\mathbf{e}_{1}^{T}\mathbf{P}\mathbf{n}_{PF}\left(\mathbf{L}\right)=1$
that
\[
\mathbf{e}_{1}^{T}\mathbf{P}=\frac{1}{\mathbf{n}_{PF}\left(\mathbf{L}^{T}\right)^{T}\mathbf{n}_{PF}\left(\mathbf{L}\right)}\mathbf{n}_{PF}\left(\mathbf{L}^{T}\right)^{T}.
\]

From the preceding equality, it follows directly that the Perron\textendash Frobenius
projection, defined as
\[
\boldsymbol{\Pi}_{PF}\left(\mathbf{L}\right)=\frac{\mathbf{n}_{PF}\left(\mathbf{L}\right)\mathbf{n}_{PF}\left(\mathbf{L}^{T}\right)^{T}}{\mathbf{n}_{PF}\left(\mathbf{L}^{T}\right)^{T}\mathbf{n}_{PF}\left(\mathbf{L}\right)},
\]
satisfies 
\[
\mathbf{P}\boldsymbol{\Pi}_{PF}\left(\mathbf{L}\right)\mathbf{P}^{-1}=\text{diag}\left(\mathbf{e}_{1}\right).
\]

\subsection{Uniqueness up to translation of the profile}

In this subsection, we assume $\mathbf{d}=\mathbf{1}_{N,1}$, we fix
$c\geq c^{\star}$ and we prove \thmref{Distribution_back_(H6)}.
The scalar front $p_{c}$ is defined as in the statement of the theorem. 
\begin{prop}
\label{prop:Uniqueness_profiles} All $\mathbf{p}\in\mathscr{P}_{c}$
have the form
\[
\mathbf{p}:\xi\mapsto p_{c}\left(\xi-\xi_{0}\right)\mathbf{n}_{PF}\left(\mathbf{L}\right)\text{ with }\xi_{0}\in\mathbb{R}.
\]
\end{prop}

\begin{proof}
Let $\mathbf{p}\in\mathscr{P}_{c}$ and 
\[
\mathbf{q}=\mathbf{P}\mathbf{p}\in\mathscr{C}^{2}\left(\mathbb{R},\mathbb{C}^{N}\right)\cap\mathscr{L}^{\infty}\left(\mathbb{R},\mathbb{C}^{N}\right).
\]

Multiplying $\left(TW\left[c\right]\right)$ on the left by $\mathbf{P}$,
we get
\[
-\mathbf{q}''-c\mathbf{q}'=\mathbf{J}\mathbf{q}-b\left[\mathbf{P}^{-1}\mathbf{q}\right]\mathbf{q}\text{ in }\mathbb{R},
\]
and in particular
\[
-q_{1}''-cq_{1}'=\left(\lambda_{PF}\left(\mathbf{L}\right)-b\left[\mathbf{P}^{-1}\mathbf{q}\right]\right)q_{1}\text{ in }\mathbb{R}.
\]

Since
\begin{align*}
\left(\boldsymbol{\Pi}_{PF}\left(\mathbf{L}\right)\mathbf{p}\right)^{T}\mathbf{n}_{PF}\left(\mathbf{L}\right) & =\left(\mathbf{P}^{-1}\text{diag}\left(\mathbf{e}_{1}\right)\mathbf{q}\right)^{T}\mathbf{n}_{PF}\left(\mathbf{L}\right)\\
 & =q_{1}\left(\mathbf{P}^{-1}\mathbf{e}_{1}\right)^{T}\mathbf{n}_{PF}\left(\mathbf{L}\right)\\
 & =q_{1},
\end{align*}
 $q_{1}$ is real-valued and in fact positive in $\mathbb{R}$.

First, let us verify that $\frac{q_{k}}{q_{1}}$ is globally bounded
in $\mathbb{R}$ for all $k\in\left[N\right]\backslash\left\{ 1\right\} $.
It is bounded in $(-\infty,0]$ since $\inf\limits _{(-\infty,0]}q_{1}>0$
by \cite[Theorem 1.5, iii)]{Girardin_2016_2}. It is bounded in $[0,+\infty)$
since a left-multiplication of the first equivalent of \thmref{Distribution_edge}
by $\mathbf{P}$ yields
\[
\mathbf{q}\left(\xi\right)\sim A\xi^{k}\text{e}^{-\frac{1}{2}\left(c-\sqrt{c^{2}-4\lambda_{PF}\left(\mathbf{L}\right)}\right)\xi}\mathbf{e}_{1}
\]
 whence
\[
\limsup_{+\infty}\left|\frac{q_{k}}{q_{1}}\right|=0.
\]

Next, let us show by induction that $q_{N+1-k}=0$ in $\mathbb{R}$
for all $k\in\left[N-1\right]$. 
\begin{itemize}
\item Basis: $k=1$. Due to the special form of $\mathbf{J}$, the equation
satisfied by $q_{N}$ is
\[
-q_{N}''-cq_{N}'=\left(\lambda_{1}-b\left[\mathbf{P}^{-1}\mathbf{q}\right]\right)q_{N}\text{ in }\mathbb{R}.
\]
Define $z=\frac{q_{N}}{q_{1}}$ and $w=\left|z\right|^{2}$. The function
$w$ is nonnegative and globally bounded. From
\[
z'=\frac{q_{N}'}{q_{1}}-\frac{q_{1}'}{q_{1}}z,
\]
\[
z''=\frac{q_{N}''}{q_{1}}-\frac{q_{1}''}{q_{1}}z-\frac{2q_{1}'}{q_{1}}z',
\]
it follows
\[
-z''-\frac{q_{1}c+2q_{1}'}{q_{1}}z'-\frac{q_{1}''+cq_{1}'}{q_{1}}z=\left(\lambda_{1}-b\left[\mathbf{P}^{-1}\mathbf{q}\right]\right)z\text{ in }\mathbb{R}.
\]
Using the equality satisfied by $q_{1}$, this equation reads: 
\[
-z''-\frac{q_{1}c+2q_{1}'}{q_{1}}z'+\left(\lambda_{PF}\left(\mathbf{L}\right)-\lambda_{1}\right)z=0\text{ in }\mathbb{R}.
\]
Now, multiplying by $\overline{z}$, taking the real part, defining
\[
\gamma=2\left(\lambda_{PF}\left(\mathbf{L}\right)-\text{Re}\left(\lambda_{1}\right)\right)>0
\]
 and using the obvious equality
\begin{align*}
\text{Re}\left(z''\overline{z}\right) & =\text{Re}\left(z\right)''\text{Re}\left(z\right)+\text{Im}\left(z\right)''\text{Im}\left(z\right)\\
 & =\frac{1}{2}w''-\left(\text{Re}\left(z\right)'\right)^{2}-\left(\text{Im}\left(z\right)'\right)^{2},
\end{align*}
it follows
\[
-w''-\frac{q_{1}c+2q_{1}'}{q_{1}}w'+\gamma w\leq0\text{ in }\mathbb{R}.
\]
This inequality implies the nonexistence of local maxima of $w$.
Since $w\in\mathscr{C}^{1}\left(\mathbb{R}\right)$, there exists
consequently $\xi_{0}\in\overline{\mathbb{R}}$ such that $w$ is
decreasing on $\left(-\infty,\xi_{0}\right)$ and increasing on $\left(\xi_{0},+\infty\right)$.
Therefore $w$ has well-defined limits at $\pm\infty$ and since $w\in\mathscr{L}^{\infty}\left(\mathbb{R}\right)$,
these limits are finite. By classical elliptic regularity and the
Harnack inequality (see Gilbarg\textendash Trudinger \cite{Gilbarg_Trudin})
applied to the equation satisfied by $q_{1}$, $\frac{q_{1}'}{q_{1}}$
is bounded in $\mathbb{R}$. By elliptic regularity again, applied
this time to the equation
\[
-w''-\frac{q_{1}c+2q_{1}'}{q_{1}}w'+\gamma w=-2\left(\text{Re}\left(z\right)'\right)^{2}-2\left(\text{Im}\left(z\right)'\right)^{2},
\]
the limits of $w$ have to be null, whence $w$ itself is null, and
then $q_{N}$ is null.
\item Inductive step: let $k\in\left[N-1\right]\backslash\left\{ 1\right\} $
and assume $q_{N+1-k}=0$. Defining 
\[
\lambda=j_{N-k,N-k}\in\text{sp}\mathbf{L}\backslash\left\{ \lambda_{PF}\left(\mathbf{L}\right)\right\} ,
\]
the equation satisfied by $q_{N+1-\left(k+1\right)}=q_{N-k}$ is
\[
-q_{N-k}''-cq_{N-k}'=\left(\lambda-b\left[\mathbf{P}^{-1}\mathbf{q}\right]\right)q_{N-k}\text{ in }\mathbb{R}.
\]
Repeating the argument detailed in the previous step shows similarly
that $q_{N-k}$ is null. 
\end{itemize}
Hence the proof by induction is ended and yields indeed $\mathbf{q}=q_{1}\mathbf{e}_{1}$
in $\mathbb{R}$. Now, back to the equation satisfied by $q_{1}$,
we find
\[
-q_{1}''-cq_{1}'=\left(\lambda_{PF}\left(\mathbf{L}\right)-b\left[q_{1}\mathbf{n}_{PF}\left(\mathbf{L}\right)\right]\right)q_{1}\text{ in }\mathbb{R},
\]
which implies in view of well-known results on the traveling wave
equation for the scalar KPP equation the existence of $\xi_{0}\in\mathbb{R}$
such that $q_{1}$ coincides with $\xi\mapsto p_{c}\left(\xi-\xi_{0}\right)$.
\end{proof}

\subsection{Global asymptotic stability}

The auxiliary functions used in the proof of \propref{Uniqueness_profiles}
can be used again to prove the global asymptotic stability of $\mathbf{v}^{\star}$
as stated in \thmref{Global_asymptotic_stability_(H6)}. In particular,
the following lemma will be used repeatedly.
\begin{lem}
\label{lem:N-1_component_vanish} There exists $\gamma>0$ such that
all bounded positive classical solutions $\mathbf{u}$ of $\left(E_{KPP}\right)$
set in $\left(0,+\infty\right)\times\mathbb{R}$ satisfying
\[
\inf\limits _{\left(t,x\right)\in\left(0,+\infty\right)\times\mathbb{R}}\mathbf{n}_{PF}\left(\mathbf{L}\right)^{T}\boldsymbol{\Pi}_{PF}\left(\mathbf{L}\right)\mathbf{u}\left(t,x\right)>0
\]
satisfy also
\[
\lim_{t\to+\infty}\left(\text{e}^{\gamma t}\sup_{x\in\mathbb{R}}\left|\left(\mathbf{I}-\boldsymbol{\Pi}_{PF}\left(\mathbf{L}\right)\right)\mathbf{u}\left(t,x\right)\right|\right)=0.
\]
\end{lem}

\begin{proof}
The proof is very similar to the first part of that of \propref{Uniqueness_profiles}.
Defining $\mathbf{v}=\mathbf{P}\mathbf{u}$, the equation satisfied
by $v_{1}$ is
\[
\partial_{t}v_{1}-\partial_{xx}v_{1}=\left(\lambda_{PF}\left(\mathbf{L}\right)-b\left[\mathbf{P}^{-1}\mathbf{v}\right]\right)v_{1}\text{ in }\left(0,+\infty\right)\times\mathbb{R}.
\]

For all $k\in\left[N\right]\backslash\left\{ 1\right\} $, there exists
$\gamma_{k}>0$ such that $v_{k}$ satisfies
\[
\left\{ \begin{matrix}\partial_{t}\left(\left|\frac{v_{k}}{v_{1}}\right|^{2}\right)-\partial_{xx}\left(\left|\frac{v_{k}}{v_{1}}\right|^{2}\right)-\frac{2\partial_{x}v_{1}}{v_{1}}\partial_{x}\left(\left|\frac{v_{k}}{v_{1}}\right|^{2}\right)+\gamma_{k}\left|\frac{v_{k}}{v_{1}}\right|^{2}\leq0 & \text{ in }\left(0,+\infty\right)\times\mathbb{R}\\
\left(\left|\frac{v_{k}}{v_{1}}\right|^{2}\right)_{|\left\{ 0\right\} \times\mathbb{R}}\in\mathscr{L}^{\infty}\left(\mathbb{R},[0,+\infty)\right),
\end{matrix}\right.
\]
that is such that $z_{k}:\left(t,x\right)\mapsto\text{e}^{\frac{\gamma_{k}}{2}t}\left|\frac{v_{k}}{v_{1}}\right|^{2}$
satisfies
\[
\left\{ \begin{matrix}\partial_{t}z_{k}-\partial_{xx}z_{k}-\frac{2\partial_{x}v_{1}}{v_{1}}\partial_{x}z_{k}+\frac{\gamma_{k}}{2}z_{k}\leq0 & \text{ in }\left(0,+\infty\right)\times\mathbb{R}\\
\left(z_{k}\right)_{|\left\{ 0\right\} \times\mathbb{R}}\in\mathscr{L}^{\infty}\left(\mathbb{R},[0,+\infty)\right).
\end{matrix}\right.
\]

Since $z_{k}$ stays bounded locally in time, by a classical argument
(detailed for instance in \cite[Proposition 3.4]{Girardin_2016_2}),
$z_{k}$ vanishes uniformly in space as $t\to+\infty$. Consequently,
\[
\text{e}^{\frac{\gamma_{k}}{4}t}\sup_{x\in\mathbb{R}}\left|v_{k}\right|\to0\text{ as }t\to+\infty.
\]

The conclusion follows from $\gamma=\min\limits _{k\in\left[N\right]}\frac{\gamma_{k}}{4}$
and the following obvious algebraic equality:
\[
\left(\mathbf{I}-\boldsymbol{\Pi}_{PF}\left(\mathbf{L}\right)\right)\mathbf{u}=\mathbf{P}^{-1}\left(\sum_{k=2}^{N}v_{k}\mathbf{e}_{k}\right).
\]
\end{proof}
We begin with the case of homogeneous initial data, which does not
require $\mathbf{d}=\mathbf{1}_{N,1}$ since $\left(E_{KPP}\right)$
reduces to $\left(E_{KPP}^{0}\right)$ in this context. 
\begin{prop}
\label{prop:Global_asymptotic_stability_for_E_KPP0} All positive
classical solutions of $\left(E_{KPP}^{0}\right)$ set in $\left(0,+\infty\right)$
converge as $t\to+\infty$ to $\mathbf{v}^{\star}$.
\end{prop}

\begin{proof}
Once again, the proof is very similar to that of \propref{Uniqueness_profiles}. 

Fix a positive classical solution $\mathbf{v}$ of $\left(E_{KPP}^{0}\right)$.
By \cite[Theorem 1.1]{Girardin_2016_2}, $\mathbf{v}\left(1\right)\gg\mathbf{0}$.
Hence the function $\mathbf{u}:t\mapsto\mathbf{v}\left(t+1\right)$
is a classical solution of $\left(E_{KPP}^{0}\right)$ set in $\left(0,+\infty\right)$
which is positive in $[0,+\infty)$ (whereas $\mathbf{v}\left(0\right)$
might have null components) and which converges to $\mathbf{v}^{\star}$
if and only if $\mathbf{v}$ converges to $\mathbf{v}^{\star}$.

The function $u=\mathbf{n}_{PF}\left(\mathbf{L}\right)^{T}\boldsymbol{\Pi}_{PF}\left(\mathbf{L}\right)\mathbf{u}$
satisfies
\[
u'=\lambda_{PF}\left(\mathbf{L}\right)u-b\left[\mathbf{u}\right]u.
\]

In order to apply \lemref{N-1_component_vanish}, it suffices to verify
\[
\inf_{t\in\left(0,+\infty\right)}u\left(t\right)>0.
\]
On one hand, since $\mathbf{u}$ is positive in $[0,+\infty)$, $u$
is positive in $[0,+\infty)$ as well. Hence any $t>0$ such that
$u'\left(t\right)=0$ is such that $b\left(\mathbf{u}\left(t\right)\right)=\lambda_{PF}\left(\mathbf{L}\right)$
and consequently any local minimum is larger than some positive constant.
On the other hand, $\liminf\limits _{t\to+\infty}u>0$ is a direct
consequence of the persistence result \cite[Theorem 1.3]{Girardin_2016_2}. 

Since $b$ is Lipschitz-continuous on the compact set $\left\{ \mathbf{v}\in\mathsf{K}\ |\ \mathbf{v}\leq\mathbf{k}\right\} $,
there exists $C_{1}>0$ such that
\[
\left|b\left[u\mathbf{n}_{PF}\left(\mathbf{L}\right)\right]-b\left[\mathbf{u}\right]\right|\leq C_{1}\left|\left(\mathbf{I}-\boldsymbol{\Pi}_{PF}\left(\mathbf{L}\right)\right)\mathbf{u}\right|\text{ in }[0,+\infty),
\]
Now $u$ satisfies
\[
u'=\lambda_{PF}\left(\mathbf{L}\right)u-b\left[u\mathbf{n}_{PF}\left(\mathbf{L}\right)\right]u+\left(b\left[u\mathbf{n}_{PF}\left(\mathbf{L}\right)\right]-b\left[\mathbf{u}\right]\right)u,
\]
 with, by \lemref{N-1_component_vanish},
\[
\left(b\left[u\mathbf{n}_{PF}\left(\mathbf{L}\right)\right]-b\left[\mathbf{u}\right]\right)u=o\left(u\right)\text{ as }t\to+\infty.
\]

It follows easily (see for instance \cite{Leman_Meleard_}) that $u$
converges to the unique constant $\alpha^{\star}>0$ such that $\lambda_{PF}\left(\mathbf{L}\right)=b\left[\alpha^{\star}\mathbf{n}_{PF}\left(\mathbf{L}\right)\right]$,
which precisely means
\[
\lim\limits _{t\to+\infty}\mathbf{u}\left(t\right)=\mathbf{v}^{\star}.
\]
\end{proof}
Finally, at the expense of assuming $\mathbf{d}=\mathbf{1}_{N,1}$,
we extend the previous result to non-homogeneous initial data. 
\begin{prop}
\label{prop:Local_global_asymptotic_stability} Assume $\mathbf{d}=\mathbf{1}_{N,1}$.
Then, for all bounded intervals $I\subset\mathbb{R}$, all bounded
positive classical solutions $\mathbf{u}$ of $\left(E_{KPP}\right)$
set in $\left(0,+\infty\right)\times\mathbb{R}$ satisfy
\[
\lim_{t\to+\infty}\sup_{x\in I}\left|\mathbf{u}\left(t,x\right)-\mathbf{v}^{\star}\right|=0.
\]

Consequently, if $\mathbf{d}=\mathbf{1}_{N,1}$, the set of bounded
nonnegative classical solutions of $\left(S_{KPP}\right)$ is exactly
$\left\{ \mathbf{0},\mathbf{v}^{\star}\right\} $.
\end{prop}

\begin{proof}
Let $\left(t_{n}\right)_{n\in\mathbb{N}}\in\left(0,+\infty\right)^{\mathbb{N}}$
such that $\lim\limits _{n\to+\infty}t_{n}=+\infty$. Then, by classical
parabolic estimates (Lieberman \cite{Lieberman_2005}) and a diagonal
extraction process, the sequence
\[
\left(\mathbf{u}_{n}\right)_{n\in\mathbb{N}}=\left(\left(t,x\right)\mapsto\mathbf{u}\left(t+t_{n},x\right)\right)_{n\in\mathbb{N}}
\]
 converges up to extraction to an entire classical solution of $\left(E_{KPP}\right)$
valued in $\prod\limits _{i=1}^{N}\left[\nu,g_{i}\left(0\right)\right]$
(see \cite[Theorems 1.2 and 1.3]{Girardin_2016_2}).

Now let us prove that $\mathbf{v}^{\star}$ is the unique bounded
entire classical solution $\tilde{\mathbf{u}}$ of $\left(E_{KPP}\right)$
satisfying
\[
\left(\inf_{\mathbb{R}^{2}}\tilde{u}_{i}\right)_{i\in\left[N\right]}\gg\mathbf{0}.
\]
 Let $\tilde{\mathbf{u}}$ be such a solution. The function $\tilde{u}=\mathbf{n}_{PF}\left(\mathbf{L}\right)^{T}\boldsymbol{\Pi}_{PF}\left(\mathbf{L}\right)\tilde{\mathbf{u}}$
satisfies
\[
\partial_{t}\tilde{u}-\partial_{xx}\tilde{u}=\lambda_{PF}\left(\mathbf{L}\right)\tilde{u}-b\left[\tilde{u}\mathbf{n}_{PF}\left(\mathbf{L}\right)\right]\tilde{u}+\left(b\left[\tilde{u}\mathbf{n}_{PF}\left(\mathbf{L}\right)\right]-b\left[\tilde{\mathbf{u}}\right]\right)\tilde{u}.
\]
 For all $\tau\in\mathbb{R}$, 
\[
\inf_{\left(t,x\right)\in\left(0,+\infty\right)\times\mathbb{R}}\tilde{u}\left(t+\tau,x\right)>0.
\]
By \lemref{N-1_component_vanish}, there exists $C>0$ such that,
for all $t>0$ and all $\tau\in\mathbb{R}$,
\[
\sup_{x\in\mathbb{R}}\left|\tilde{u}\left(t+\tau,x\right)\mathbf{n}_{PF}\left(\mathbf{L}\right)-\tilde{\mathbf{u}}\left(t+\tau,x\right)\right|\leq C\text{e}^{-\gamma t}.
\]
It follows that for all $t>0$,
\[
\sup_{\left(t',x\right)\in\mathbb{R}^{2}}\left|\tilde{u}\left(t',x\right)\mathbf{n}_{PF}\left(\mathbf{L}\right)-\tilde{\mathbf{u}}\left(t',x\right)\right|\leq C\text{e}^{-\gamma t}
\]
and then passing the right-hand side to the limit $t\to+\infty$,
we find
\[
\tilde{u}\left(t',x\right)\mathbf{n}_{PF}\left(\mathbf{L}\right)=\tilde{\mathbf{u}}\left(t',x\right)\text{ for all }\left(t',x\right)\in\mathbb{R}^{2}.
\]
 Consequently, $\tilde{u}$ satisfies
\[
\partial_{t}\tilde{u}-\partial_{xx}\tilde{u}=\lambda_{PF}\left(\mathbf{L}\right)\tilde{u}-b\left[\tilde{u}\mathbf{n}_{PF}\left(\mathbf{L}\right)\right]\tilde{u}.
\]
By standard results on the scalar KPP equation, $\tilde{u}=\alpha^{\star}$
in $\mathbb{R}^{2}$, that is $\tilde{\mathbf{u}}=\mathbf{v}^{\star}$.

A standard compactness argument ends the proof. 
\end{proof}

\section{The back of the fronts: vanishingly small mutations in monostable
two-component systems}

In this section, we assume $\left(H_{7}\right)$ and recall the existence
and uniqueness of $\left(\mathbf{r},\eta,\mathbf{m}\right)\in\mathsf{K}^{++}\times\left(0,+\infty\right)\times\mathsf{S}^{++}\left(\mathbf{0},1\right)$
such that
\[
\mathbf{L}=\mathbf{R}+\eta\left(\begin{matrix}-1 & 1\\
1 & -1
\end{matrix}\right)\mathbf{M}\text{ with }\left(\mathbf{R},\mathbf{M}\right)=\left(\text{diag}\left(\mathbf{r}\right),\text{diag}\left(\mathbf{m}\right)\right).
\]
The various objects and notations of the problem now depend \textit{a
priori} on $\eta$ and a subscript $_{\eta}$ might be added accordingly.
The following definitions are recalled: 
\[
\alpha_{i}=\frac{r_{i}}{c_{i,i}}\text{ for all }i\in\left\{ 1,2\right\} ,
\]
\[
\mathbf{v}_{m}=\frac{1}{\det\mathbf{C}}\left(\begin{matrix}r_{1}c_{2,2}-r_{2}c_{1,2}\\
r_{2}c_{1,1}-r_{1}c_{2,1}
\end{matrix}\right)\text{ if }\det\mathbf{C}\neq0,
\]
\[
\mathscr{E}=\left\{ \left(\eta,\mathbf{p},c\right)\in\left(0,+\infty\right)^{2}\times\mathscr{C}^{2}\left(\mathbb{R},\mathbb{R}^{2}\right)\ |\ \mathbf{p}\in\mathscr{P}_{c,\eta},\ c\geq c_{\eta}^{\star}\right\} ,
\]
\[
\partial_{t}\mathbf{u}-\mathbf{D}\partial_{xx}\mathbf{u}=\mathbf{R}\mathbf{u}-\left(\mathbf{C}\mathbf{u}\right)\circ\mathbf{u}.\quad\left(E_{KPP}\right)_{0}
\]

\subsection{Preparatory lemmas}

The proof of \thmref{(H7)_Monostable} will use the following lemmas
which are of independent interest.
\begin{lem}
\label{lem:(H7)_Uniform_L_infty_estimates} Let $i\in\left\{ 1,2\right\} $,
$j=3-i$ and
\[
\eta\in\left(0,\frac{r_{i}c_{i,j}}{m_{j}c_{i,i}}\right].
\]

Then for all traveling wave solutions $\left(\mathbf{p},c\right)$
of $\left(E_{KPP}\right)_{\eta}$, 
\[
p_{i}\leq\alpha_{i}\text{ in }\mathbb{R}.
\]
\end{lem}

\begin{rem*}
This lemma is straightforwardly generalizable to the case $N>2$.
\end{rem*}
\begin{proof}
Having in mind the proof of \cite[Theorem 1.5, ii)]{Girardin_2016_2},
it suffices to investigate the sign of
\[
r_{i}p_{i}-\eta m_{i}p_{i}+\eta m_{j}p_{j}-\left(c_{i,i}p_{i}+c_{i,j}p_{j}\right)p_{i}=p_{i}\left(r_{i}-\eta m_{i}-c_{i,i}p_{i}\right)+p_{j}\left(\eta m_{j}-c_{i,j}p_{i}\right).
\]
This quantity is nonpositive provided
\[
p_{i}\geq\max\left(\frac{r_{i}-\eta m_{i}}{c_{i,i}},\frac{\eta m_{j}}{c_{i,j}}\right).
\]
Since
\[
\frac{r_{i}}{c_{i,i}}\geq\frac{r_{i}-\eta m_{i}}{c_{i,i}}\text{ for all }\eta\geq0,
\]
\[
\frac{r_{i}}{c_{i,i}}\geq\frac{\eta m_{j}}{c_{i,j}}\text{ for all }\eta\leq\frac{r_{i}c_{i,j}}{m_{j}c_{i,i}},
\]
 we deduce indeed $p_{i}\leq\frac{r_{i}}{c_{i,i}}$.
\end{proof}
{} 
\begin{lem}
\label{lem:(H7)_Monostable_Uniform_estimates_from_below} Let $i\in\left\{ 1,2\right\} $,
$j=3-i$ and assume
\[
\frac{r_{i}}{r_{j}}>\frac{c_{i,j}}{c_{j,j}}.
\]

Let
\[
\overline{\eta}_{i}=\frac{1}{2}\min\left(\frac{r_{j}c_{j,i}}{m_{i}c_{j,j}},\frac{r_{j}}{m_{i}}\left(\frac{r_{i}}{r_{j}}-\frac{c_{i,j}}{c_{j,j}}\right)\right),
\]
\[
\rho_{i}=\frac{1}{2}\frac{r_{j}}{c_{i,i}}\left(\frac{r_{i}}{r_{j}}-\frac{c_{i,j}}{c_{j,j}}\right).
\]

Then for all $\rho\in\left(0,\rho_{i}\right]$, all $\eta\in\left(0,\overline{\eta}_{i}\right)$
and all traveling wave solutions $\left(\mathbf{p},c\right)$ of $\left(E_{KPP}\right)_{\eta}$,
there exists a unique 
\[
\xi_{\rho}\in p_{i}^{-1}\left(\left\{ \rho\right\} \right).
\]

Furthermore $p_{i}$ is decreasing in $\left(\xi_{\rho},+\infty\right)$
and $p_{i}-\rho$ is positive in $\left(-\infty,\xi_{\rho}\right)$. 
\end{lem}

\begin{rem*}
The following proof is mostly due to Griette\textendash Raoul \cite[Proposition 5.1]{Griette_Raoul}.
\end{rem*}
\begin{proof}
Let $\zeta\in\mathbb{R}$ such that $p_{i}\left(\zeta\right)$ is
a local minimum of $p_{i}$. Then
\[
r_{i}p_{i}\left(\zeta\right)-\eta m_{i}p_{i}\left(\zeta\right)+\eta m_{j}p_{j}\left(\zeta\right)-\left(c_{i,i}p_{i}\left(\zeta\right)+c_{i,j}p_{j}\left(\zeta\right)\right)p_{i}\left(\zeta\right)\leq0.
\]

This implies
\[
r_{i}p_{i}\left(\zeta\right)-\eta m_{i}p_{i}\left(\zeta\right)-\left(c_{i,i}p_{i}\left(\zeta\right)+c_{i,j}p_{j}\left(\zeta\right)\right)p_{i}\left(\zeta\right)<0,
\]
whence
\[
r_{i}-\eta m_{i}<c_{i,i}p_{i}\left(\zeta\right)+c_{i,j}p_{j}\left(\zeta\right),
\]
whence by \lemref{(H7)_Uniform_L_infty_estimates}
\[
r_{i}-\eta m_{i}<c_{i,i}p_{i}\left(\zeta\right)+c_{i,j}\frac{r_{j}}{c_{j,j}},
\]
and then
\begin{align*}
p_{i}\left(\zeta\right) & >\frac{1}{c_{i,i}}\left(r_{i}-\frac{r_{j}c_{i,j}}{c_{j,j}}\right)-\frac{\eta m_{i}}{c_{i,i}}\\
 & >\frac{r_{j}}{c_{i,i}}\left(\frac{r_{i}}{r_{j}}-\frac{c_{i,j}}{c_{j,j}}\right)-\frac{\overline{\eta}_{i}m_{i}}{c_{i,i}}\\
 & \geq\frac{1}{2}\frac{r_{j}}{c_{i,i}}\left(\frac{r_{i}}{r_{j}}-\frac{c_{i,j}}{c_{j,j}}\right)\\
 & =\rho_{i}.
\end{align*}

Now let $\rho\in\left(0,\rho_{i}\right]$ and $\xi_{\rho}\in p_{i}^{-1}\left(\left\{ \rho\right\} \right)$.

Since $p_{i}\left(\xi_{\rho}\right)$ cannot be a local minimum, there
exists a neighborhood of $\xi_{\rho}$ in which $p_{i}$ is strictly
monotonic. Assume it is increasing. Then by continuity of $p_{i}'$
and the previous estimate on local minima, $p_{i}$ is increasing
in $\left(-\infty,\xi_{\rho}\right)$. By classical elliptic regularity,
$\mathbf{p}$ converges as $\xi\to-\infty$ to a solution of $\mathbf{L}\mathbf{v}=\mathbf{C}\mathbf{v}\circ\mathbf{v}$,
and by \cite[Theorem 1.5, iii)]{Girardin_2016_2}, this solution is
positive. But in view of the preceding estimates, necessarily
\[
\lim_{\xi\to-\infty}p_{i}\left(\xi\right)>\rho_{i}\geq p_{i}\left(\xi_{\rho}\right),
\]
which contradicts the monotonicity of $p_{i}$ in $\left(-\infty,\xi_{\rho}\right)$.
Hence $p_{i}$ is decreasing in a neighborhood of $\xi_{\rho}$ and
then in $\left(\xi_{\rho},+\infty\right)$. Consequently, 
\[
p_{i}^{-1}\left(\left\{ \rho\right\} \right)=\left\{ \xi_{\rho}\right\} .
\]
 This holds for all $\rho\in\left(0,\rho_{i}\right]$ and therefore
ends the proof. 
\end{proof}

\subsection{Convergence at the back}

Let $i\in\left\{ 1,2\right\} $, $j=3-i$, $\left(c_{\eta}\right)_{\eta\geq0}$
and $\left(\mathbf{p}_{\eta}\right)_{\eta>0}$ such that
\[
\left\{ \begin{matrix}\left(\eta,\mathbf{p}_{\eta},c_{\eta}\right)\in\mathscr{E} & \text{for all }\eta>0,\\
c_{0}=\lim\limits _{\eta\to0}c_{\eta},
\end{matrix}\right.
\]
 and assume from now on that 
\[
\frac{r_{i}}{r_{j}}>\frac{c_{i,j}}{c_{j,j}}
\]
so that the assumptions of \thmref{(H7)_Monostable} are satisfied.
Define subsequently
\[
\mathbf{v}_{s}=\left\{ \begin{matrix}\alpha_{i}\mathbf{e}_{i} & \text{if }\frac{r_{i}}{r_{j}}\geq\frac{c_{i,i}}{c_{j,i}},\\
\mathbf{v}_{m} & \text{if }\frac{r_{i}}{r_{j}}<\frac{c_{i,i}}{c_{j,i}}.
\end{matrix}\right.
\]
\begin{prop}
\label{prop:(H7)_Monostable_Back} There exists $\left(\zeta_{\eta}\right)_{\eta>0}$
such that, as $\eta\to0$, $\left(\xi\mapsto\mathbf{p}_{\eta}\left(\xi+\zeta_{\eta}\right),c_{\eta}\right)_{\eta>0}$
converges up to extraction in $\mathscr{C}_{loc}^{2}\left(\mathbb{R},\mathbb{R}^{2}\right)\times\mathbb{R}$
to a traveling wave solution $\left(\mathbf{p}_{back},c_{0}\right)$
of $\left(E_{KPP}\right)_{0}$ achieving one of the following connections:
\begin{enumerate}
\item $\mathbf{0}$ to $\mathbf{v}_{s}$,
\item $\alpha_{j}\mathbf{e}_{j}$ to $\mathbf{v}_{s}$,
\item $\mathbf{0}$ to $\alpha_{i}\mathbf{e}_{i}$ with $\mathbf{p}$ semi-extinct.
\end{enumerate}
\end{prop}

\begin{proof}
Let $\rho=\min\left(\rho_{i},v_{s,i}\right)$. By virtue of \lemref{(H7)_Monostable_Uniform_estimates_from_below},
for all $\eta>0$, there exists a unique $\zeta_{\eta}$ such that:
\begin{itemize}
\item $p_{\eta,i}$ is decreasing in $\left(\zeta_{\eta},+\infty\right)$, 
\item $p_{\eta,i}\left(\zeta_{\eta}\right)=\rho$,
\item $p_{\eta,i}-\rho$ is positive in $\left(-\infty,\zeta_{\eta}\right)$.
\end{itemize}
By \lemref{(H7)_Uniform_L_infty_estimates}, classical elliptic estimates
(Gilbarg\textendash Trudinger \cite{Gilbarg_Trudin}) and a diagonal
extraction process, $\left(\xi\mapsto\mathbf{p}_{\eta}\left(\xi+\xi_{\eta}\right)\right)_{\eta>0}$
converges in $\mathscr{C}_{loc}^{2}$ up to extraction. Let $\mathbf{p}$
be its limit. We have directly $\mathbf{0}\leq\mathbf{p}\leq\boldsymbol{\alpha}$
in $\mathbb{R}$. In view of the normalization, we also have:
\begin{itemize}
\item $p_{i}$ is nonincreasing in $\left(0,+\infty\right)$, 
\item $p_{i}\left(0\right)=\rho$,
\item $p_{i}-\rho$ is nonnegative in $\left(-\infty,0\right)$.
\end{itemize}
Let $\left(\xi_{n}\right)_{n\in\mathbb{N}}$ such that $\xi_{n}\to-\infty$
as $n\to+\infty$. Defining
\[
\hat{\mathbf{p}}_{n}:\xi\mapsto\mathbf{p}\left(\xi+\xi_{n}\right)\text{ for all }n\in\mathbb{N},
\]
 by classical elliptic estimates and a diagonal extraction process
again, $\left(\hat{\mathbf{p}}_{n}\right)_{n\in\mathbb{N}}$ converges
up to extraction in $\mathscr{C}_{loc}^{2}$ to a function $\hat{\mathbf{p}}$
satisfying
\[
-\mathbf{D}\hat{\mathbf{p}}''-c\hat{\mathbf{p}}'=\mathbf{R}\hat{\mathbf{p}}-\left(\mathbf{C}\hat{\mathbf{p}}\right)\circ\hat{\mathbf{p}}
\]
and such that
\[
\left(\rho,0\right)\leq\left(\hat{p}_{i},\hat{p}_{j}\right)\leq\left(\alpha_{i},\alpha_{j}\right).
\]
In particular, $\hat{\mathbf{p}}$ is a stationary solution of
\[
\left\{ \begin{matrix}\partial_{t}\mathbf{u}-\partial_{xx}\mathbf{u}-c_{0}\partial_{x}\mathbf{u}=\mathbf{R}\mathbf{u}-\left(\mathbf{C}\mathbf{u}\right)\circ\mathbf{u} & \text{in }\left(0,+\infty\right)\times\mathbb{R}\\
\mathbf{u}\left(0,x\right)=\hat{\mathbf{p}}\left(x\right) & \text{for all }x\in\mathbb{R}.
\end{matrix}\right.
\]
Applying the comparison principle for two-components competitive parabolic
systems to $\hat{\mathbf{p}}$ and to the solution of
\[
\left\{ \begin{matrix}\partial_{t}\mathbf{u}-\partial_{xx}\mathbf{u}-c_{0}\partial_{x}\mathbf{u}=\mathbf{R}\mathbf{u}-\left(\mathbf{C}\mathbf{u}\right)\circ\mathbf{u} & \text{in }\left(0,+\infty\right)\times\mathbb{R}\\
\left(u_{i},u_{j}\right)\left(0,x\right)=\left(\rho,\sup\hat{p}_{j}\right) & \text{for all }x\in\mathbb{R},
\end{matrix}\right.
\]
 which is homogeneous in space and is therefore the solution of
\[
\left\{ \begin{matrix}\partial_{t}\mathbf{u}=\mathbf{R}\mathbf{u}-\left(\mathbf{C}\mathbf{u}\right)\circ\mathbf{u} & \text{in }\left(0,+\infty\right)\times\mathbb{R}\\
\left(u_{i},u_{j}\right)\left(0,x\right)=\left(\rho,\sup\hat{p}_{j}\right) & \text{for all }x\in\mathbb{R},
\end{matrix}\right.
\]
we directly obtain $\hat{\mathbf{p}}=\mathbf{v}_{s}$ if $\sup\hat{p}_{j}>0$
and $\hat{\mathbf{p}}=\alpha_{i}\mathbf{e}_{i}$ if $\sup\hat{p}_{j}=0$.
In other words, if $\mathbf{v}_{s}=\alpha_{i}\mathbf{e}_{i}$, $\hat{\mathbf{p}}=\alpha_{i}\mathbf{e}_{i}$,
and if $\mathbf{v}_{s}=\mathbf{v}_{m}$, $\hat{\mathbf{p}}\in\left\{ \mathbf{v}_{s},\alpha_{i}\mathbf{e}_{i}\right\} $.
Since $\mathbf{v}_{s}$ and $\alpha_{i}\mathbf{e}_{i}$ are isolated
steady states and $\mathbf{p}$ is continuous, the last diagonal extraction
was not necessary and $\left(\hat{\mathbf{p}}_{n}\right)_{n\in\mathbb{N}}$
converges indeed to $\hat{\mathbf{p}}$, that is
\[
\lim_{-\infty}\mathbf{p}\in\left\{ \mathbf{v}_{s},\alpha_{i}\mathbf{e}_{i}\right\} .
\]

Since $p_{i}$ is nonincreasing in $\left(0,+\infty\right)$, it converges
as $\xi\to+\infty$. By classical elliptic regularity, 
\[
\lim_{+\infty}\left(-d_{i}p_{i}''-c_{0}p_{i}'\right)=0,
\]
whence either 
\[
\lim_{+\infty}p_{i}=0
\]
or $p_{j}$ converges as well, its limit being 
\[
\lim_{+\infty}p_{j}=\frac{1}{c_{i,j}}\left(r_{i}-c_{i,i}\lim_{+\infty}p_{i}\right).
\]
In the second case, using $-d_{j}p_{j}''-c_{0}p_{j}'\to0$, $p_{i}\left(0\right)=\rho$
and the monotonicity of $p_{i}$ in $\left(0,+\infty\right)$, we
find $\lim\limits _{+\infty}\mathbf{p}\in\left\{ \alpha_{j}\mathbf{e}_{j},\mathbf{0}\right\} $,
which contradicts directly $\lim\limits _{+\infty}p_{i}>0$. Hence
$p_{i}$ converges to $0$. 

Subsequently, since $p_{j}$ is positive, every local minimum of $p_{j}$
satisfies 
\[
r_{j}\leq c_{j,j}p_{j}\left(\xi\right)+c_{j,i}p_{i}\left(\xi\right),
\]
which proves that for all sequences $\left(\xi_{n}\right)_{n\in\mathbb{N}}$
such that $\xi_{n}\to+\infty$ and $p_{j}\left(\xi_{n}\right)$ is
a local minimum of $p_{j}$, $p_{j}\left(\xi_{n}\right)$ converges
to $\alpha_{j}$. But then, by $\mathscr{C}^{1}$ regularity, either
$p_{j}$ is monotonic in a neighborhood of $+\infty$ or there exists
a sequence $\left(\xi_{n}\right)_{n\in\mathbb{N}}$ such that $\xi_{n}\to+\infty$,
$p_{j}\left(\xi_{n}\right)$ is a local minimum of $p_{j}$ and $\left(p_{j}\left(\xi_{n}\right)\right)_{n\in\mathbb{N}}$
converges to $\liminf\limits _{+\infty}p_{j}$. It turns out that
in both cases $p_{j}$ converges, the possible limits being $0$ and
$\alpha_{j}$. 

Therefore $\mathbf{p}$ is a traveling wave achieving exactly one
of the following connections:
\begin{enumerate}
\item $\mathbf{0}$ to $\mathbf{v}_{s}$, 
\item $\alpha_{j}\mathbf{e}_{j}$ to $\mathbf{v}_{s}$, 
\item $\mathbf{0}$ to $\alpha_{i}\mathbf{e}_{i}$ with $\alpha_{i}\mathbf{e}_{i}\neq\mathbf{v}_{s}$, 
\item $\alpha_{j}\mathbf{e}_{j}$ to $\alpha_{i}\mathbf{e}_{i}$ with $\alpha_{i}\mathbf{e}_{i}\neq\mathbf{v}_{s}$. 
\end{enumerate}
It remains to show that the third case is semi-extinct and the fourth
case is impossible. We will actually prove both statements simultaneously
by proving that $\lim\limits _{-\infty}\mathbf{p}=\alpha_{i}\mathbf{e}_{i}\neq\mathbf{v}_{s}$
implies $p_{j}=0$ in $\mathbb{R}$.

Assume $\lim\limits _{-\infty}\mathbf{p}=\alpha_{i}\mathbf{e}_{i}$
and $\mathbf{v}_{s}=\mathbf{v}_{m}$. Assume also by contradiction
that $p_{j}$ is positive in $\mathbb{R}$. 

Multiplying the equation 
\[
-d_{j}p_{j}''-c_{0}p_{j}'=\left(r_{j}-c_{j,j}p_{j}-c_{j,i}p_{i}\right)p_{j},
\]
by the function
\[
\varphi:\xi\mapsto\text{e}^{\frac{c_{0}}{d_{j}}\xi},
\]
 we find
\[
-d_{j}\left(\varphi p_{j}'\right)'=\left(r_{j}-c_{j,j}p_{j}-c_{j,i}p_{i}\right)\varphi p_{j}.
\]
Recall that $\mathbf{v}_{s}=\mathbf{v}_{m}$ implies $\frac{r_{i}}{r_{j}}<\frac{c_{i,i}}{c_{j,i}}$,
that is $r_{j}-c_{j,i}\alpha_{i}>0$. Therefore the quantity
\[
\overline{\xi}=\sup\left\{ \xi\in\mathbb{R}\ |\ \forall\zeta\in\left(-\infty,\xi\right)\quad r_{j}-c_{j,j}p_{j}\left(\zeta\right)-c_{j,i}p_{i}\left(\zeta\right)>0\right\} 
\]
 is well-defined in $\mathbb{R}\cup\left\{ +\infty\right\} $. In
$\left(-\infty,\overline{\xi}\right)$, $\varphi p_{j}'$ is decreasing.
Since on one hand $\lim\limits _{-\infty}\varphi=0$ and on the other
hand $\lim\limits _{-\infty}p_{j}'=0$ by classical elliptic regularity,
the limit of $\varphi p_{j}'$ itself is $0$. Consequently, $\varphi p_{j}'$
is negative in $\left(-\infty,\overline{\xi}\right)$. It follows
that $p_{j}$ itself is decreasing in $\left(-\infty,\overline{\xi}\right)$.
But then $\lim\limits _{-\infty}p_{j}=0$ implies that $p_{j}$ is
negative in $\left(-\infty,\overline{\xi}\right)$, which obviously
contradicts the positivity of $p_{j}$. This ends the proof.
\end{proof}

\section{Discussion}

\subsection{Why is it likely hopeless to search for a general result on the behavior
at the back of the front?\label{subsec:Delicate_back}}

First of all, the linearization of $\left(S_{KPP}\right)$ at $\mathbf{0}$
being cooperative, it is natural to wonder whether the dynamics of
$\left(E_{KPP}\right)$ near some constant positive solution $\mathbf{u}$
of $\left(S_{KPP}\right)$ might be purely competitive or cooperative.
In general, neither is the case. The linearized reaction term at any
constant solution $\mathbf{u}$ of $\left(S_{KPP}\right)$ is

\[
\mathbf{L}_{\mathbf{u}}=\mathbf{L}-\text{diag}\left(\mathbf{c}\left(\mathbf{u}\right)\right)-\left(\mathbf{u}\mathbf{1}_{1,N}\right)\circ D\mathbf{c}\left(\mathbf{u}\right).
\]
 In the Lotka\textendash Volterra case where there exists $\mathbf{C}\gg\mathbf{0}$
such that $\mathbf{c}\left(\mathbf{v}\right)=\mathbf{C}\mathbf{v}$,
it reads
\[
\mathbf{L}_{\mathbf{u}}=\mathbf{L}-\text{diag}\left(\mathbf{C}\mathbf{u}\right)-\left(\mathbf{u}\mathbf{1}_{1,N}\right)\circ\mathbf{C}.
\]
 On one hand, it is clear that if there exists $\left(i,j\right)\in\left[N\right]^{2}$
such that $l_{i,j}=0$, then $l_{\mathbf{u},i,j}<0$. On the other
hand, assuming that there exists $i\in\left[N\right]$ such that $l_{i,i}\leq0$,
we find
\[
-l_{i,i}u_{i}+u_{i}c_{i,i}u_{i}>0.
\]
 Since $\mathbf{L}\mathbf{u}-\left(\mathbf{C}\mathbf{u}\right)\circ\mathbf{u}=\mathbf{0}$,
it follows
\[
\sum_{j\in\left[N\right]\backslash\left\{ i\right\} }\left(l_{i,j}u_{j}-u_{i}c_{i,j}u_{j}\right)>0,
\]
 whence there exists $j\in\left[N\right]\backslash\left\{ i\right\} $
such that $l_{i,j}u_{j}-u_{i}c_{i,j}u_{j}>0$, that is such that
\[
l_{\mathbf{u},i,j}=l_{i,j}-u_{i}c_{i,j}>0.
\]
 Hence the competitive dynamics and the cooperative dynamics are indeed
intertwined near $\mathbf{u}$.

Next, in view of the literature on non-cooperative KPP systems, it
could be tempting to conjecture the uniqueness and the local stability
of the constant positive solution of $\left(S_{KPP}\right)$ (see
for instance Dockery\textendash Hutson\textendash Mischaikow\textendash Pernarowski
\cite{Dockery_1998} or Morris\textendash Börger\textendash Crooks
\cite{Morris_Borger_Crooks}). However, if $\mathbf{c}$ is linear
as before and if
\[
\left(N,\mathbf{L},\mathbf{C}\right)=\left(2,\mathbf{I}_{2}+\frac{1}{5}\left(\begin{matrix}-1 & 1\\
1 & -1
\end{matrix}\right),\frac{1}{10}\left(\begin{matrix}1 & 9\\
9 & 1
\end{matrix}\right)\right),
\]
then this property fails. Indeed, straightforward computations show
that the set of constant positive solutions of $\left(S_{KPP}\right)$
is
\[
\left\{ \left(\begin{matrix}3-\sqrt{\frac{15}{2}}\\
3+\sqrt{\frac{15}{2}}
\end{matrix}\right),\mathbf{1}_{2,1},\left(\begin{matrix}3+\sqrt{\frac{15}{2}}\\
3-\sqrt{\frac{15}{2}}
\end{matrix}\right)\right\} .
\]
 From the associated linearizations, it is easily found that, with
respect to $\left(E_{KPP}^{0}\right)$, the symmetric solution $\mathbf{1}_{2,1}$
is a saddle point whereas the other two solutions are stable nodes. 

Last, we also point out that if $\mathbf{d}=\mathbf{1}_{2,1}$ then
the preceding counter-example admits a family of traveling waves connecting
$\mathbf{0}$ to the saddle point $\mathbf{1}_{2,1}$. Indeed, looking
for profiles $\mathbf{p}$ of the form $\xi\mapsto p\left(\xi\right)\mathbf{1}_{2,1}$,
$\left(TW\left[c\right]\right)$ reduces to
\[
-p''-cp'=p-p^{2},
\]
which, by virtue of well-known results on the scalar KPP equation,
admits solutions connecting $0$ to $1$ if and only if $c\geq2$.
Hence we cannot hope to prove that all traveling waves connect $\mathbf{0}$
to a stable steady state. 

\subsection{What about the general separated competition case, with $\mathbf{d}$
and $\mathbf{a}$ possibly different from $\mathbf{1}_{N,1}$?}

The general case might be more subtle than expected, even regarding
the ODE system $\left(E_{KPP}^{0}\right)$: although the linearization
at $\mathbf{v}^{\star}$,
\[
\mathbf{L}_{\mathbf{v}^{\star}}=\mathbf{L}-\lambda_{\mathbf{a}}\mathbf{A}-\mathbf{A}\mathbf{v}^{\star}\left(\nabla b\left(\mathbf{v}^{\star}\right)^{T}\right),
\]
seems to be adequately described as a matrix of the form $-\mathbf{P}-\mathbf{Q}$
with $\mathbf{P}=\lambda_{\mathbf{a}}\mathbf{A}-\mathbf{L}$ a singular
M-matrix and $\mathbf{Q}=\mathbf{A}\mathbf{v}^{\star}\left(\nabla b\left(\mathbf{v}^{\star}\right)^{T}\right)$
a positive rank-one matrix, a recent paper by Bierkens and Ran \cite{Bierkens_Ran}
highlights thanks to a counter-example that such matrices can have
eigenvalues with positive real part (and there is in addition a counter-example
with irreducible $-\mathbf{P}$, so that irreducibility is not a sufficient
condition to ensure all eigenvalues are negative). Therefore it is
unclear whether $\mathbf{v}^{\star}$ is always locally asymptotically
stable with respect to $\left(E_{KPP}^{0}\right)$. Actually, the
main purpose of the study of Bierkens and Ran is to establish several
conditions sufficient to guarantee that all eigenvalues have a negative
real part (conditions among which we find $N=2$ and, of course, $\mathbf{a}=\mathbf{1}_{N,1}$).

In the case $N=2$, classical calculations show that the system $\left(E_{KPP}\right)$
is not subjected to Turing instabilities with respect to periodic
perturbations. Therefore it might be fruitful to investigate more
thoroughly the two-component system. Nevertheless, to this day we
do not have any further result.

\subsection{Where does \conjref{(H_7)} come from?}

Let us bring forth some insight into the limiting problem. What are
the spreading properties of $\left(E_{KPP}\right)_{0}$ with respect
to front-like initial data? What are the propagating solutions of
$\left(E_{KPP}\right)_{0}$ invading the null state?

Concerning the bistable case, we have at our disposal a recent result
by Carrère \cite{Carrere_2017} which can be summed up as follows.
Consider the Cauchy problem where $\left(-\infty,0\right)$ is initially
inhabited mostly but not only (in a sense made rigorous by Carrère)
by $u_{1}$ and $\left(0,+\infty\right)$ is completely uninhabited.
Let $c_{\alpha_{1}\mathbf{e}_{1}\to\alpha_{2}\mathbf{e}_{2}}$ be
the speed of the bistable front equal to $\alpha_{1}\mathbf{e}_{1}$
at $-\infty$ and to $\alpha_{2}\mathbf{e}_{2}$ at $+\infty$, as
given by Kan-On \cite{Kan_on_1995} and Gardner \cite{Gardner_1982}.
Recall that the following bounds hold true:
\[
-2\sqrt{d_{2}r_{2}}<c_{\alpha_{1}\mathbf{e}_{1}\to\alpha_{2}\mathbf{e}_{2}}<2\sqrt{d_{1}r_{1}}.
\]

Carrère\textquoteright s theorem is then:
\begin{enumerate}
\item if $2\sqrt{d_{1}r_{1}}>2\sqrt{d_{2}r_{2}}$ and $c_{\alpha_{1}\mathbf{e}_{1}\to\alpha_{2}\mathbf{e}_{2}}>0$,
then asymptotically in time, $u_{2}$ is extinct and $u_{1}$ spreads
at speed $2\sqrt{d_{1}r_{1}}$;
\item if $2\sqrt{d_{1}r_{1}}<2\sqrt{d_{2}r_{2}}$ and $c_{\alpha_{1}\mathbf{e}_{1}\to\alpha_{2}\mathbf{e}_{2}}>0$,
then asymptotically in time, $u_{2}$ spreads on the right at speed
$2\sqrt{d_{2}r_{2}}$ but is then replaced by $u_{1}$ at speed $c_{\alpha_{1}\mathbf{e}_{1}\to\alpha_{2}\mathbf{e}_{2}}$;
\item if $2\sqrt{d_{1}r_{1}}<2\sqrt{d_{2}r_{2}}$ and $c_{\alpha_{1}\mathbf{e}_{1}\to\alpha_{2}\mathbf{e}_{2}}<0$,
then asymptotically in time, $u_{2}$ chases $u_{1}$ on the left
at speed $c_{\alpha_{1}\mathbf{e}_{1}\to\alpha_{2}\mathbf{e}_{2}}$
and spreads on the right at speed $2\sqrt{d_{2}r_{2}}$.
\end{enumerate}
This result was long-awaited but, as far as we know, Carrère\textquoteright s
proof is the first one. 

Up to the sign of $c_{\alpha_{1}\mathbf{e}_{1}\to\alpha_{2}\mathbf{e}_{2}}$,
the second and the third cases above are identical. Recall that the
sign of $c_{\alpha_{1}\mathbf{e}_{1}\to\alpha_{2}\mathbf{e}_{2}}$
is in general a tough problem, although recently some particular cases
have been successfully solved (strong competition in Girardin\textendash Nadin
\cite{Girardin_Nadin_2015}, special choices of parameter values in
Guo\textendash Lin \cite{Guo_Lin_2013}, perturbation of the standing
wave in Risler \cite{Risler_2017}). 

A natural conjecture in view of Carrère\textquoteright s result is
the long-time convergence, in the first case, to a traveling wave
connecting $\mathbf{0}$ to $\alpha_{1}\mathbf{e}_{1}$ at speed $2\sqrt{d_{1}r_{1}}$
and with a semi-extinct profile $\mathbf{p}=p\mathbf{e}_{1}$. However,
in the second and third cases, a more complex limit seems to arise. 

The entire solutions connecting three or more stationary states with
decreasingly ordered speeds were first described in the scalar setting
by Fife and McLeod \cite{Fife_McLeod_19} and are referred to as \textit{propagating
terraces}, or simply \textit{terraces,} since the work of Ducrot,
Giletti and Matano \cite{Ducrot_Giletti_Matano}. A terrace with $n-1$
intermediate states is defined as a finite family of traveling waves
$\left(\left(\mathbf{p}_{i},c_{i}\right)\right)_{i\in\left[n\right]}$
such that $\mathbf{p}_{i}\left(-\infty\right)=\mathbf{p}_{i+1}\left(+\infty\right)$
for all $i\in\left[n-1\right]$ and such that $\left(c_{i}\right)_{i\in\left[n\right]}$
is decreasing. Provided the uniqueness (up to translation of the profile)
of the traveling wave connecting $\mathbf{v}_{i}=\mathbf{p}_{i}\left(+\infty\right)$
to $\mathbf{v}_{i+1}=\mathbf{p}_{i}\left(-\infty\right)$ at speed
$c_{i}$, the terrace is equivalently defined as the family $\left(\left(\mathbf{v}_{i},c_{i}\right)_{i\in\left[n\right]},\mathbf{v}_{n+1}\right)$.
However, in general, this family only defines a family of  terraces
that will be denoted hereafter $\mathscr{T}\left(\left(\mathbf{v}_{i},c_{i}\right)_{i\in\left[n\right]},\mathbf{v}_{n+1}\right)$. 

In terms of this definition, the expected limits in the second and
third cases studied by Carrère are terraces belonging to
\[
\mathscr{T}\left(\mathbf{0},2\sqrt{d_{2}r_{2}},\alpha_{2}\mathbf{e}_{2},c_{\alpha_{1}\mathbf{e}_{1}\to\alpha_{2}\mathbf{e}_{2}},\alpha_{1}\mathbf{e}_{1}\right)
\]
with a semi-extinct first profile.

The obvious conjecture is then that all propagating solutions invading
$\mathbf{0}$ apart from semi-extinct monostable traveling waves belong
to
\[
\bigcup_{i\in\left\{ 1,2\right\} }\bigcup_{c\geq2\sqrt{d_{i}r_{i}}}\mathscr{T}\left(\mathbf{0},c,\alpha_{i}\mathbf{e}_{i},c_{\alpha_{3-i}\mathbf{e}_{3-i}\to\alpha_{i}\mathbf{e}_{i}},\alpha_{3-i}\mathbf{e}_{3-i}\right)
\]
and have a semi-extinct first profile.

The bistable case being more or less understood, we now turn our attention
to the monostable case. Let $\mathbf{v}_{s}\in\left\{ \alpha_{1}\mathbf{e}_{1},\alpha_{2}\mathbf{e}_{2},\mathbf{v}_{m}\right\} $
be the unique stable state, $\mathbf{v}_{u}\in\left\{ \mathbf{0},\alpha_{1}\mathbf{e}_{1},\alpha_{2}\mathbf{e}_{2}\right\} $
be an unstable state and consider the Cauchy problem with compactly
supported perturbations of $\mathbf{v}_{u}$ as initial data. Although
the case $\mathbf{v}_{u}=\alpha_{i}\mathbf{e}_{i}$ with 
\[
i\in\mathsf{I}_{u}=\left\{ j\in\left\{ 1,2\right\} \ |\ \alpha_{j}\mathbf{e}_{j}\neq\mathbf{v}_{s}\right\} .
\]
 is well understood (Lewis, Li and Weinberger proved the uniqueness
of the spreading speed $c_{\mathbf{v}_{s}\to\alpha_{i}\mathbf{e}_{i}}^{\star}$
\cite{Lewis_Weinberg,Weinberger_Lew}), the case $\mathbf{v}_{u}=\mathbf{0}$
is much more intricate: in particular, for $\mathbf{v}_{s}=\mathbf{v}_{m}$,
a recent theorem analogous to that of Carrère and due to Lin and Li
\cite{Lin_Li_2012} shows that if $d_{2}r_{2}>d_{1}r_{1}$, then $u_{2}$
will invade first at speed $2\sqrt{d_{2}r_{2}}$ and then be chased
by $u$. Although straightforward comparisons show that the replacement
occurs somewhere in $\left[c_{\mathbf{v}_{m}\to\alpha_{2}\mathbf{e}_{2}}^{\star}t,2\sqrt{d_{1}r_{1}}t\right]$,
the exact speed of $u$ is a delicate question, unsettled in the paper
of Lin and Li.

Tang and Fife \cite{Tang_Fife_1980} established by phase-plane analysis
that traveling waves connecting $\mathbf{0}$ to $\mathbf{v}_{s}$
exist if and only if the speed $c$ satisfies $c\geq c_{\mathbf{v}_{s}\to\mathbf{0}}^{TW}$,
where
\[
c_{\mathbf{v}_{s}\to\mathbf{0}}^{TW}=2\sqrt{\max\limits _{i\in\left\{ 1,2\right\} }d_{i}r_{i}}
\]
is linearly determinate. 

Terraces connecting $\mathbf{0}$ to $\mathbf{v}_{s}$ through an
intermediate unstable state $\alpha_{i}\mathbf{e}_{i}$ with $i\in\mathsf{I}_{u}$
should involve semi-extinct monostable traveling waves connecting
$\mathbf{0}$ to $\alpha_{i}\mathbf{e}_{i}$ and monostable traveling
waves connecting $\alpha_{i}\mathbf{e}_{i}$ to $\mathbf{v}_{s}$.
Again, there exists a minimal wave speed $c_{\mathbf{v}_{s}\to\alpha_{i}\mathbf{e}_{i}}^{TW}$,
as proved for instance by Kan\textendash On \cite{Kan_On_1997} or
Lewis\textendash Li\textendash Weinberger \cite{Li_Weinberger_}.
Recall that $c_{\mathbf{v}_{s}\to\alpha_{i}\mathbf{e}_{i}}^{TW}$
is not linearly determinate in general, however it is bounded from
below by the linear speed:
\[
c_{\mathbf{v}_{s}\to\alpha_{i}\mathbf{e}_{i}}^{TW}\geq2\sqrt{d_{3-i}r_{3-i}\left(1-\frac{c_{3-i,i}r_{i}}{c_{i,i}r_{3-i}}\right)}.
\]
In any case, it is natural to expect that for all $i\in\mathsf{I}_{u}$,
terraces belonging to $\mathscr{T}\left(\mathbf{0},c,\alpha_{i}\mathbf{e}_{i},c',\mathbf{v}_{s}\right)$
with a semi-extinct first profile exist if and only if
\[
\left\{ \begin{matrix}c_{\mathbf{v}_{s}\to\alpha_{i}\mathbf{e}_{i}}\leq c'\\
2\sqrt{d_{i}r_{i}}\leq c\\
c'<c.
\end{matrix}\right.
\]

Consequently, the conjecture is that all propagating solutions invading
$\mathbf{0}$ apart from (possibly semi-extinct) monostable traveling
waves belong to
\[
\bigcup_{i\in\mathsf{I}_{u}}\bigcup_{c\geq2\sqrt{d_{i}r_{i}}}\bigcup_{c'\geq c_{\mathbf{v}_{s}\to\alpha_{i}\mathbf{e}_{i}}}\mathscr{T}\left(\mathbf{0},c,\alpha_{i}\mathbf{e}_{i},c',\mathbf{v}_{s}\right)
\]
and have a semi-extinct first profile.

Having these conjectures in mind, we introduce small mutations and
wonder how they affect the outcome. An heuristic answer due to Elliott
and Cornell \cite{Elliott_Cornel} suggests that \textquotedblleft the
only role of mutations is to ensure that both morphs travel at the
same speed\textquotedblright . Therefore, there might exist functions
$\mathbf{u}^{0}:\mathbb{R}\to\mathsf{K}$ such that the solutions
$\left(\mathbf{u}_{\eta}\right)_{\eta\geq0}$ of the Cauchy problem
associated with $\left(E_{KPP}\right)_{\eta}$ with initial data $\mathbf{u}^{0}$
admit as long-time asymptotic a traveling wave if $\eta>0$ and a
terrace of $\mathscr{T}\left(\mathbf{0},c,\alpha_{i}\mathbf{e}_{i},c',\mathbf{v}\right)$
if $\eta=0$. We refer hereafter to such traveling waves as quasi-$\mathscr{T}\left(\mathbf{0},c,\alpha_{i}\mathbf{e}_{i},c',\mathbf{v}\right)$
traveling waves.

In order to study these special traveling waves, we resort to numerical
simulations. We find two completely different behaviors.
\begin{itemize}
\item In the bistable case (\figref{Numerics_bistable_spreading}), quasi-$\mathscr{T}\left(\mathbf{0},2\sqrt{d_{i}r_{i}},\alpha_{i}\mathbf{e}_{i},c_{\alpha_{j}\mathbf{e}_{j}\to\alpha_{i}\mathbf{e}_{i}},\alpha_{j}\mathbf{e}_{j}\right)$
traveling waves (with $i\in\left\{ 1,2\right\} $ and $j=3-i$) converge
as $\eta\to0$ to a semi-extinct traveling wave connecting $\mathbf{0}$
to $\alpha_{j}\mathbf{e}_{j}$ if $c_{\alpha_{j}\mathbf{e}_{j}\to\alpha_{i}\mathbf{e}_{i}}>0$
and to $\alpha_{i}\mathbf{e}_{i}$ if $c_{\alpha_{j}\mathbf{e}_{j}\to\alpha_{i}\mathbf{e}_{i}}<0$.
\end{itemize}

\begin{figure}
\subfloat[$t=0$]{\resizebox{.3\hsize}{!}{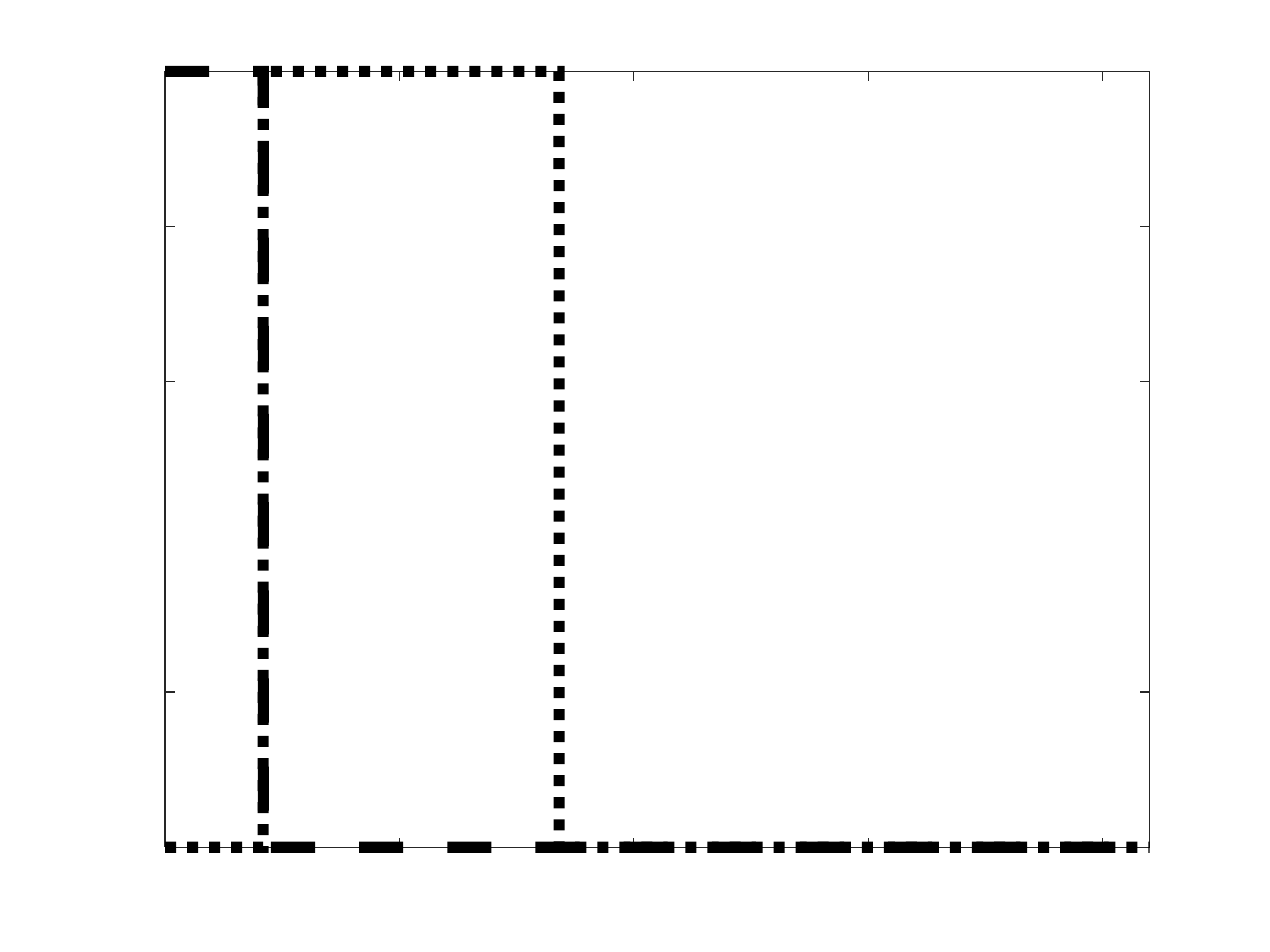}}\subfloat[$t=20$]{\resizebox{.3\hsize}{!}{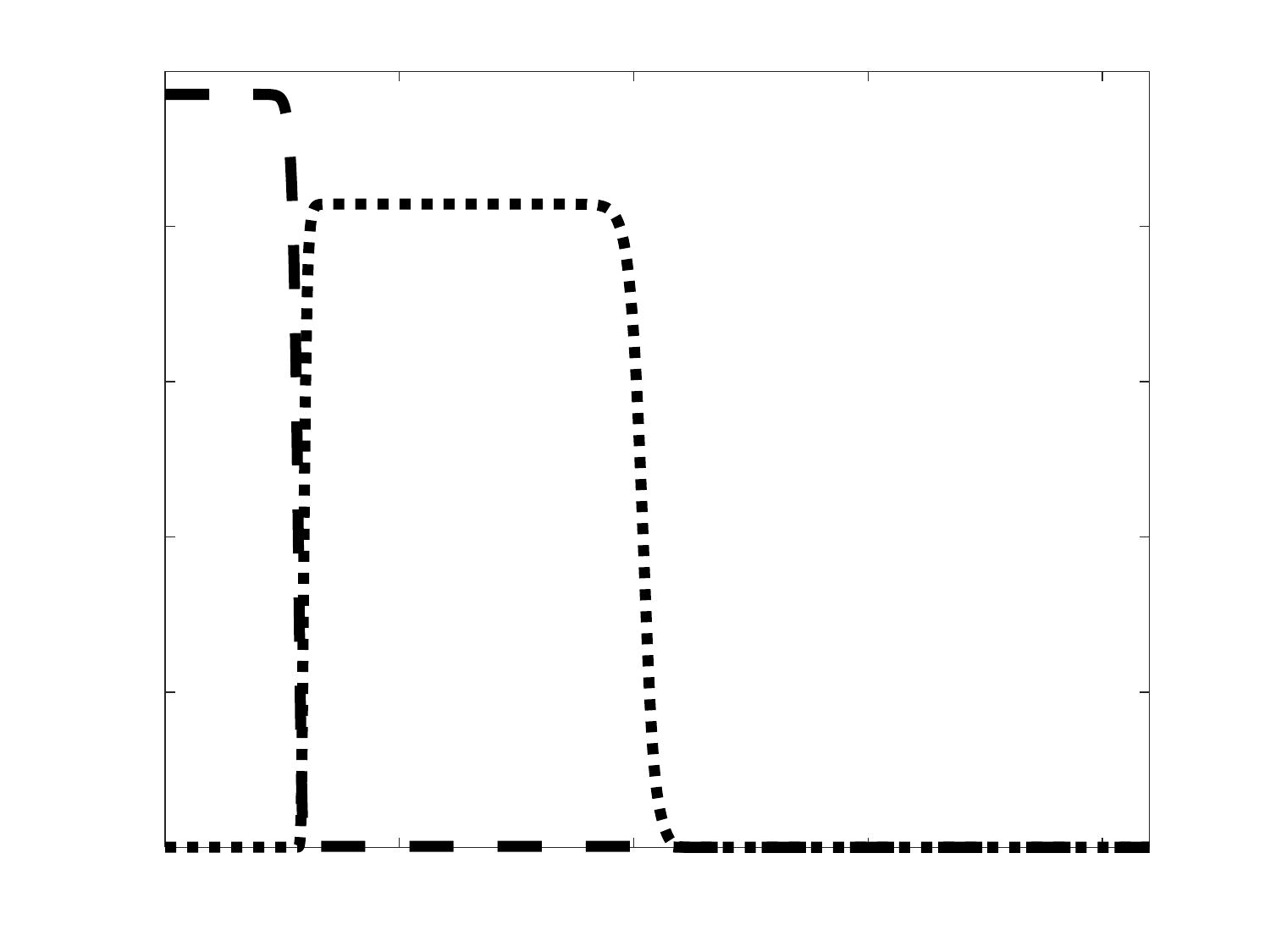}}\subfloat[$t=50$]{\resizebox{.3\hsize}{!}{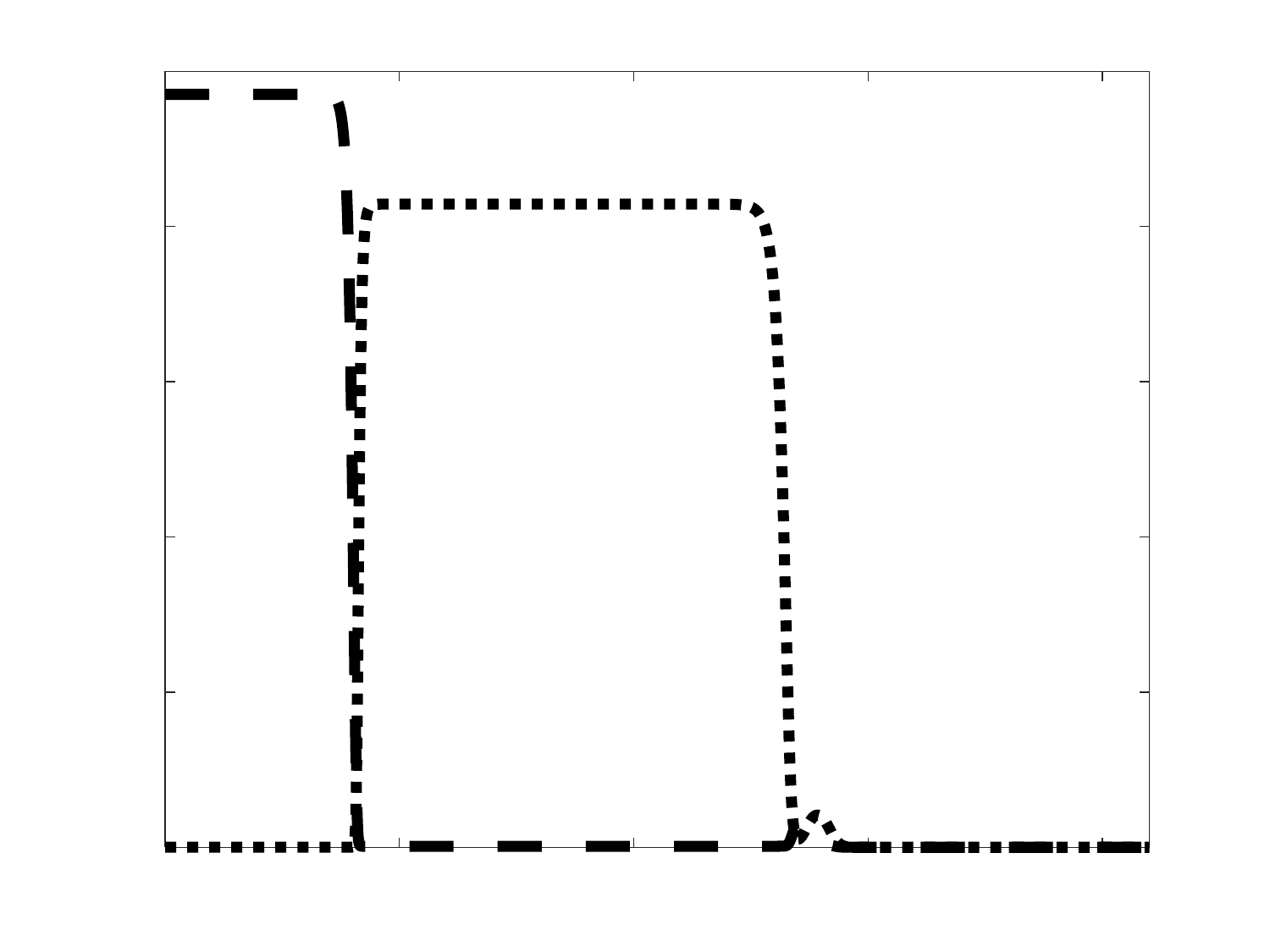}}\\
\subfloat[$t=55$]{\resizebox{.3\hsize}{!}{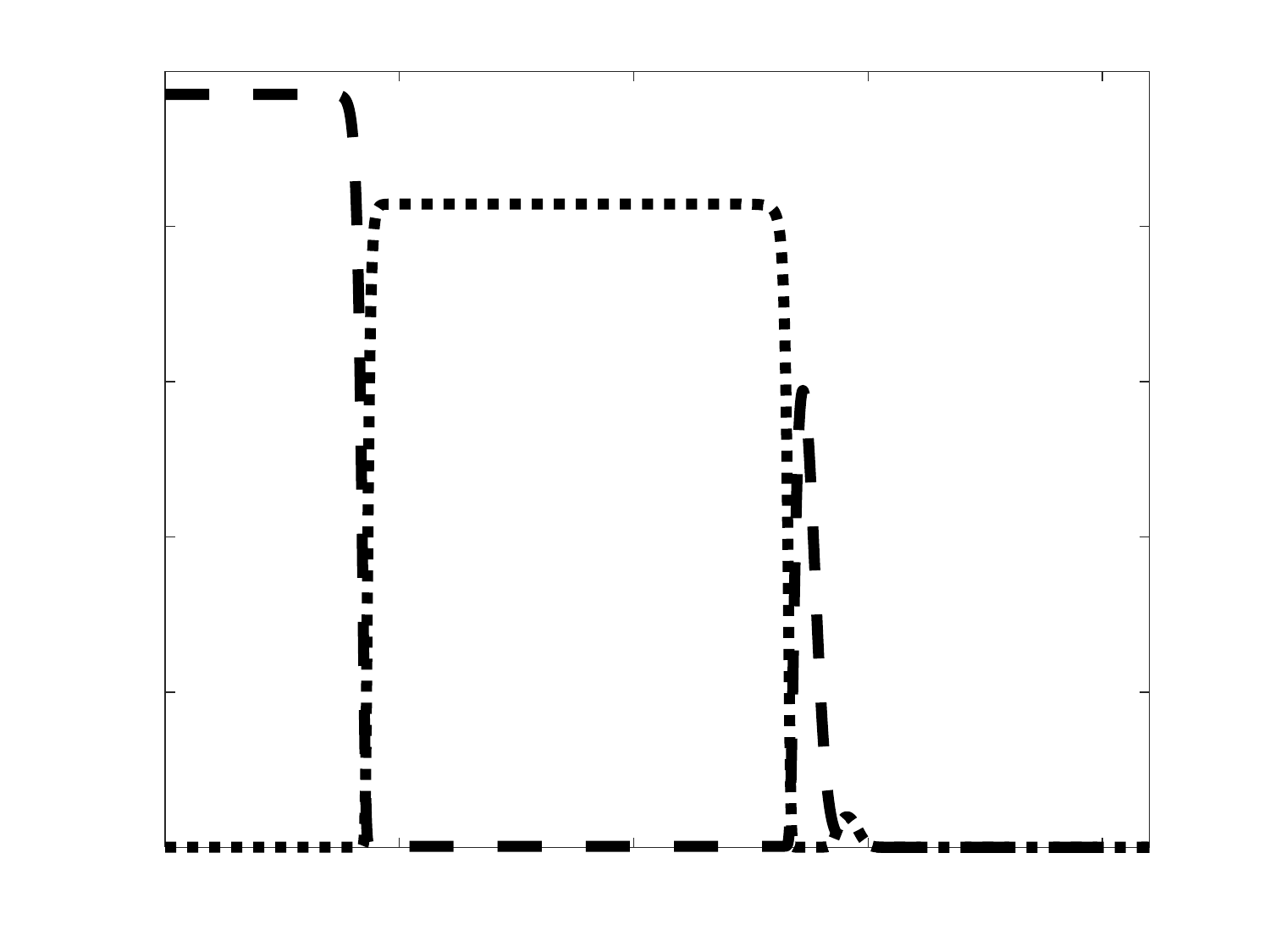}}\subfloat[$t=65$]{\resizebox{.3\hsize}{!}{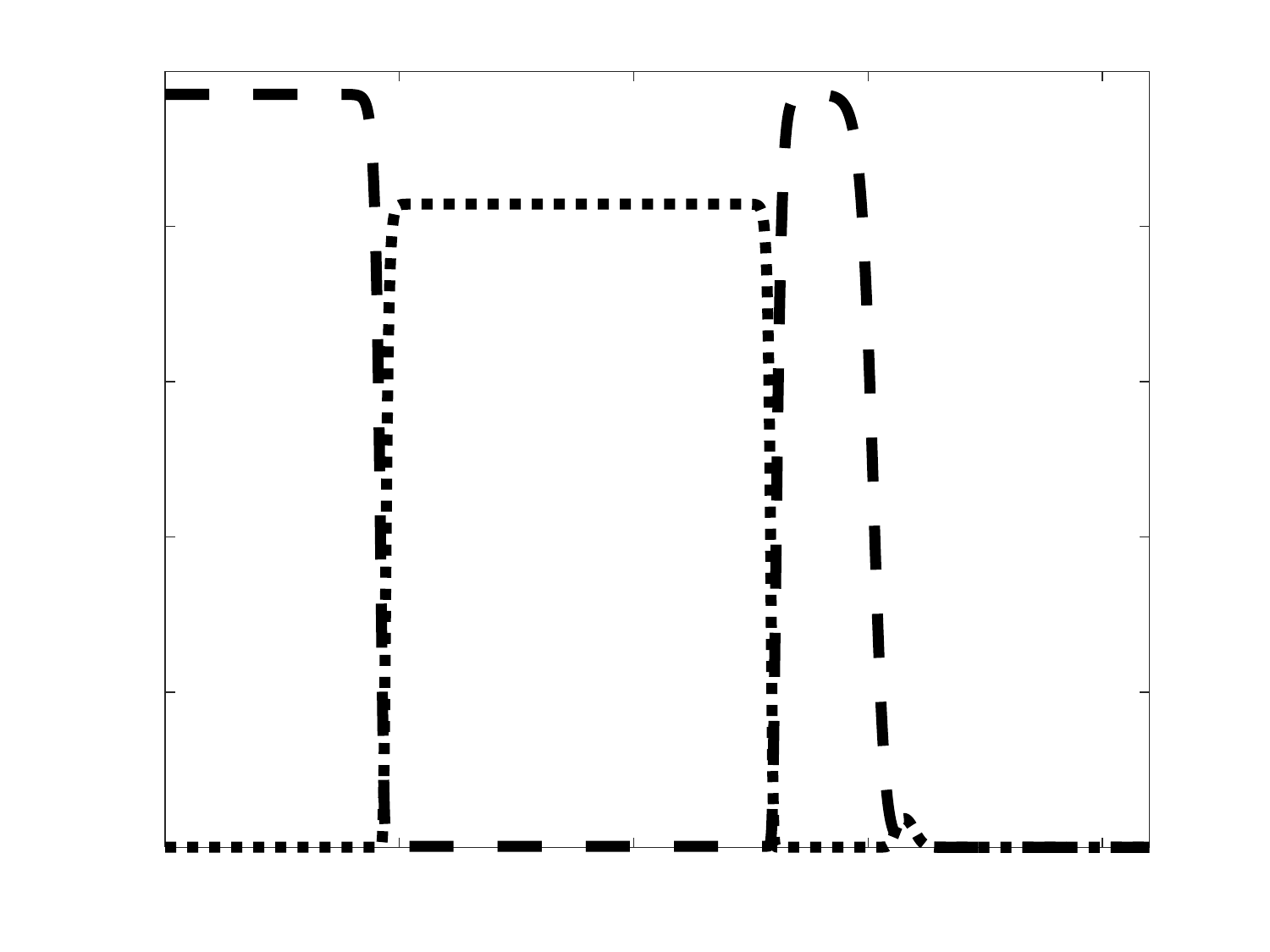}}\subfloat[$t=100$]{\resizebox{.3\hsize}{!}{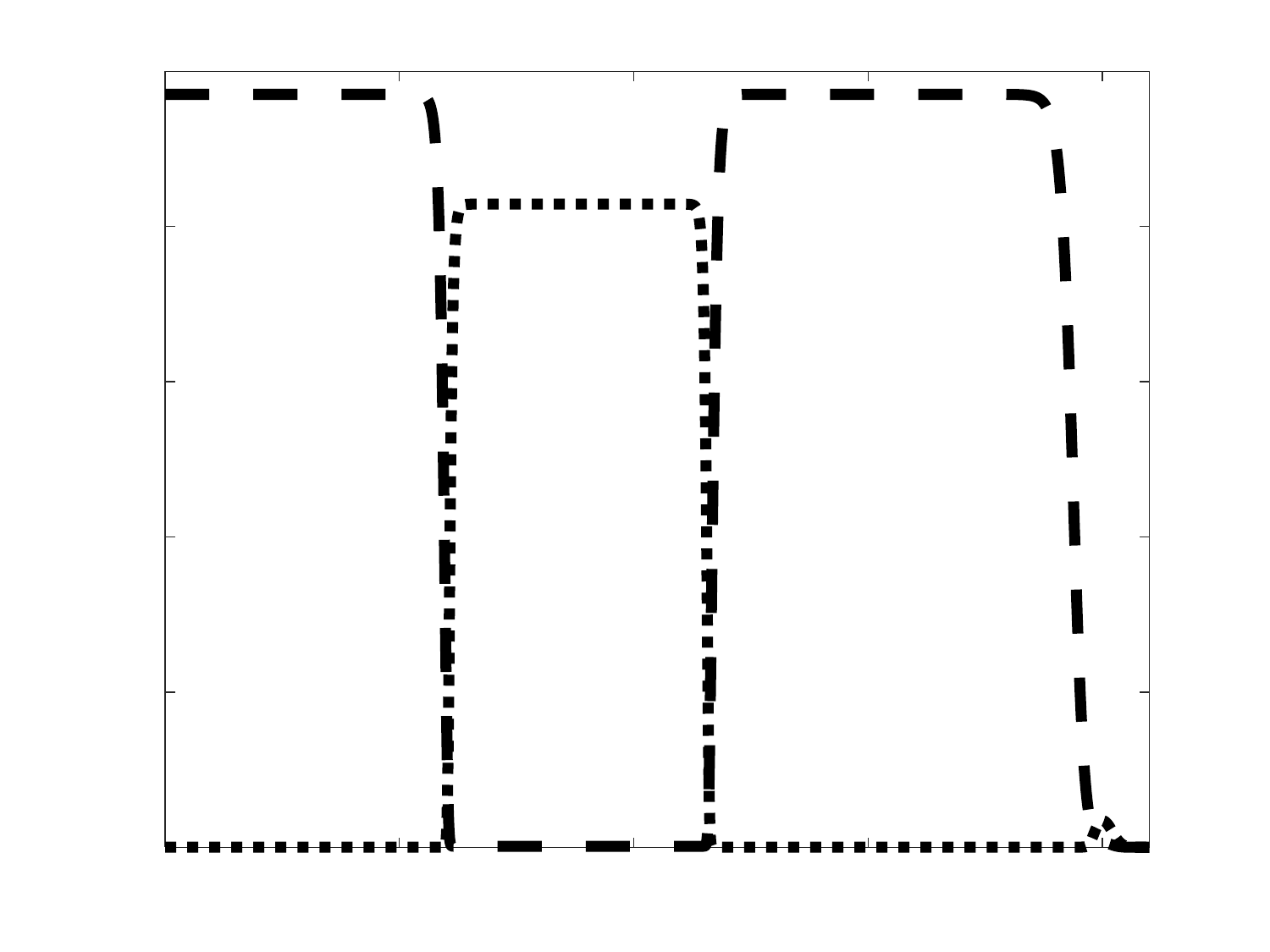}}

\caption{{\small{}\label{fig:Numerics_bistable_spreading} }Numerical simulation
of the bistable case with initial data corresponding to a competition\textendash diffusion
terrace ($u_{1}$ dashed line, $u_{2}$ dotted line, $x$ as horizontal
axis). \protect \\
Parameter values: $\mathbf{d}=\left(1,1.5125\right)^{T}$, $\mathbf{r}=\mathbf{1}_{2,1}$,
$\mathbf{m}=\mathbf{1}_{2,1}$, $\eta=0.025$, $c_{1,1}=c_{2,2}=1$,
$c_{1,2}=20$, $c_{2,1}=110$, so that \cite{Girardin_Nadin_2015}
$c_{\alpha_{1}\mathbf{e}_{1}\to\alpha_{2}\mathbf{e}_{2}}>0$. \protect \\
The traveling wave which is on the right at $t=100$, driven by a
very small bump of $u_{2}$ but dominated at the back by $u_{1}$,
is the long-time asymptotic. Indeed the $u_{2}$-dominated area in
the middle shrinks from both sides at a speed close to $\left|c_{\alpha_{1}\mathbf{e}_{1}\to\alpha_{2}\mathbf{e}_{2}}\right|$
and will ultimately disappear.}
\end{figure}

\begin{itemize}
\item In the monostable case (\figref{Numerics_monostable_spreading}),
for all $i\in\mathsf{I}_{u}$, quasi-$\mathscr{T}\left(\mathbf{0},2\sqrt{d_{i}r_{i}},\alpha_{i}\mathbf{e}_{i},c',\mathbf{v}_{s}\right)$
traveling waves connect $\mathbf{0}$ to $\mathbf{v}_{s}$ through
an intermediate bump of $u_{i}$. As $\eta\to0$, the amplitude of
this bump tends to $\alpha_{i}$ while its length tends slowly to
$+\infty$ (seemingly like $\ln\eta$). Therefore, depending on the
normalization, the limit of the profiles as $\eta\to0$ is either
a semi-extinct connection between $\mathbf{0}$ and $\alpha_{i}\mathbf{e}_{i}$
or a monostable connection between $\alpha_{i}\mathbf{e}_{i}$ and
$\mathbf{v}_{s}$.
\begin{figure}
\subfloat[$t=0$]{\resizebox{.3\hsize}{!}{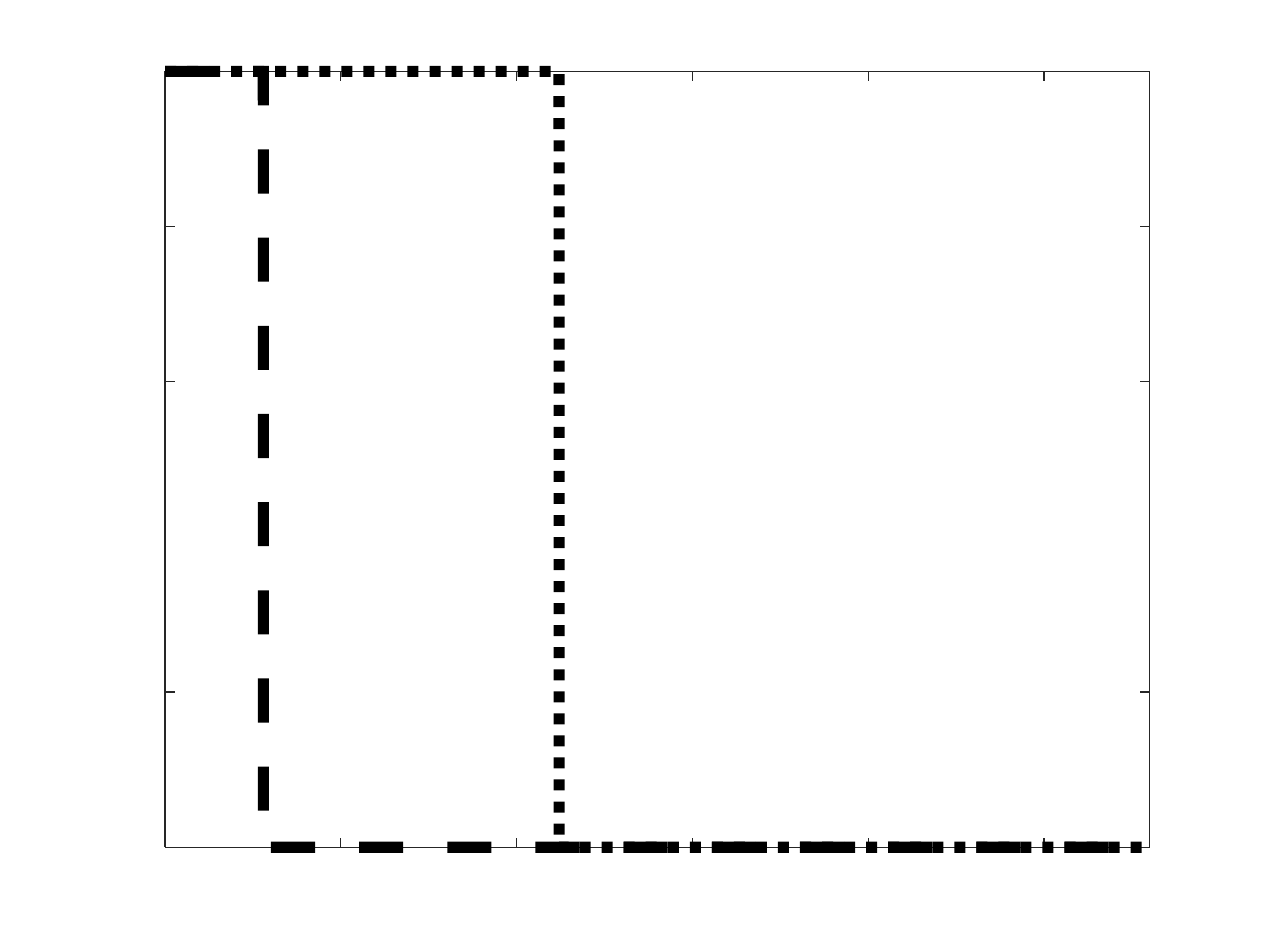}}\subfloat[$t=5$]{\resizebox{.3\hsize}{!}{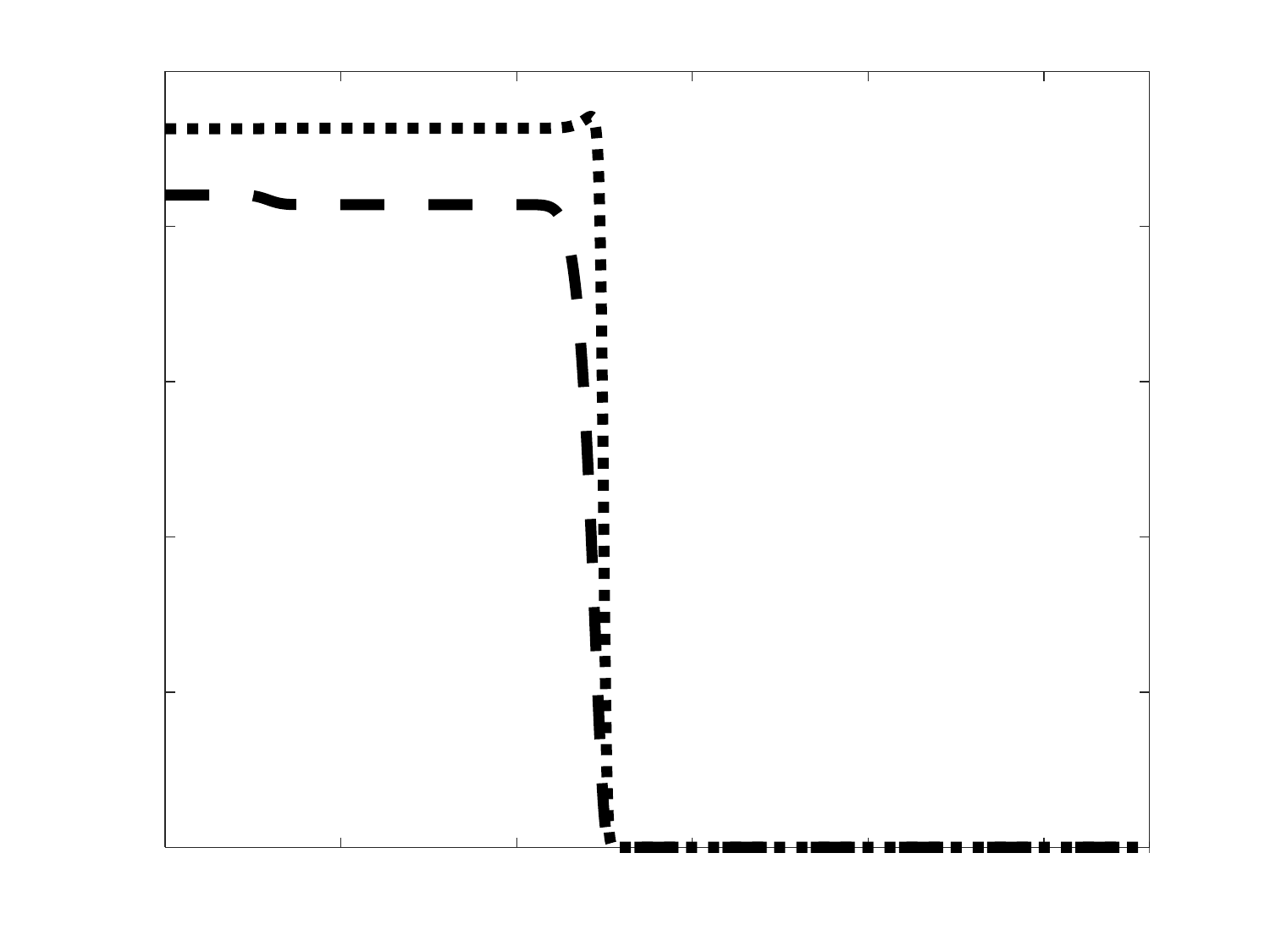}}\subfloat[$t=40$]{\resizebox{.3\hsize}{!}{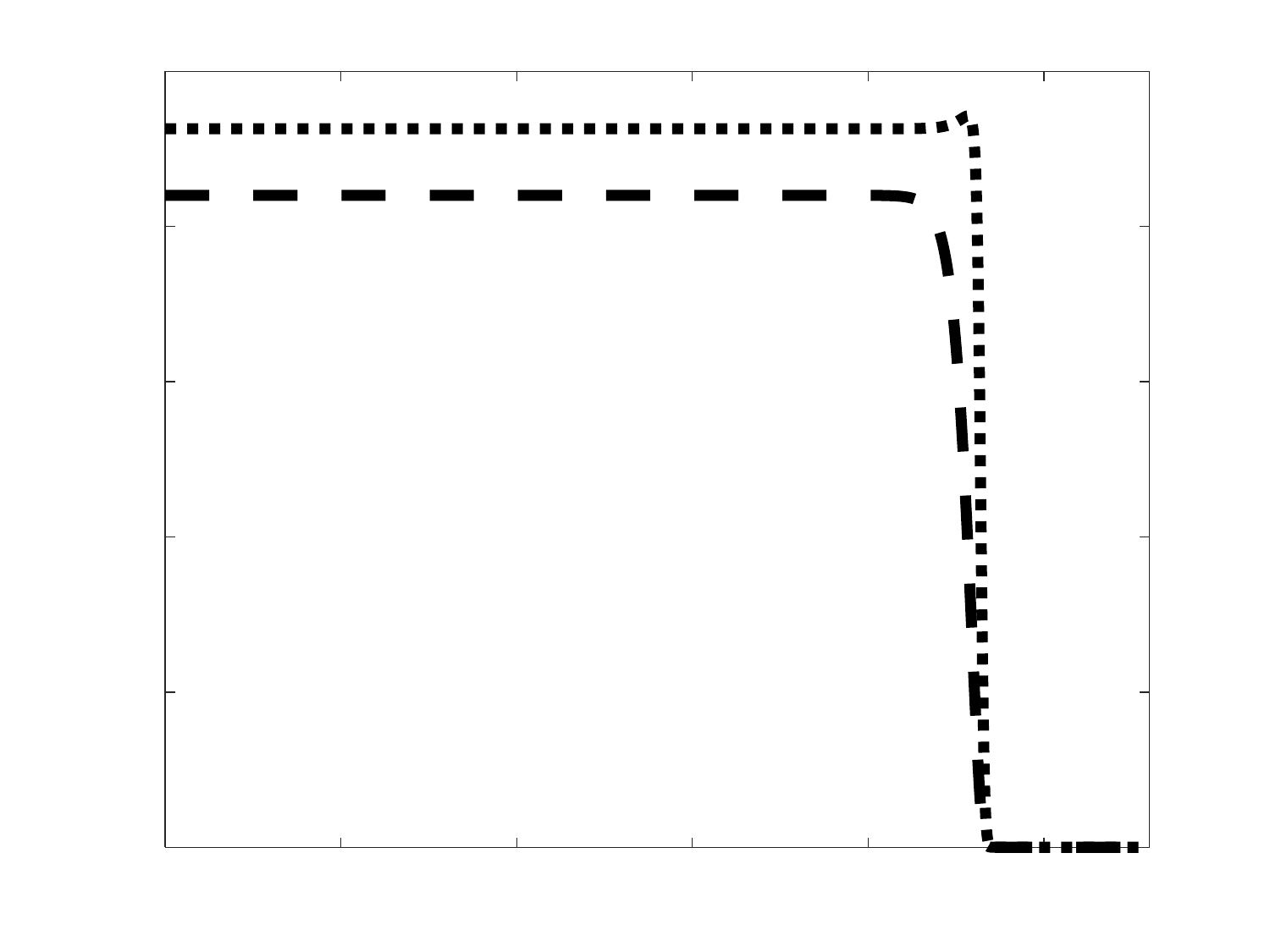}}\\
\subfloat[$t=0$]{\resizebox{.3\hsize}{!}{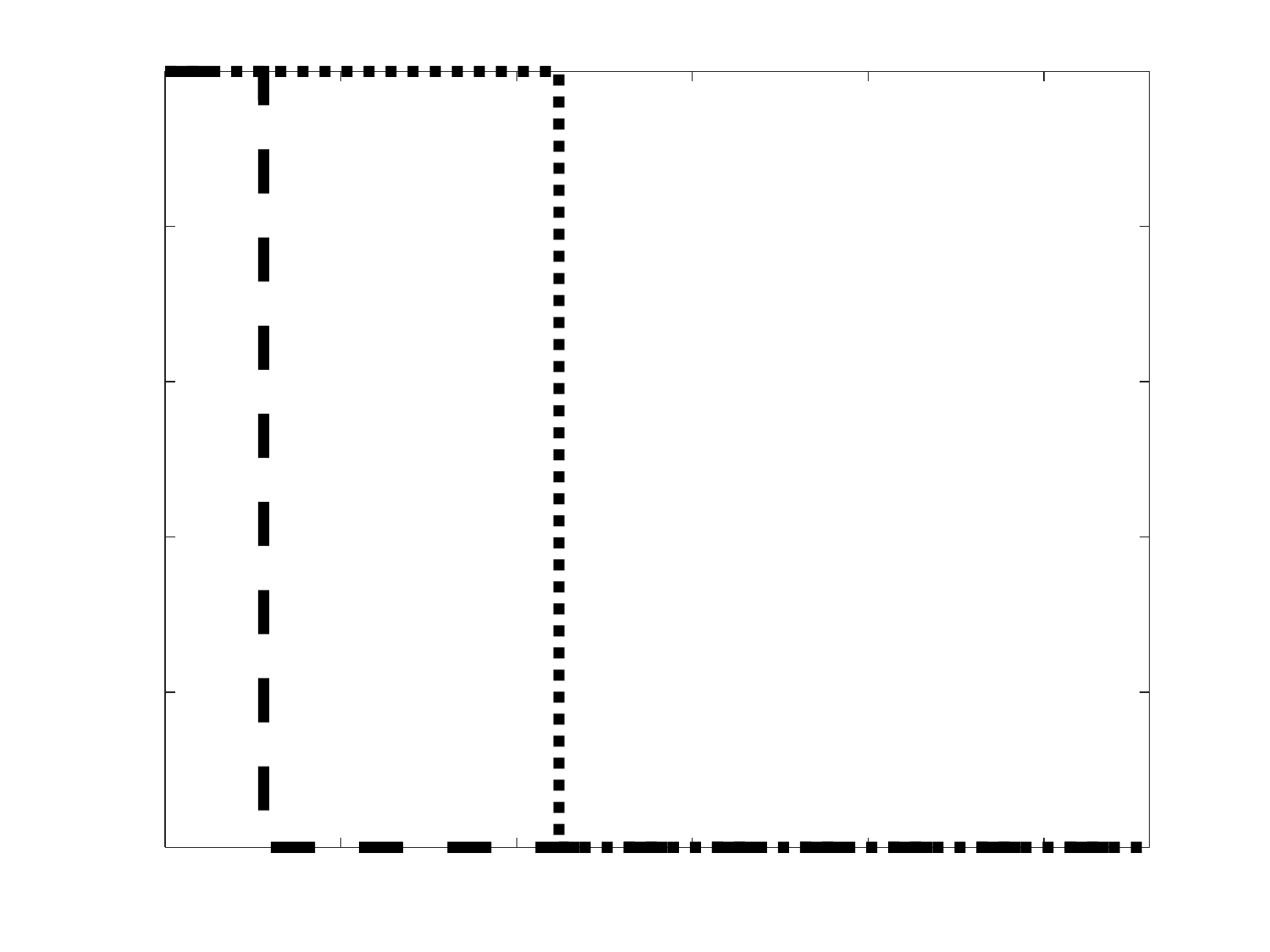}}\subfloat[$t=5$]{\resizebox{.3\hsize}{!}{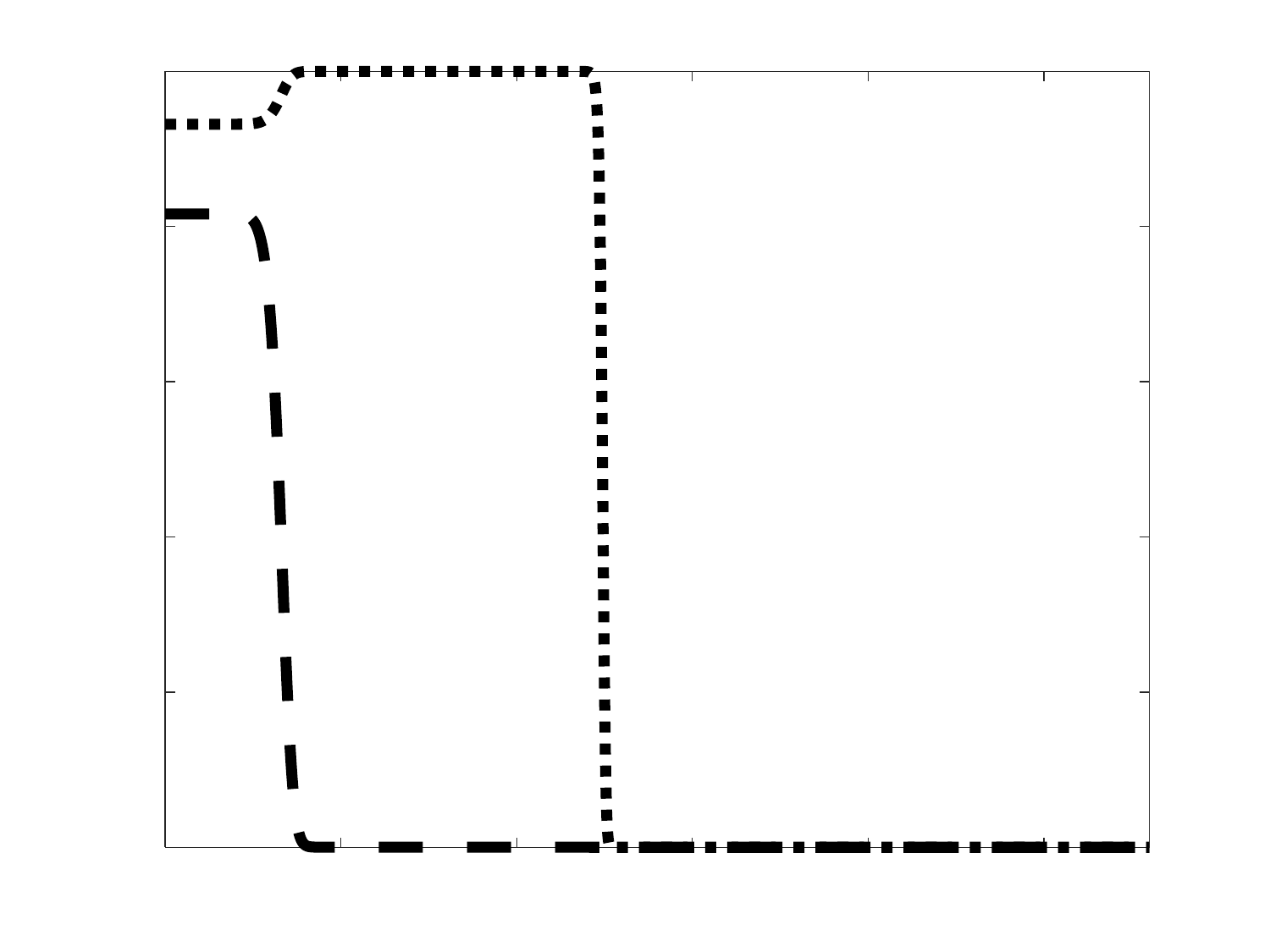}}\subfloat[$t=40$]{\resizebox{.3\hsize}{!}{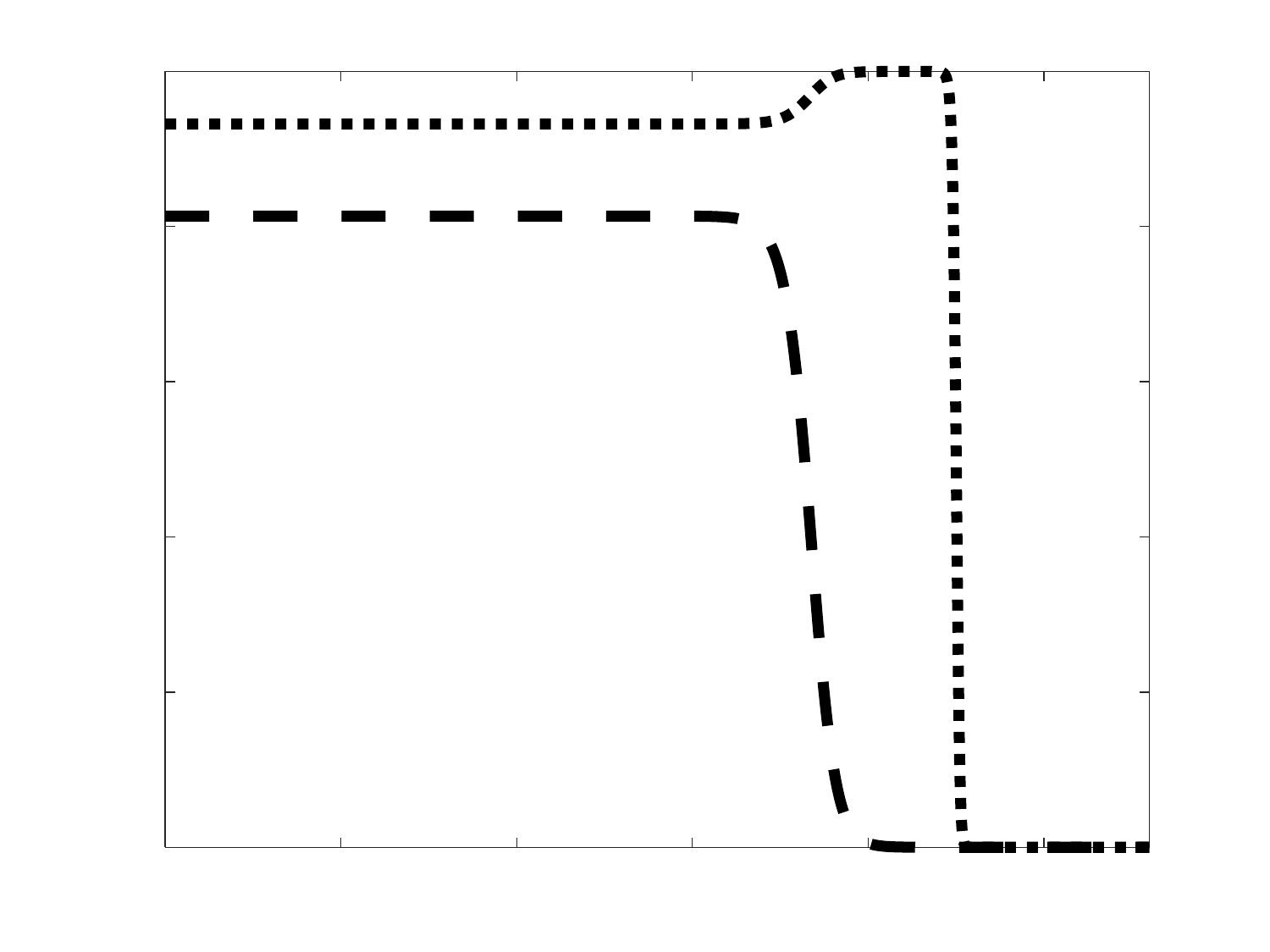}}\\
\subfloat[$t=0$]{\resizebox{.3\hsize}{!}{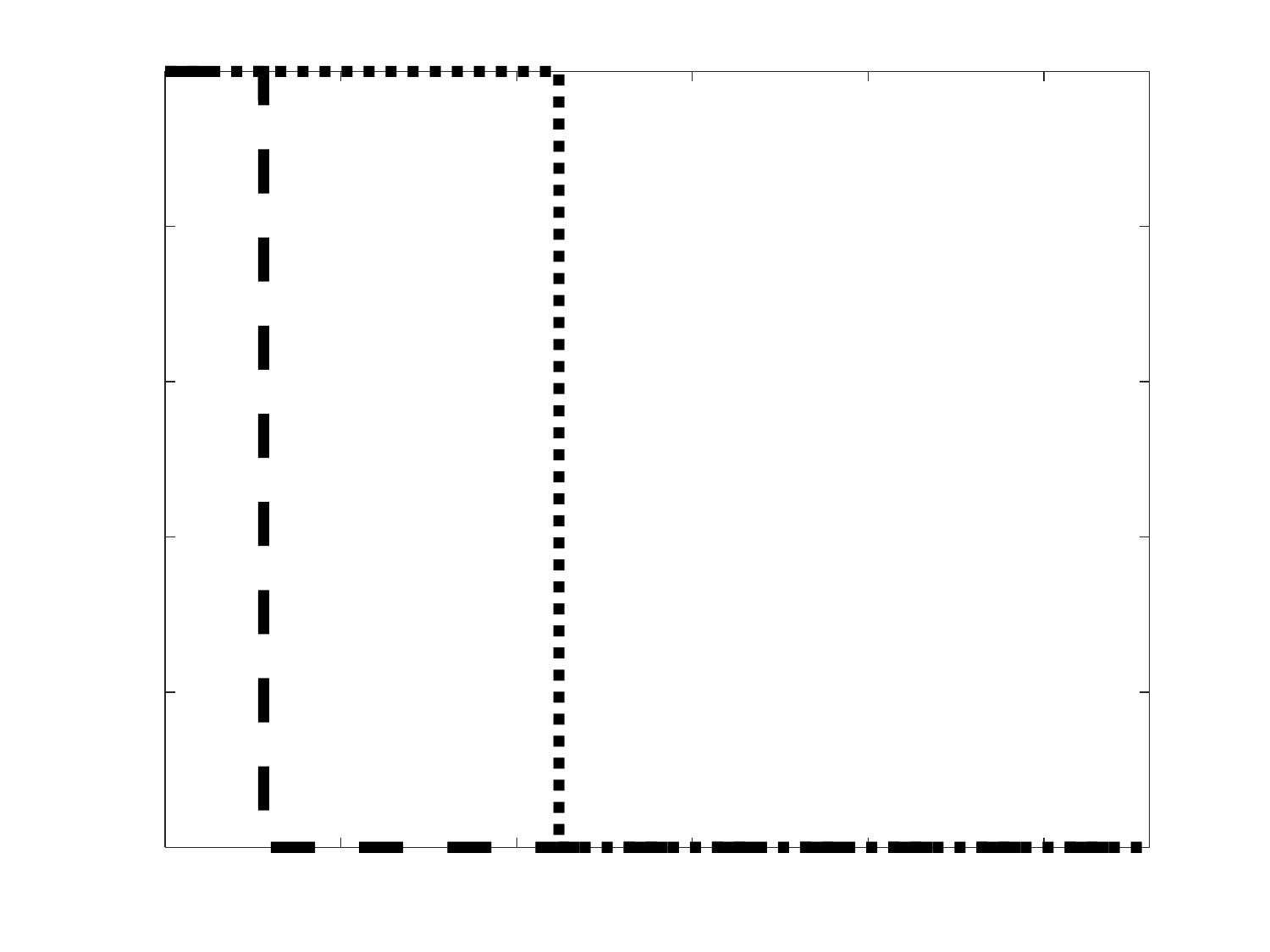}}\subfloat[$t=5$]{\resizebox{.3\hsize}{!}{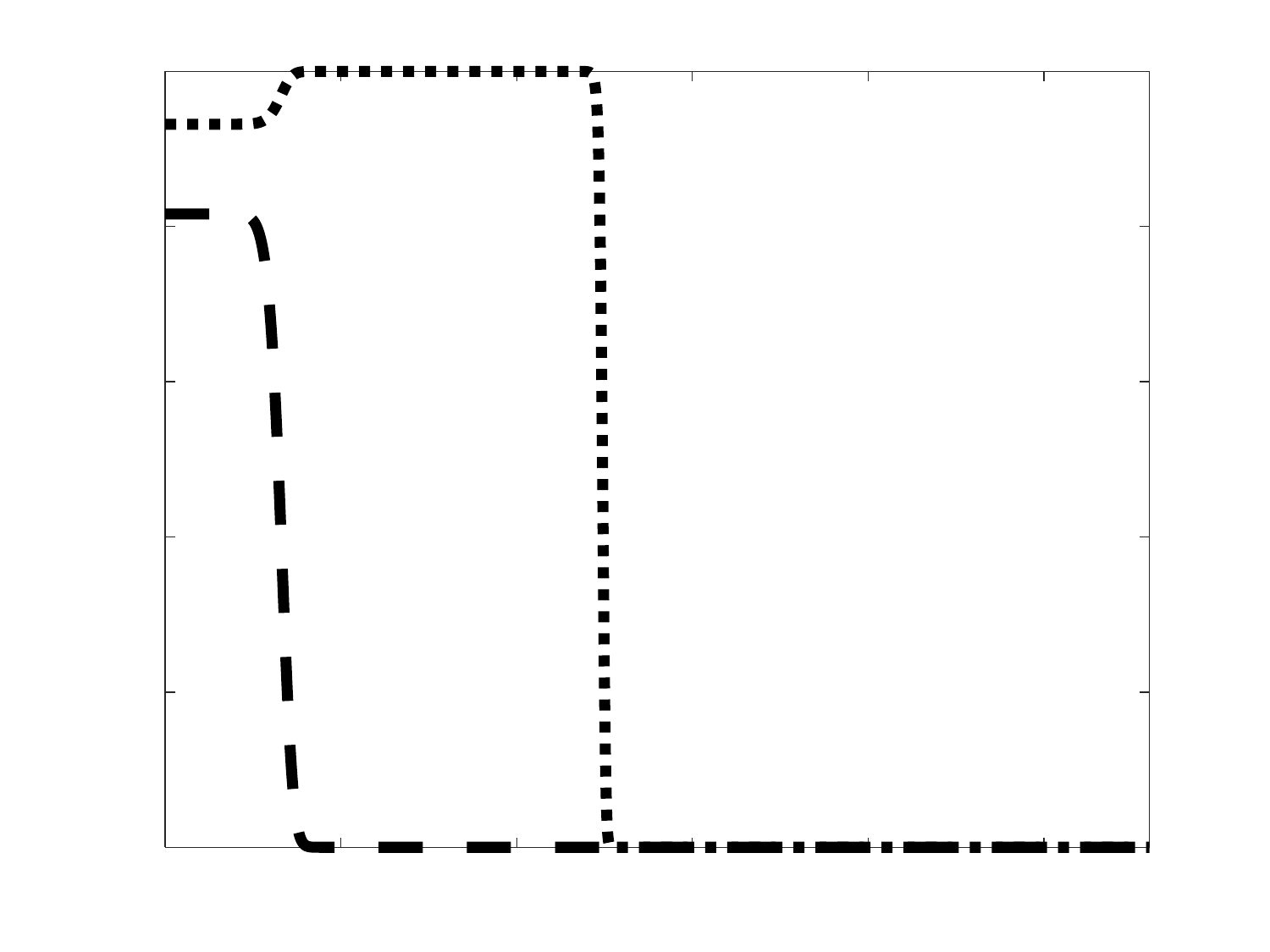}}\subfloat[$t=40$]{\resizebox{.3\hsize}{!}{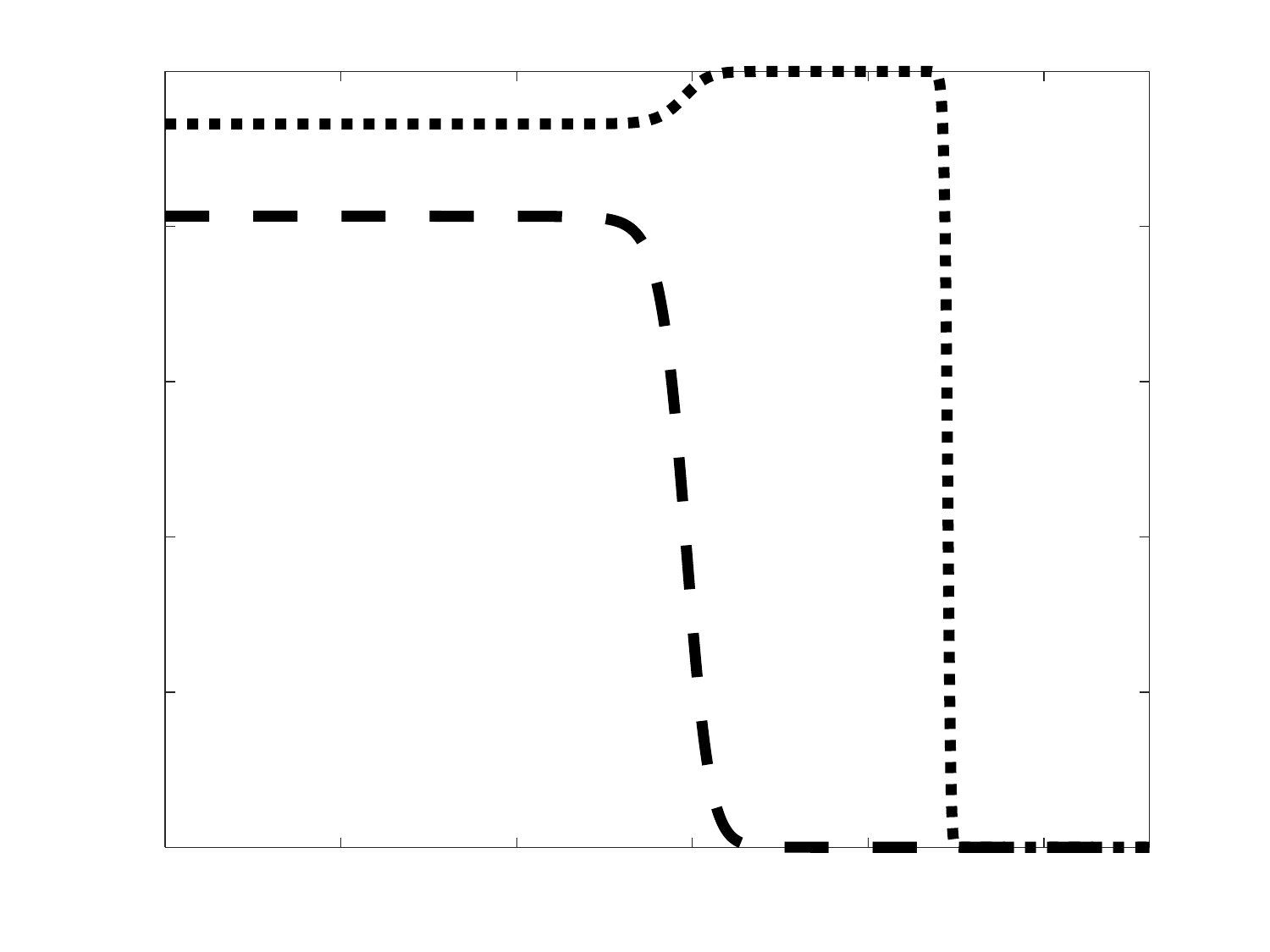}}

\caption{\label{fig:Numerics_monostable_spreading} Numerical simulations of
the monostable case with initial data corresponding to a competition\textendash diffusion
terrace ($u_{1}$ dashed line, $u_{2}$ dotted line, $x$ as horizontal
axis). \protect \\
Parameter values: $\mathbf{d}=\left(1,\frac{1}{3}\right)^{T}$, $\mathbf{r}=\left(1,6\right)^{T}$,
$\mathbf{m}=\mathbf{1}_{2,1}$, $c_{1,1}=1$, $c_{2,2}=6$, $c_{1,2}=0.2$,
$c_{2,1}=0.5$, $\eta=2.5\times10^{-1}$ on the first line, $\eta=2.5\times10^{-6}$
on the second line, $\eta=2.5\times10^{-11}$ on the third line.}
\end{figure}
\end{itemize}

\subsection{Why is \conjref{(H_7)} silent about the bistable case with $c_{\alpha_{1}\mathbf{e}_{1}\to\alpha_{2}\mathbf{e}_{2}}=0$?}

In this very special case, additional asymmetry assumptions on the
coefficients are necessary in order to exclude connections between
$\mathbf{0}$ and the saddle-point $\mathbf{v}_{m}$, as indicated
by the following immediate proposition, built on a counter-example
given in \subsecref{Delicate_back}.
\begin{prop}
Assume $\left(H_{7}\right)$, $\mathbf{d}=\mathbf{1}_{2,1}$, $\mathbf{r}=\mathbf{1}_{2,1}$,
$\mathbf{m}=\frac{1}{\sqrt{2}}\mathbf{1}_{2,1}$ and the existence
of $a\in\left(1,+\infty\right)$ such that
\[
\mathbf{C}=\left(\begin{matrix}1 & a\\
a & 1
\end{matrix}\right).
\]
 Then $\mathbf{v}_{m}=\frac{1}{\lambda_{PF}\left(\mathbf{C}\right)}\mathbf{1}_{2,1}\in\mathsf{K}^{++}$
is a saddle-point and, for all $\eta\geq0$ and all $c\geq2$, there
exists a unique $p_{c,\eta}\in\mathscr{C}^{2}\left(\mathbb{R}\right)$
such that
\[
\left\{ \begin{matrix}p_{c,\eta}\mathbf{1}_{2,1}\in\mathscr{P}_{c,\eta}\\
p_{c,\eta}\left(0\right)=\frac{1}{2\lambda_{PF}\left(\mathbf{C}\right)}\\
\lim\limits _{\xi\to-\infty}p_{c,\eta}\left(\xi\right)=\frac{1}{\lambda_{PF}\left(\mathbf{C}\right)}.
\end{matrix}\right.
\]
In particular, $\left(p_{c,\eta}\mathbf{1}_{2,1},c\right)$ connects
$\mathbf{0}$ to $\mathbf{v}_{m}$. 

Furthermore, 

\[
\left(c,\eta\right)\mapsto p_{c,\eta}\in\mathscr{C}\left([2,+\infty)\times[0,+\infty),\mathscr{W}^{2,\infty}\left(\mathbb{R},\mathbb{R}\right)\right).
\]
\end{prop}

\section*{Acknowledgments}

The author thanks Grégoire Nadin for the attention he paid to this
work and Elaine Crooks for fruitful discussions. He also thanks anonymous
referees and the associate editor for very valuable comments which
led to a much clearer manuscript.

\bibliographystyle{plain}
\bibliography{ref}

\end{document}

%% file: bistable_0.pdf_tex
\begingroup%
  \makeatletter%
  \providecommand\color[2][]{%
    \errmessage{(Inkscape) Color is used for the text in Inkscape, but the package 'color.sty' is not loaded}%
    \renewcommand\color[2][]{}%
  }%
  \providecommand\transparent[1]{%
    \errmessage{(Inkscape) Transparency is used (non-zero) for the text in Inkscape, but the package 'transparent.sty' is not loaded}%
    \renewcommand\transparent[1]{}%
  }%
  \providecommand\rotatebox[2]{#2}%
  \ifx\svgwidth\undefined%
    \setlength{\unitlength}{432bp}%
    \ifx\svgscale\undefined%
      \relax%
    \else%
      \setlength{\unitlength}{\unitlength * \real{\svgscale}}%
    \fi%
  \else%
    \setlength{\unitlength}{\svgwidth}%
  \fi%
  \global\let\svgwidth\undefined%
  \global\let\svgscale\undefined%
  \makeatother%
  \begin{picture}(1,0.75)%
    \put(0,0){\includegraphics[width=\unitlength,page=1]{bistable_0.pdf}}%
    \put(0.13,0.05381773){\color[rgb]{0.14901961,0.14901961,0.14901961}\makebox(0,0)[b]{\smash{0}}}%
    \put(0.31452431,0.05381773){\color[rgb]{0.14509804,0.14509804,0.14509804}\makebox(0,0)[b]{\smash{100}}}%
    \put(0.49904686,0.05381773){\color[rgb]{0.14509804,0.14509804,0.14509804}\makebox(0,0)[b]{\smash{200}}}%
    \put(0.68357118,0.05381773){\color[rgb]{0.14509804,0.14509804,0.14509804}\makebox(0,0)[b]{\smash{300}}}%
    \put(0.8680955,0.05381773){\color[rgb]{0.14509804,0.14509804,0.14509804}\makebox(0,0)[b]{\smash{400}}}%
    \put(0.12131164,0.08249998){\color[rgb]{0.14901961,0.14901961,0.14901961}\makebox(0,0)[rb]{\smash{0}}}%
    \put(0.12131164,0.20475001){\color[rgb]{0.14509804,0.14509804,0.14509804}\makebox(0,0)[rb]{\smash{0.2}}}%
    \put(0.12131164,0.32700001){\color[rgb]{0.14509804,0.14509804,0.14509804}\makebox(0,0)[rb]{\smash{0.4}}}%
    \put(0.12131164,0.44925001){\color[rgb]{0.14509804,0.14509804,0.14509804}\makebox(0,0)[rb]{\smash{0.6}}}%
    \put(0.12131164,0.5715){\color[rgb]{0.14509804,0.14509804,0.14509804}\makebox(0,0)[rb]{\smash{0.8}}}%
  \end{picture}%
\endgroup%

%% file: bistable_4.pdf_tex
\begingroup%
  \makeatletter%
  \providecommand\color[2][]{%
    \errmessage{(Inkscape) Color is used for the text in Inkscape, but the package 'color.sty' is not loaded}%
    \renewcommand\color[2][]{}%
  }%
  \providecommand\transparent[1]{%
    \errmessage{(Inkscape) Transparency is used (non-zero) for the text in Inkscape, but the package 'transparent.sty' is not loaded}%
    \renewcommand\transparent[1]{}%
  }%
  \providecommand\rotatebox[2]{#2}%
  \ifx\svgwidth\undefined%
    \setlength{\unitlength}{432bp}%
    \ifx\svgscale\undefined%
      \relax%
    \else%
      \setlength{\unitlength}{\unitlength * \real{\svgscale}}%
    \fi%
  \else%
    \setlength{\unitlength}{\svgwidth}%
  \fi%
  \global\let\svgwidth\undefined%
  \global\let\svgscale\undefined%
  \makeatother%
  \begin{picture}(1,0.75)%
    \put(0,0){\includegraphics[width=\unitlength,page=1]{bistable_4.pdf}}%
    \put(0.13,0.05381773){\color[rgb]{0.14901961,0.14901961,0.14901961}\makebox(0,0)[b]{\smash{0}}}%
    \put(0.31452431,0.05381773){\color[rgb]{0.14509804,0.14509804,0.14509804}\makebox(0,0)[b]{\smash{100}}}%
    \put(0.49904686,0.05381773){\color[rgb]{0.14509804,0.14509804,0.14509804}\makebox(0,0)[b]{\smash{200}}}%
    \put(0.68357118,0.05381773){\color[rgb]{0.14509804,0.14509804,0.14509804}\makebox(0,0)[b]{\smash{300}}}%
    \put(0.8680955,0.05381773){\color[rgb]{0.14509804,0.14509804,0.14509804}\makebox(0,0)[b]{\smash{400}}}%
    \put(0.12131164,0.08249998){\color[rgb]{0.14901961,0.14901961,0.14901961}\makebox(0,0)[rb]{\smash{0}}}%
    \put(0.12131164,0.20475001){\color[rgb]{0.14509804,0.14509804,0.14509804}\makebox(0,0)[rb]{\smash{0.2}}}%
    \put(0.12131164,0.32700001){\color[rgb]{0.14509804,0.14509804,0.14509804}\makebox(0,0)[rb]{\smash{0.4}}}%
    \put(0.12131164,0.44925001){\color[rgb]{0.14509804,0.14509804,0.14509804}\makebox(0,0)[rb]{\smash{0.6}}}%
    \put(0.12131164,0.5715){\color[rgb]{0.14509804,0.14509804,0.14509804}\makebox(0,0)[rb]{\smash{0.8}}}%
  \end{picture}%
\endgroup%

%% file: bistable_10.pdf_tex
\begingroup%
  \makeatletter%
  \providecommand\color[2][]{%
    \errmessage{(Inkscape) Color is used for the text in Inkscape, but the package 'color.sty' is not loaded}%
    \renewcommand\color[2][]{}%
  }%
  \providecommand\transparent[1]{%
    \errmessage{(Inkscape) Transparency is used (non-zero) for the text in Inkscape, but the package 'transparent.sty' is not loaded}%
    \renewcommand\transparent[1]{}%
  }%
  \providecommand\rotatebox[2]{#2}%
  \ifx\svgwidth\undefined%
    \setlength{\unitlength}{432bp}%
    \ifx\svgscale\undefined%
      \relax%
    \else%
      \setlength{\unitlength}{\unitlength * \real{\svgscale}}%
    \fi%
  \else%
    \setlength{\unitlength}{\svgwidth}%
  \fi%
  \global\let\svgwidth\undefined%
  \global\let\svgscale\undefined%
  \makeatother%
  \begin{picture}(1,0.75)%
    \put(0,0){\includegraphics[width=\unitlength,page=1]{bistable_10.pdf}}%
    \put(0.13,0.05381773){\color[rgb]{0.14901961,0.14901961,0.14901961}\makebox(0,0)[b]{\smash{0}}}%
    \put(0.31452431,0.05381773){\color[rgb]{0.14509804,0.14509804,0.14509804}\makebox(0,0)[b]{\smash{100}}}%
    \put(0.49904686,0.05381773){\color[rgb]{0.14509804,0.14509804,0.14509804}\makebox(0,0)[b]{\smash{200}}}%
    \put(0.68357118,0.05381773){\color[rgb]{0.14509804,0.14509804,0.14509804}\makebox(0,0)[b]{\smash{300}}}%
    \put(0.8680955,0.05381773){\color[rgb]{0.14509804,0.14509804,0.14509804}\makebox(0,0)[b]{\smash{400}}}%
    \put(0.12131164,0.08249998){\color[rgb]{0.14901961,0.14901961,0.14901961}\makebox(0,0)[rb]{\smash{0}}}%
    \put(0.12131164,0.20475001){\color[rgb]{0.14509804,0.14509804,0.14509804}\makebox(0,0)[rb]{\smash{0.2}}}%
    \put(0.12131164,0.32700001){\color[rgb]{0.14509804,0.14509804,0.14509804}\makebox(0,0)[rb]{\smash{0.4}}}%
    \put(0.12131164,0.44925001){\color[rgb]{0.14509804,0.14509804,0.14509804}\makebox(0,0)[rb]{\smash{0.6}}}%
    \put(0.12131164,0.5715){\color[rgb]{0.14509804,0.14509804,0.14509804}\makebox(0,0)[rb]{\smash{0.8}}}%
  \end{picture}%
\endgroup%

%% file: bistable_11.pdf_tex
\begingroup%
  \makeatletter%
  \providecommand\color[2][]{%
    \errmessage{(Inkscape) Color is used for the text in Inkscape, but the package 'color.sty' is not loaded}%
    \renewcommand\color[2][]{}%
  }%
  \providecommand\transparent[1]{%
    \errmessage{(Inkscape) Transparency is used (non-zero) for the text in Inkscape, but the package 'transparent.sty' is not loaded}%
    \renewcommand\transparent[1]{}%
  }%
  \providecommand\rotatebox[2]{#2}%
  \ifx\svgwidth\undefined%
    \setlength{\unitlength}{432bp}%
    \ifx\svgscale\undefined%
      \relax%
    \else%
      \setlength{\unitlength}{\unitlength * \real{\svgscale}}%
    \fi%
  \else%
    \setlength{\unitlength}{\svgwidth}%
  \fi%
  \global\let\svgwidth\undefined%
  \global\let\svgscale\undefined%
  \makeatother%
  \begin{picture}(1,0.75)%
    \put(0,0){\includegraphics[width=\unitlength,page=1]{bistable_11.pdf}}%
    \put(0.13,0.05381773){\color[rgb]{0.14901961,0.14901961,0.14901961}\makebox(0,0)[b]{\smash{0}}}%
    \put(0.31452431,0.05381773){\color[rgb]{0.14509804,0.14509804,0.14509804}\makebox(0,0)[b]{\smash{100}}}%
    \put(0.49904686,0.05381773){\color[rgb]{0.14509804,0.14509804,0.14509804}\makebox(0,0)[b]{\smash{200}}}%
    \put(0.68357118,0.05381773){\color[rgb]{0.14509804,0.14509804,0.14509804}\makebox(0,0)[b]{\smash{300}}}%
    \put(0.8680955,0.05381773){\color[rgb]{0.14509804,0.14509804,0.14509804}\makebox(0,0)[b]{\smash{400}}}%
    \put(0.12131164,0.08249998){\color[rgb]{0.14901961,0.14901961,0.14901961}\makebox(0,0)[rb]{\smash{0}}}%
    \put(0.12131164,0.20475001){\color[rgb]{0.14509804,0.14509804,0.14509804}\makebox(0,0)[rb]{\smash{0.2}}}%
    \put(0.12131164,0.32700001){\color[rgb]{0.14509804,0.14509804,0.14509804}\makebox(0,0)[rb]{\smash{0.4}}}%
    \put(0.12131164,0.44925001){\color[rgb]{0.14509804,0.14509804,0.14509804}\makebox(0,0)[rb]{\smash{0.6}}}%
    \put(0.12131164,0.5715){\color[rgb]{0.14509804,0.14509804,0.14509804}\makebox(0,0)[rb]{\smash{0.8}}}%
  \end{picture}%
\endgroup%

%% file: bistable_13.pdf_tex
\begingroup%
  \makeatletter%
  \providecommand\color[2][]{%
    \errmessage{(Inkscape) Color is used for the text in Inkscape, but the package 'color.sty' is not loaded}%
    \renewcommand\color[2][]{}%
  }%
  \providecommand\transparent[1]{%
    \errmessage{(Inkscape) Transparency is used (non-zero) for the text in Inkscape, but the package 'transparent.sty' is not loaded}%
    \renewcommand\transparent[1]{}%
  }%
  \providecommand\rotatebox[2]{#2}%
  \ifx\svgwidth\undefined%
    \setlength{\unitlength}{432bp}%
    \ifx\svgscale\undefined%
      \relax%
    \else%
      \setlength{\unitlength}{\unitlength * \real{\svgscale}}%
    \fi%
  \else%
    \setlength{\unitlength}{\svgwidth}%
  \fi%
  \global\let\svgwidth\undefined%
  \global\let\svgscale\undefined%
  \makeatother%
  \begin{picture}(1,0.75)%
    \put(0,0){\includegraphics[width=\unitlength,page=1]{bistable_13.pdf}}%
    \put(0.13,0.05381773){\color[rgb]{0.14901961,0.14901961,0.14901961}\makebox(0,0)[b]{\smash{0}}}%
    \put(0.31452431,0.05381773){\color[rgb]{0.14509804,0.14509804,0.14509804}\makebox(0,0)[b]{\smash{100}}}%
    \put(0.49904686,0.05381773){\color[rgb]{0.14509804,0.14509804,0.14509804}\makebox(0,0)[b]{\smash{200}}}%
    \put(0.68357118,0.05381773){\color[rgb]{0.14509804,0.14509804,0.14509804}\makebox(0,0)[b]{\smash{300}}}%
    \put(0.8680955,0.05381773){\color[rgb]{0.14509804,0.14509804,0.14509804}\makebox(0,0)[b]{\smash{400}}}%
    \put(0.12131164,0.08249998){\color[rgb]{0.14901961,0.14901961,0.14901961}\makebox(0,0)[rb]{\smash{0}}}%
    \put(0.12131164,0.20475001){\color[rgb]{0.14509804,0.14509804,0.14509804}\makebox(0,0)[rb]{\smash{0.2}}}%
    \put(0.12131164,0.32700001){\color[rgb]{0.14509804,0.14509804,0.14509804}\makebox(0,0)[rb]{\smash{0.4}}}%
    \put(0.12131164,0.44925001){\color[rgb]{0.14509804,0.14509804,0.14509804}\makebox(0,0)[rb]{\smash{0.6}}}%
    \put(0.12131164,0.5715){\color[rgb]{0.14509804,0.14509804,0.14509804}\makebox(0,0)[rb]{\smash{0.8}}}%
  \end{picture}%
\endgroup%

%% file: bistable_20.pdf_tex
\begingroup%
  \makeatletter%
  \providecommand\color[2][]{%
    \errmessage{(Inkscape) Color is used for the text in Inkscape, but the package 'color.sty' is not loaded}%
    \renewcommand\color[2][]{}%
  }%
  \providecommand\transparent[1]{%
    \errmessage{(Inkscape) Transparency is used (non-zero) for the text in Inkscape, but the package 'transparent.sty' is not loaded}%
    \renewcommand\transparent[1]{}%
  }%
  \providecommand\rotatebox[2]{#2}%
  \ifx\svgwidth\undefined%
    \setlength{\unitlength}{432bp}%
    \ifx\svgscale\undefined%
      \relax%
    \else%
      \setlength{\unitlength}{\unitlength * \real{\svgscale}}%
    \fi%
  \else%
    \setlength{\unitlength}{\svgwidth}%
  \fi%
  \global\let\svgwidth\undefined%
  \global\let\svgscale\undefined%
  \makeatother%
  \begin{picture}(1,0.75)%
    \put(0,0){\includegraphics[width=\unitlength,page=1]{bistable_20.pdf}}%
    \put(0.13,0.05381773){\color[rgb]{0.14901961,0.14901961,0.14901961}\makebox(0,0)[b]{\smash{0}}}%
    \put(0.31452431,0.05381773){\color[rgb]{0.14509804,0.14509804,0.14509804}\makebox(0,0)[b]{\smash{100}}}%
    \put(0.49904686,0.05381773){\color[rgb]{0.14509804,0.14509804,0.14509804}\makebox(0,0)[b]{\smash{200}}}%
    \put(0.68357118,0.05381773){\color[rgb]{0.14509804,0.14509804,0.14509804}\makebox(0,0)[b]{\smash{300}}}%
    \put(0.8680955,0.05381773){\color[rgb]{0.14509804,0.14509804,0.14509804}\makebox(0,0)[b]{\smash{400}}}%
    \put(0.12131164,0.08249998){\color[rgb]{0.14901961,0.14901961,0.14901961}\makebox(0,0)[rb]{\smash{0}}}%
    \put(0.12131164,0.20475001){\color[rgb]{0.14509804,0.14509804,0.14509804}\makebox(0,0)[rb]{\smash{0.2}}}%
    \put(0.12131164,0.32700001){\color[rgb]{0.14509804,0.14509804,0.14509804}\makebox(0,0)[rb]{\smash{0.4}}}%
    \put(0.12131164,0.44925001){\color[rgb]{0.14509804,0.14509804,0.14509804}\makebox(0,0)[rb]{\smash{0.6}}}%
    \put(0.12131164,0.5715){\color[rgb]{0.14509804,0.14509804,0.14509804}\makebox(0,0)[rb]{\smash{0.8}}}%
  \end{picture}%
\endgroup%

%% file: monostable_1_0.pdf_tex
\begingroup%
  \makeatletter%
  \providecommand\color[2][]{%
    \errmessage{(Inkscape) Color is used for the text in Inkscape, but the package 'color.sty' is not loaded}%
    \renewcommand\color[2][]{}%
  }%
  \providecommand\transparent[1]{%
    \errmessage{(Inkscape) Transparency is used (non-zero) for the text in Inkscape, but the package 'transparent.sty' is not loaded}%
    \renewcommand\transparent[1]{}%
  }%
  \providecommand\rotatebox[2]{#2}%
  \ifx\svgwidth\undefined%
    \setlength{\unitlength}{432bp}%
    \ifx\svgscale\undefined%
      \relax%
    \else%
      \setlength{\unitlength}{\unitlength * \real{\svgscale}}%
    \fi%
  \else%
    \setlength{\unitlength}{\svgwidth}%
  \fi%
  \global\let\svgwidth\undefined%
  \global\let\svgscale\undefined%
  \makeatother%
  \begin{picture}(1,0.75)%
    \put(0,0){\includegraphics[width=\unitlength,page=1]{monostable_1_0.pdf}}%
    \put(0.13,0.05381773){\color[rgb]{0.14901961,0.14901961,0.14901961}\makebox(0,0)[b]{\smash{0}}}%
    \put(0.26839235,0.05381773){\color[rgb]{0.14509804,0.14509804,0.14509804}\makebox(0,0)[b]{\smash{50}}}%
    \put(0.40678647,0.05381773){\color[rgb]{0.14509804,0.14509804,0.14509804}\makebox(0,0)[b]{\smash{100}}}%
    \put(0.54517884,0.05381773){\color[rgb]{0.14509804,0.14509804,0.14509804}\makebox(0,0)[b]{\smash{150}}}%
    \put(0.68357118,0.05381773){\color[rgb]{0.14509804,0.14509804,0.14509804}\makebox(0,0)[b]{\smash{200}}}%
    \put(0.82196352,0.05381773){\color[rgb]{0.14509804,0.14509804,0.14509804}\makebox(0,0)[b]{\smash{250}}}%
    \put(0.12131164,0.08249998){\color[rgb]{0.14901961,0.14901961,0.14901961}\makebox(0,0)[rb]{\smash{0}}}%
    \put(0.12131164,0.20475001){\color[rgb]{0.14509804,0.14509804,0.14509804}\makebox(0,0)[rb]{\smash{0.2}}}%
    \put(0.12131164,0.32700001){\color[rgb]{0.14509804,0.14509804,0.14509804}\makebox(0,0)[rb]{\smash{0.4}}}%
    \put(0.12131164,0.44925001){\color[rgb]{0.14509804,0.14509804,0.14509804}\makebox(0,0)[rb]{\smash{0.6}}}%
    \put(0.12131164,0.5715){\color[rgb]{0.14509804,0.14509804,0.14509804}\makebox(0,0)[rb]{\smash{0.8}}}%
  \end{picture}%
\endgroup%

%% file: monostable_1_1.pdf_tex
\begingroup%
  \makeatletter%
  \providecommand\color[2][]{%
    \errmessage{(Inkscape) Color is used for the text in Inkscape, but the package 'color.sty' is not loaded}%
    \renewcommand\color[2][]{}%
  }%
  \providecommand\transparent[1]{%
    \errmessage{(Inkscape) Transparency is used (non-zero) for the text in Inkscape, but the package 'transparent.sty' is not loaded}%
    \renewcommand\transparent[1]{}%
  }%
  \providecommand\rotatebox[2]{#2}%
  \ifx\svgwidth\undefined%
    \setlength{\unitlength}{432bp}%
    \ifx\svgscale\undefined%
      \relax%
    \else%
      \setlength{\unitlength}{\unitlength * \real{\svgscale}}%
    \fi%
  \else%
    \setlength{\unitlength}{\svgwidth}%
  \fi%
  \global\let\svgwidth\undefined%
  \global\let\svgscale\undefined%
  \makeatother%
  \begin{picture}(1,0.75)%
    \put(0,0){\includegraphics[width=\unitlength,page=1]{monostable_1_1.pdf}}%
    \put(0.13,0.05381773){\color[rgb]{0.14901961,0.14901961,0.14901961}\makebox(0,0)[b]{\smash{0}}}%
    \put(0.26839235,0.05381773){\color[rgb]{0.14509804,0.14509804,0.14509804}\makebox(0,0)[b]{\smash{50}}}%
    \put(0.40678647,0.05381773){\color[rgb]{0.14509804,0.14509804,0.14509804}\makebox(0,0)[b]{\smash{100}}}%
    \put(0.54517884,0.05381773){\color[rgb]{0.14509804,0.14509804,0.14509804}\makebox(0,0)[b]{\smash{150}}}%
    \put(0.68357118,0.05381773){\color[rgb]{0.14509804,0.14509804,0.14509804}\makebox(0,0)[b]{\smash{200}}}%
    \put(0.82196352,0.05381773){\color[rgb]{0.14509804,0.14509804,0.14509804}\makebox(0,0)[b]{\smash{250}}}%
    \put(0.12131164,0.08249998){\color[rgb]{0.14901961,0.14901961,0.14901961}\makebox(0,0)[rb]{\smash{0}}}%
    \put(0.12131164,0.20475001){\color[rgb]{0.14509804,0.14509804,0.14509804}\makebox(0,0)[rb]{\smash{0.2}}}%
    \put(0.12131164,0.32700001){\color[rgb]{0.14509804,0.14509804,0.14509804}\makebox(0,0)[rb]{\smash{0.4}}}%
    \put(0.12131164,0.44925001){\color[rgb]{0.14509804,0.14509804,0.14509804}\makebox(0,0)[rb]{\smash{0.6}}}%
    \put(0.12131164,0.5715){\color[rgb]{0.14509804,0.14509804,0.14509804}\makebox(0,0)[rb]{\smash{0.8}}}%
  \end{picture}%
\endgroup%

%% file: monostable_1_8.pdf_tex
\begingroup%
  \makeatletter%
  \providecommand\color[2][]{%
    \errmessage{(Inkscape) Color is used for the text in Inkscape, but the package 'color.sty' is not loaded}%
    \renewcommand\color[2][]{}%
  }%
  \providecommand\transparent[1]{%
    \errmessage{(Inkscape) Transparency is used (non-zero) for the text in Inkscape, but the package 'transparent.sty' is not loaded}%
    \renewcommand\transparent[1]{}%
  }%
  \providecommand\rotatebox[2]{#2}%
  \ifx\svgwidth\undefined%
    \setlength{\unitlength}{432bp}%
    \ifx\svgscale\undefined%
      \relax%
    \else%
      \setlength{\unitlength}{\unitlength * \real{\svgscale}}%
    \fi%
  \else%
    \setlength{\unitlength}{\svgwidth}%
  \fi%
  \global\let\svgwidth\undefined%
  \global\let\svgscale\undefined%
  \makeatother%
  \begin{picture}(1,0.75)%
    \put(0,0){\includegraphics[width=\unitlength,page=1]{monostable_1_8.pdf}}%
    \put(0.13,0.05381773){\color[rgb]{0.14901961,0.14901961,0.14901961}\makebox(0,0)[b]{\smash{0}}}%
    \put(0.26839235,0.05381773){\color[rgb]{0.14509804,0.14509804,0.14509804}\makebox(0,0)[b]{\smash{50}}}%
    \put(0.40678647,0.05381773){\color[rgb]{0.14509804,0.14509804,0.14509804}\makebox(0,0)[b]{\smash{100}}}%
    \put(0.54517884,0.05381773){\color[rgb]{0.14509804,0.14509804,0.14509804}\makebox(0,0)[b]{\smash{150}}}%
    \put(0.68357118,0.05381773){\color[rgb]{0.14509804,0.14509804,0.14509804}\makebox(0,0)[b]{\smash{200}}}%
    \put(0.82196352,0.05381773){\color[rgb]{0.14509804,0.14509804,0.14509804}\makebox(0,0)[b]{\smash{250}}}%
    \put(0.12131164,0.08249998){\color[rgb]{0.14901961,0.14901961,0.14901961}\makebox(0,0)[rb]{\smash{0}}}%
    \put(0.12131164,0.20475001){\color[rgb]{0.14509804,0.14509804,0.14509804}\makebox(0,0)[rb]{\smash{0.2}}}%
    \put(0.12131164,0.32700001){\color[rgb]{0.14509804,0.14509804,0.14509804}\makebox(0,0)[rb]{\smash{0.4}}}%
    \put(0.12131164,0.44925001){\color[rgb]{0.14509804,0.14509804,0.14509804}\makebox(0,0)[rb]{\smash{0.6}}}%
    \put(0.12131164,0.5715){\color[rgb]{0.14509804,0.14509804,0.14509804}\makebox(0,0)[rb]{\smash{0.8}}}%
  \end{picture}%
\endgroup%

%% file: monostable_6_0.pdf_tex
\begingroup%
  \makeatletter%
  \providecommand\color[2][]{%
    \errmessage{(Inkscape) Color is used for the text in Inkscape, but the package 'color.sty' is not loaded}%
    \renewcommand\color[2][]{}%
  }%
  \providecommand\transparent[1]{%
    \errmessage{(Inkscape) Transparency is used (non-zero) for the text in Inkscape, but the package 'transparent.sty' is not loaded}%
    \renewcommand\transparent[1]{}%
  }%
  \providecommand\rotatebox[2]{#2}%
  \ifx\svgwidth\undefined%
    \setlength{\unitlength}{432bp}%
    \ifx\svgscale\undefined%
      \relax%
    \else%
      \setlength{\unitlength}{\unitlength * \real{\svgscale}}%
    \fi%
  \else%
    \setlength{\unitlength}{\svgwidth}%
  \fi%
  \global\let\svgwidth\undefined%
  \global\let\svgscale\undefined%
  \makeatother%
  \begin{picture}(1,0.75)%
    \put(0,0){\includegraphics[width=\unitlength,page=1]{monostable_6_0.pdf}}%
    \put(0.13,0.05381773){\color[rgb]{0.14901961,0.14901961,0.14901961}\makebox(0,0)[b]{\smash{0}}}%
    \put(0.26839235,0.05381773){\color[rgb]{0.14509804,0.14509804,0.14509804}\makebox(0,0)[b]{\smash{50}}}%
    \put(0.40678647,0.05381773){\color[rgb]{0.14509804,0.14509804,0.14509804}\makebox(0,0)[b]{\smash{100}}}%
    \put(0.54517884,0.05381773){\color[rgb]{0.14509804,0.14509804,0.14509804}\makebox(0,0)[b]{\smash{150}}}%
    \put(0.68357118,0.05381773){\color[rgb]{0.14509804,0.14509804,0.14509804}\makebox(0,0)[b]{\smash{200}}}%
    \put(0.82196352,0.05381773){\color[rgb]{0.14509804,0.14509804,0.14509804}\makebox(0,0)[b]{\smash{250}}}%
    \put(0.12131164,0.08249998){\color[rgb]{0.14901961,0.14901961,0.14901961}\makebox(0,0)[rb]{\smash{0}}}%
    \put(0.12131164,0.20475001){\color[rgb]{0.14509804,0.14509804,0.14509804}\makebox(0,0)[rb]{\smash{0.2}}}%
    \put(0.12131164,0.32700001){\color[rgb]{0.14509804,0.14509804,0.14509804}\makebox(0,0)[rb]{\smash{0.4}}}%
    \put(0.12131164,0.44925001){\color[rgb]{0.14509804,0.14509804,0.14509804}\makebox(0,0)[rb]{\smash{0.6}}}%
    \put(0.12131164,0.5715){\color[rgb]{0.14509804,0.14509804,0.14509804}\makebox(0,0)[rb]{\smash{0.8}}}%
  \end{picture}%
\endgroup%

%% file: monostable_6_1.pdf_tex
\begingroup%
  \makeatletter%
  \providecommand\color[2][]{%
    \errmessage{(Inkscape) Color is used for the text in Inkscape, but the package 'color.sty' is not loaded}%
    \renewcommand\color[2][]{}%
  }%
  \providecommand\transparent[1]{%
    \errmessage{(Inkscape) Transparency is used (non-zero) for the text in Inkscape, but the package 'transparent.sty' is not loaded}%
    \renewcommand\transparent[1]{}%
  }%
  \providecommand\rotatebox[2]{#2}%
  \ifx\svgwidth\undefined%
    \setlength{\unitlength}{432bp}%
    \ifx\svgscale\undefined%
      \relax%
    \else%
      \setlength{\unitlength}{\unitlength * \real{\svgscale}}%
    \fi%
  \else%
    \setlength{\unitlength}{\svgwidth}%
  \fi%
  \global\let\svgwidth\undefined%
  \global\let\svgscale\undefined%
  \makeatother%
  \begin{picture}(1,0.75)%
    \put(0,0){\includegraphics[width=\unitlength,page=1]{monostable_6_1.pdf}}%
    \put(0.13,0.05381773){\color[rgb]{0.14901961,0.14901961,0.14901961}\makebox(0,0)[b]{\smash{0}}}%
    \put(0.26839235,0.05381773){\color[rgb]{0.14509804,0.14509804,0.14509804}\makebox(0,0)[b]{\smash{50}}}%
    \put(0.40678647,0.05381773){\color[rgb]{0.14509804,0.14509804,0.14509804}\makebox(0,0)[b]{\smash{100}}}%
    \put(0.54517884,0.05381773){\color[rgb]{0.14509804,0.14509804,0.14509804}\makebox(0,0)[b]{\smash{150}}}%
    \put(0.68357118,0.05381773){\color[rgb]{0.14509804,0.14509804,0.14509804}\makebox(0,0)[b]{\smash{200}}}%
    \put(0.82196352,0.05381773){\color[rgb]{0.14509804,0.14509804,0.14509804}\makebox(0,0)[b]{\smash{250}}}%
    \put(0.12131164,0.08249998){\color[rgb]{0.14901961,0.14901961,0.14901961}\makebox(0,0)[rb]{\smash{0}}}%
    \put(0.12131164,0.20475001){\color[rgb]{0.14509804,0.14509804,0.14509804}\makebox(0,0)[rb]{\smash{0.2}}}%
    \put(0.12131164,0.32700001){\color[rgb]{0.14509804,0.14509804,0.14509804}\makebox(0,0)[rb]{\smash{0.4}}}%
    \put(0.12131164,0.44925001){\color[rgb]{0.14509804,0.14509804,0.14509804}\makebox(0,0)[rb]{\smash{0.6}}}%
    \put(0.12131164,0.5715){\color[rgb]{0.14509804,0.14509804,0.14509804}\makebox(0,0)[rb]{\smash{0.8}}}%
  \end{picture}%
\endgroup%

%% file: monostable_6_8.pdf_tex
\begingroup%
  \makeatletter%
  \providecommand\color[2][]{%
    \errmessage{(Inkscape) Color is used for the text in Inkscape, but the package 'color.sty' is not loaded}%
    \renewcommand\color[2][]{}%
  }%
  \providecommand\transparent[1]{%
    \errmessage{(Inkscape) Transparency is used (non-zero) for the text in Inkscape, but the package 'transparent.sty' is not loaded}%
    \renewcommand\transparent[1]{}%
  }%
  \providecommand\rotatebox[2]{#2}%
  \ifx\svgwidth\undefined%
    \setlength{\unitlength}{432bp}%
    \ifx\svgscale\undefined%
      \relax%
    \else%
      \setlength{\unitlength}{\unitlength * \real{\svgscale}}%
    \fi%
  \else%
    \setlength{\unitlength}{\svgwidth}%
  \fi%
  \global\let\svgwidth\undefined%
  \global\let\svgscale\undefined%
  \makeatother%
  \begin{picture}(1,0.75)%
    \put(0,0){\includegraphics[width=\unitlength,page=1]{monostable_6_8.pdf}}%
    \put(0.13,0.05381773){\color[rgb]{0.14901961,0.14901961,0.14901961}\makebox(0,0)[b]{\smash{0}}}%
    \put(0.26839235,0.05381773){\color[rgb]{0.14509804,0.14509804,0.14509804}\makebox(0,0)[b]{\smash{50}}}%
    \put(0.40678647,0.05381773){\color[rgb]{0.14509804,0.14509804,0.14509804}\makebox(0,0)[b]{\smash{100}}}%
    \put(0.54517884,0.05381773){\color[rgb]{0.14509804,0.14509804,0.14509804}\makebox(0,0)[b]{\smash{150}}}%
    \put(0.68357118,0.05381773){\color[rgb]{0.14509804,0.14509804,0.14509804}\makebox(0,0)[b]{\smash{200}}}%
    \put(0.82196352,0.05381773){\color[rgb]{0.14509804,0.14509804,0.14509804}\makebox(0,0)[b]{\smash{250}}}%
    \put(0.12131164,0.08249998){\color[rgb]{0.14901961,0.14901961,0.14901961}\makebox(0,0)[rb]{\smash{0}}}%
    \put(0.12131164,0.20475001){\color[rgb]{0.14509804,0.14509804,0.14509804}\makebox(0,0)[rb]{\smash{0.2}}}%
    \put(0.12131164,0.32700001){\color[rgb]{0.14509804,0.14509804,0.14509804}\makebox(0,0)[rb]{\smash{0.4}}}%
    \put(0.12131164,0.44925001){\color[rgb]{0.14509804,0.14509804,0.14509804}\makebox(0,0)[rb]{\smash{0.6}}}%
    \put(0.12131164,0.5715){\color[rgb]{0.14509804,0.14509804,0.14509804}\makebox(0,0)[rb]{\smash{0.8}}}%
  \end{picture}%
\endgroup%

%% file: monostable_11_0.pdf_tex
\begingroup%
  \makeatletter%
  \providecommand\color[2][]{%
    \errmessage{(Inkscape) Color is used for the text in Inkscape, but the package 'color.sty' is not loaded}%
    \renewcommand\color[2][]{}%
  }%
  \providecommand\transparent[1]{%
    \errmessage{(Inkscape) Transparency is used (non-zero) for the text in Inkscape, but the package 'transparent.sty' is not loaded}%
    \renewcommand\transparent[1]{}%
  }%
  \providecommand\rotatebox[2]{#2}%
  \ifx\svgwidth\undefined%
    \setlength{\unitlength}{432bp}%
    \ifx\svgscale\undefined%
      \relax%
    \else%
      \setlength{\unitlength}{\unitlength * \real{\svgscale}}%
    \fi%
  \else%
    \setlength{\unitlength}{\svgwidth}%
  \fi%
  \global\let\svgwidth\undefined%
  \global\let\svgscale\undefined%
  \makeatother%
  \begin{picture}(1,0.75)%
    \put(0,0){\includegraphics[width=\unitlength,page=1]{monostable_11_0.pdf}}%
    \put(0.13,0.05381773){\color[rgb]{0.14901961,0.14901961,0.14901961}\makebox(0,0)[b]{\smash{0}}}%
    \put(0.26839235,0.05381773){\color[rgb]{0.14509804,0.14509804,0.14509804}\makebox(0,0)[b]{\smash{50}}}%
    \put(0.40678647,0.05381773){\color[rgb]{0.14509804,0.14509804,0.14509804}\makebox(0,0)[b]{\smash{100}}}%
    \put(0.54517884,0.05381773){\color[rgb]{0.14509804,0.14509804,0.14509804}\makebox(0,0)[b]{\smash{150}}}%
    \put(0.68357118,0.05381773){\color[rgb]{0.14509804,0.14509804,0.14509804}\makebox(0,0)[b]{\smash{200}}}%
    \put(0.82196352,0.05381773){\color[rgb]{0.14509804,0.14509804,0.14509804}\makebox(0,0)[b]{\smash{250}}}%
    \put(0.12131164,0.08249998){\color[rgb]{0.14901961,0.14901961,0.14901961}\makebox(0,0)[rb]{\smash{0}}}%
    \put(0.12131164,0.20475001){\color[rgb]{0.14509804,0.14509804,0.14509804}\makebox(0,0)[rb]{\smash{0.2}}}%
    \put(0.12131164,0.32700001){\color[rgb]{0.14509804,0.14509804,0.14509804}\makebox(0,0)[rb]{\smash{0.4}}}%
    \put(0.12131164,0.44925001){\color[rgb]{0.14509804,0.14509804,0.14509804}\makebox(0,0)[rb]{\smash{0.6}}}%
    \put(0.12131164,0.5715){\color[rgb]{0.14509804,0.14509804,0.14509804}\makebox(0,0)[rb]{\smash{0.8}}}%
  \end{picture}%
\endgroup%

%% file: monostable_11_1.pdf_tex
\begingroup%
  \makeatletter%
  \providecommand\color[2][]{%
    \errmessage{(Inkscape) Color is used for the text in Inkscape, but the package 'color.sty' is not loaded}%
    \renewcommand\color[2][]{}%
  }%
  \providecommand\transparent[1]{%
    \errmessage{(Inkscape) Transparency is used (non-zero) for the text in Inkscape, but the package 'transparent.sty' is not loaded}%
    \renewcommand\transparent[1]{}%
  }%
  \providecommand\rotatebox[2]{#2}%
  \ifx\svgwidth\undefined%
    \setlength{\unitlength}{432bp}%
    \ifx\svgscale\undefined%
      \relax%
    \else%
      \setlength{\unitlength}{\unitlength * \real{\svgscale}}%
    \fi%
  \else%
    \setlength{\unitlength}{\svgwidth}%
  \fi%
  \global\let\svgwidth\undefined%
  \global\let\svgscale\undefined%
  \makeatother%
  \begin{picture}(1,0.75)%
    \put(0,0){\includegraphics[width=\unitlength,page=1]{monostable_11_1.pdf}}%
    \put(0.13,0.05381773){\color[rgb]{0.14901961,0.14901961,0.14901961}\makebox(0,0)[b]{\smash{0}}}%
    \put(0.26839235,0.05381773){\color[rgb]{0.14509804,0.14509804,0.14509804}\makebox(0,0)[b]{\smash{50}}}%
    \put(0.40678647,0.05381773){\color[rgb]{0.14509804,0.14509804,0.14509804}\makebox(0,0)[b]{\smash{100}}}%
    \put(0.54517884,0.05381773){\color[rgb]{0.14509804,0.14509804,0.14509804}\makebox(0,0)[b]{\smash{150}}}%
    \put(0.68357118,0.05381773){\color[rgb]{0.14509804,0.14509804,0.14509804}\makebox(0,0)[b]{\smash{200}}}%
    \put(0.82196352,0.05381773){\color[rgb]{0.14509804,0.14509804,0.14509804}\makebox(0,0)[b]{\smash{250}}}%
    \put(0.12131164,0.08249998){\color[rgb]{0.14901961,0.14901961,0.14901961}\makebox(0,0)[rb]{\smash{0}}}%
    \put(0.12131164,0.20475001){\color[rgb]{0.14509804,0.14509804,0.14509804}\makebox(0,0)[rb]{\smash{0.2}}}%
    \put(0.12131164,0.32700001){\color[rgb]{0.14509804,0.14509804,0.14509804}\makebox(0,0)[rb]{\smash{0.4}}}%
    \put(0.12131164,0.44925001){\color[rgb]{0.14509804,0.14509804,0.14509804}\makebox(0,0)[rb]{\smash{0.6}}}%
    \put(0.12131164,0.5715){\color[rgb]{0.14509804,0.14509804,0.14509804}\makebox(0,0)[rb]{\smash{0.8}}}%
  \end{picture}%
\endgroup%

%% file: monostable_11_8.pdf_tex
\begingroup%
  \makeatletter%
  \providecommand\color[2][]{%
    \errmessage{(Inkscape) Color is used for the text in Inkscape, but the package 'color.sty' is not loaded}%
    \renewcommand\color[2][]{}%
  }%
  \providecommand\transparent[1]{%
    \errmessage{(Inkscape) Transparency is used (non-zero) for the text in Inkscape, but the package 'transparent.sty' is not loaded}%
    \renewcommand\transparent[1]{}%
  }%
  \providecommand\rotatebox[2]{#2}%
  \ifx\svgwidth\undefined%
    \setlength{\unitlength}{432bp}%
    \ifx\svgscale\undefined%
      \relax%
    \else%
      \setlength{\unitlength}{\unitlength * \real{\svgscale}}%
    \fi%
  \else%
    \setlength{\unitlength}{\svgwidth}%
  \fi%
  \global\let\svgwidth\undefined%
  \global\let\svgscale\undefined%
  \makeatother%
  \begin{picture}(1,0.75)%
    \put(0,0){\includegraphics[width=\unitlength,page=1]{monostable_11_8.pdf}}%
    \put(0.13,0.05381773){\color[rgb]{0.14901961,0.14901961,0.14901961}\makebox(0,0)[b]{\smash{0}}}%
    \put(0.26839235,0.05381773){\color[rgb]{0.14509804,0.14509804,0.14509804}\makebox(0,0)[b]{\smash{50}}}%
    \put(0.40678647,0.05381773){\color[rgb]{0.14509804,0.14509804,0.14509804}\makebox(0,0)[b]{\smash{100}}}%
    \put(0.54517884,0.05381773){\color[rgb]{0.14509804,0.14509804,0.14509804}\makebox(0,0)[b]{\smash{150}}}%
    \put(0.68357118,0.05381773){\color[rgb]{0.14509804,0.14509804,0.14509804}\makebox(0,0)[b]{\smash{200}}}%
    \put(0.82196352,0.05381773){\color[rgb]{0.14509804,0.14509804,0.14509804}\makebox(0,0)[b]{\smash{250}}}%
    \put(0.12131164,0.08249998){\color[rgb]{0.14901961,0.14901961,0.14901961}\makebox(0,0)[rb]{\smash{0}}}%
    \put(0.12131164,0.20475001){\color[rgb]{0.14509804,0.14509804,0.14509804}\makebox(0,0)[rb]{\smash{0.2}}}%
    \put(0.12131164,0.32700001){\color[rgb]{0.14509804,0.14509804,0.14509804}\makebox(0,0)[rb]{\smash{0.4}}}%
    \put(0.12131164,0.44925001){\color[rgb]{0.14509804,0.14509804,0.14509804}\makebox(0,0)[rb]{\smash{0.6}}}%
    \put(0.12131164,0.5715){\color[rgb]{0.14509804,0.14509804,0.14509804}\makebox(0,0)[rb]{\smash{0.8}}}%
  \end{picture}%
\endgroup%